\documentclass[twoside]{article}
\usepackage{latexsym}
\usepackage{amsmath,amsthm,amssymb}
\usepackage{epsfig}
\usepackage{stmaryrd}
\usepackage{multicol}
\usepackage{dsfont}
\usepackage{comment}
\usepackage{algorithmic}
\usepackage{algorithm}
\usepackage{natbib}
\usepackage{caption}
\usepackage{psfrag}
\usepackage[final=true,colorlinks=true,breaklinks=true,dvips=true,linkcolor=blue,citecolor=blue]{hyperref}
\usepackage{graphics}

\usepackage[margin=1in]{geometry}

\newtheorem{problem}{\sc Problem}[section]
\newtheorem{theorem}{\sc Theorem}[section]
\newtheorem{lemma}{\sc Lemma}[section]
\newtheorem{definition}{\sc Definition}[section]
\newtheorem{corollary}{\sc Corollary}[section]
\newtheorem{proposition}{\sc Proposition}[section]
\newtheorem{assumption}{\sc Assumption}

\newcommand{\mb}[1]{\mbox{\boldmath $#1$}}

\def\a{\alpha}
\def\b{\beta}
\def\c{\gamma}

\def\tpi{\tilde{\pi}}
\def\hpi{\hat{\pi}}

\def\argmax{\operatornamewithlimits{arg\,max}}
\def\argmin{\operatornamewithlimits{arg\,min}}
\def\conv{\operatornamewithlimits{conv}}
\def\cone{\operatornamewithlimits{cone}}
\def\rank{\operatornamewithlimits{rank}}
\def\ext{\operatornamewithlimits{ext}}
\def\zhull{\operatornamewithlimits{z-hull}}
\def\rside{\operatornamewithlimits{r-side}}

\def\bydef{\mathinner{\,\stackrel{\text{\tiny def}}{=}\,}}

\def\R{\mathbf{R}}

\def\W{\mathcal{W}}
\def\H{\mathcal{H}}
\def\B{\mathcal{B}_{LU}}

\def\V{\mathcal{V}}

\def\abs#1{\mathinner{\left|{#1}\right|}}   
\newcommand{\scprod}[2]{\mathinner{\mb{#1}^T\mb{#2}}}
\def\argmax{\operatornamewithlimits{arg\,max}}
\def\opt1DGauss#1#2#3{\frac{1}{\sqrt{2\pi}#3}\, e^{-\frac{(#1-#2)^2}{2#3^2}}}      
\def\1DGauss#1#2#3{\frac{e^{-\frac{(#1-#2)^2}{2#3^2}}}{\sqrt{2\pi}#3}}      

\newcommand{\vlow}[2]{\mathinner{\underline{#1}_{#2}}}
\newcommand{\vup}[2]{\mathinner{\overline{#1}_{#2}}}

\newcommand{\aff}[3]{\mathinner{{#1}_{#2,#3}}}
\newcommand{\minofgt}[1]{\mathinner{[\vlow{y}{#1},\vup{y}{#1}]}}
\newcommand{\expl}[2]{\mathinner{\underset{#1}{\underbrace{#2}}}}

\newcommand{\cotan}[2]{\mathinner{\textup{cotan}\left({#1},\,{#2}\right)}}

\begin{document}

\title{Optimality of Affine Policies in Multi-stage Robust Optimization}
\author{Dimitris Bertsimas \thanks{Sloan School of Management and Operations Research Center, Massachusetts 
Institute of Technology, 77 Massachusetts Avenue, E40-147, Cambridge, MA 02139, USA. Email: \texttt{dbertsim@mit.edu}.} \and Dan A. Iancu\thanks{Operations Research Center, Massachusetts Institute of Technology, 
77 Massachusetts Avenue, E40-130, Cambridge, MA 02139, USA. Email: \texttt{daniancu@mit.edu}.} \and Pablo A. Parrilo \thanks{Laboratory for Information and Decision Systems, Massachusetts Institute of Technology,
77 Massachusetts Avenue, 32D-726, Cambridge, MA 02139, USA. Email: \texttt{parrilo@mit.edu}.}} 
\date{}

\maketitle

\begin{abstract}
In this paper, we show the optimality of a certain class of disturbance-affine control policies 
in the context of one-dimensional, constrained, multi-stage robust optimization. 
Our results cover the finite horizon case, with minimax (worst-case) objective,
and convex state costs plus linear control costs. 
We develop a new 
proof methodology, which explores the relationship between the geometrical 
properties of the feasible set of solutions and the structure of the objective 
function. Apart from providing an elegant and conceptually simple proof 
technique, the approach also entails very fast algorithms for the case 
of piecewise affine state costs, which we explore in connection with a classical 
inventory management application.
\end{abstract}

\section{Introduction.}
\label{sec:introduction}
Multi-stage optimization problems under uncertainty have been prevalent in 
numerous fields of science and engineering, and have elicited interest 
from diverse research communities, on both a theoretical and a practical 
level. Several solution approaches have been proposed, with various degrees 
of generality, tractability, and performance guarantees. Some of the 
most successful ones include exact and approximate dynamic programming, 
stochastic programming, sampling-based methods, and, more recently, 
robust and adaptive optimization, which is the focus of the present paper.

The topics of robust optimization and robust control have been studied, 
under different names, by a variety of academic groups, mostly in operations research 
(\cite{BenTal00}, \cite{BenTalRoos02}, \cite{BenTal02}, \cite{BerSim03}, 
\cite{BerSim04}, \cite{BertPach04})
and control theory (\cite{Bkas71}, \cite{Doyle91}, \cite{ElGhao98}, 
\cite{PabloMor03}, \cite{BempMor03}, \cite{KerrMacie04}, \citet{Zhou98}, 
\citet{Paga05}),
with considerable effort put into justifying the assumptions and general 
modeling philosophy. As such, the goal of the current paper 
will not be to \emph{justify} the use of robust (and, more generally, 
distribution-free) techniques. Rather, we will take the modeling approach
as a given, and investigate tractability and performance issues in the context of 
a certain class of optimization problems. 
More precisely, we will be concerned with the following multi-stage 
decision problem: 
\begin{problem}
  \label{prob:initial_problem}
  Consider the following one-dimensional, discrete, linear, time-varying dynamical system:
  \begin{equation}
    \label{eq:dynam_sys_evolution}
    x_{k+1} = \a_k \cdot x_k + \b_k \cdot u_k + \c_k \cdot w_k
  \end{equation}
  where $\a_k, \b_k, \c_k \neq 0$ are known scalars, 
  and the initial state $x_1 \in \R$ is specified. 
  The random disturbances $w_k$ are unknown, but bounded:
  \begin{equation}
    \label{eq:uncertainty_set}
    w_k \in \W_k \bydef [\vlow{w}{k},\vup{w}{k}].
  \end{equation}
  We would like to find a sequence of robust controllers $\{u_k\}$, 
  obeying certain constraints:
  \begin{align}
    u_k \in \left[ L_k, U_k \right]
    \label{eq:control_constraint}
  \end{align}
  ($L_k,U_k$ are known and fixed), and optimizing the 
  following min-max cost function over a finite horizon $1,\dots,T$:
  \begin{align}
    J_{mM} &\bydef~ \underset{u_1}{\min} \left[ c_1 \cdot u_1 + 
      \underset{w_1}{\max} \left[ 
        h_1(x_2) + \dots + \underset{u_k}{\min} \left[ c_k \cdot u_k + 
          \underset{w_k}{\max} \left[ h_k(x_{k+1}) + 
            \dots \right. \right. \right. \right. \nonumber \\
      & \qquad \qquad \qquad \qquad  \left. \left. \left.
          + \underset{u_{T}}{\min} \left( c_T \cdot u_T + \underset{w_{T}}{\max}
            ~ h_T(x_{T+1}) ~\right) \dots \right] \dots \right] ~\right]
      \label{eq:cost_function}
  \end{align}
  where the functions $h_k : \R \rightarrow \R$ are known, convex and coercive, 
  and $c_k \geq 0$ are fixed and known.
\end{problem}

The problem corresponds to a situation in which, at every time step $k$,
the decision maker has to compute a control action $u_k$, in such a way that certain 
constraints \eqref{eq:control_constraint} are obeyed, and a cost penalizing 
both the state ($h_k(x_{k+1})$) and the control ($c_k \cdot u_k$) is minimized.
The uncertainty, $w_k$, always acts so as to maximize the costs,
hence the problem solved by the decision maker corresponds to a worst-case scenario
(a minimization of the maximum possible cost). 
Examples of such problems include the case of quadratic state costs, 
$h_k(x_{k+1}) = r_{k+1} \cdot x_{k+1}^2, (r_{k+1} \geq 0)$, as well as 
norm-1 or norm-$\infty$ costs, $h_k(x_{k+1}) = r_{k+1} \cdot \abs{x_{k+1}},
(r_{k+1} \geq 0)$, all of which have been studied extensively in the literature 
in the unconstrained case (see, for example, \citet{Zhou98}, and \citet{Paga05}).

The solution to Problem \ref{prob:initial_problem} could be obtained using a ``classical'' 
Dynamic Programming formulation (see \citet{Bkas05}), in which the optimal policies 
$u_k^*(x_k)$ and the optimal value functions 
$J_k^*(x_k)$ are computed backwards in time, starting at the 
end of the planning horizon, $k=T$. The resulting policies 
are piecewise affine in the state $x_t$, and have properties that are 
well known and documented in the literature, dating back to \cite{Scarf58}. 

In the current paper, we would like to study the performance of a new class of 
policies, where instead of regarding the controllers $u_k$ as functions of the 
state $x_k$, one seeks direct parameterizations in the observed disturbances:
\begin{equation}
  \label{eq:param_in_disturbance}
  u_k : \W_1 \times \W_2 \times \dots \times \W_{k-1} \rightarrow \R.
\end{equation}
In this framework, we would require that constraint \eqref{eq:control_constraint} 
should be robustly feasible:
\begin{align}
  u_k(\mb{w}^{k}) \in \left[ L_k, U_k \right], \qquad
  \forall \, \mb{w}^{k} \bydef \left(w_1,\dots,w_{k-1}\right)
  \in \W_1 \times \dots \times \W_{k-1}.
  \label{eq:robust_control_constraint}
\end{align}
Note that if we insisted on this category of parameterizations, then we would have to consider 
a new state for the system, $\mb{X}_k$, which would include all the past-observed disturbances, 
$\mb{w}^k$. Furthermore, while $x_k$ summarizes all the past 
information needed to make an optimal decision at stage $k$, the same would not necessarily 
be true for $\mb{w}^k$, so we may want to include even more information in $\mb{X}_k$ (for example, 
the previous controls $\{u_t\}_{1 \leq t < k}$ or the previous states
$\{x_t\}_{1 \leq t < k}$ or some combination thereof).
Compared with the original, compact state formulation, $x_k$, the new state $\mb{X}_k$ 
would become much larger, and solving the 
Dynamic Program with state variable $\mb{X}_k$ would produce exactly 
the same optimal objective function value $J_{mM}$. 
Therefore, one would be tempted to ask what the benefit for 
introducing such a complicated state might be.

The hope is that, by considering policies over a larger state, one might be 
able to obtain simpler functional forms, for instance, affine policies. These have 
a very compact representation, since only the coefficients of 
the parameterization are needed, and, for certain classes of convex costs 
$h_k(\cdot)$, there may be efficient procedures available for computing them.

This approach is also not new in the literature. It has been originally 
advocated in the context of stochastic programming (\citet{RockWets78}), then in 
robust optimization (\citet{BenTal04}), and extended to 
linear systems theory (\cite{BenBoy05}, \cite{BenBoyd06}), with notable 
contributions from researchers in robust model predictive control and receding 
horizon control (see \cite{BempMor03}, \cite{KerrMacie04}, \cite{Loef02}, 
\citet{Boyd08}, and references therein). In all the papers, which usually deal with the more general 
case of multi-dimensional linear systems, the authors show how 
the reformulation can be done, and 
how the corresponding affine policies can be found by solving specific
types of optimization problems, which vary from linear and quadratic programs 
(\cite{BenBoy05}, \cite{KerrMacie04}, \cite{KerrMacie03}) to conic and 
semi-definite (\cite{BenBoy05}, \cite{Loef02}, \cite{BertDB07}, \cite{PabloMor03}),
or even multi-parametric, linear or quadratic programs (\cite{BempMor03}). 
The first steps towards analyzing the properties of such 
parameterizations were made in \cite{KerrMacie04}, 
where the authors show that, under suitable conditions, 
the resulting affine parameterization has certain desirable system theoretic 
properties (stability and robust invariance). Another notable contribution 
was made by \cite{KerrGoul05}, who prove that 
the class of affine disturbance feedback policies is  
equivalent to the class of affine state feedback policies with memory of 
prior states, thus subsuming the well known classes of open-loop and 
pre-stabilizing control policies.
However, to the best of our knowledge, apart from these theoretical advances, 
there has been very little progress in proving results about the quality 
of the objective function values resulting from the use of such parameterizations.

Our main result, summarized in Theorem \ref{thm:main_theorem} of Section 
\ref{sec:optim-affine-polic}, is that, for Problem \ref{prob:initial_problem}
stated above, affine policies of 
the form (\ref{eq:param_in_disturbance}) are, in fact, optimal. 
Furthermore, we are able to prove that a certain 
(affine) relaxation of the state costs is also possible, without any loss of 
optimality, which gives rise to very efficient algorithms for computing 
the optimal affine policies when the state costs $h_k(\cdot)$ are piece-wise affine.
To the best of our knowledge, this is the first result 
of its kind, and it provides intuition and motivation for the widespread 
advocation of such policies in both theory and applications. 
Our theoretical results are tight
(if the conditions in Problem \ref{prob:initial_problem} are 
slightly perturbed, then simple counterexamples for Theorem \ref{thm:main_theorem}
can be found), and the proof of the theorem itself is atypical, consisting of a 
forward induction and making use of polyhedral geometry to construct 
the optimal affine policies. Thus, we are able to gain 
insight into the structure and properties of these policies, which 
we explore in connection with a classical inventory management problem.

The paper is organized as follows. Section 2 presents an 
overview of the Dynamic Programming formulation in state variable $x_k$, 
extracting the optimal policies $u_k^*(x_k)$ and optimal value functions $J_k^*(x_k)$, 
as well as some of their properties. Section 3 contains our main result, and 
briefly discusses some immediate extensions and computational implications.
In Section 4, we introduce the constructive proof for building the affine 
control policies and the affine cost relaxations, and present 
counterexamples that prevent a generalization of the results.
In Section 5, we explore our results in connection with a classical inventory management problem.
Section 6 presents our conclusions and directions for future research.

\section{Dynamic Programming Solution.}
\label{sec:DP_formulation}
As already mentioned in the introduction, the solution to Problem 
\ref{prob:initial_problem} can be obtained using a ``classical'' 
Dynamic Programming formulation \citep{Bkas05}, in which the state is taken to be 
$x_k$, and the optimal policies 
$u_k^*(x_k)$ and optimal value functions $J_k^*(x_k)$ are computed 
starting at the end of the planning horizon, $k=T$, and moving backwards 
in time. In this section, we will briefly outline the solution 
of the Dynamic Program for our problem, and will state some of the key 
properties that will be used throughout the rest of the paper.
For completeness, a full proof of the results is included in 
Section \ref{sec:appendix:dynam-progr-solution} of the Appendix.

In order to simplify the notation, we remark that, since the 
constraints on the controls $u_k$ and 
the bounds on the disturbances $w_k$ are time-varying, and independent for 
different time-periods, we can restrict attention, without loss of 
generality\footnote{Such a system can always be obtained 
by the linear change of variables
$\tilde{x}_k = \frac{x_k}{\prod_{i=1}^{k-1}\a_i}$, and by suitably 
scaling the bounds $L_k,U_k,\vlow{w}{k},\vup{w}{k}$.},
to a system with $\a_k=\b_k=\c_k=1$. With this simplification, the problem
that we would like to solve is the following:
\begin{equation*}
  (DP) \quad
  \begin{aligned}
    J_{mM} \bydef \quad &\underset{u_1}{\min} \left[ c_1 \cdot u_1 + 
      \underset{w_1}{\max} \left[ 
        h_1(x_2) + \dots + \underset{u_k}{\min} \left[ c_k \cdot u_k + 
          \underset{w_k}{\max} \left[ h_k(x_{k+1}) + 
            \dots \right. \right. \right. \right. \\
    \qquad \qquad \qquad & \qquad \qquad \qquad \qquad  \left. \left. \left.
          + \underset{u_{T}}{\min} \left( c_T \cdot u_T + \underset{w_{T}}{\max}
            ~ h_T(x_{T+1}) ~\right) \dots \right] \dots \right] ~\right] \\
    \text{s.t.} \quad &  x_{k+1} = x_k + u_k + w_k \\
    & L_k \leq u_k \leq U_k 
    \qquad \qquad \qquad \forall \, k \in \{1,2,\dots,T\} \\
    & w_k \in \W_k = [\vlow{w}{k},\vup{w}{k}] .
  \end{aligned}
\end{equation*}
The corresponding Bellman recursion for $(DP)$ can then be written as follows:
\begin{align*}
  J_{k}^*(x_k) &\bydef \underset{L_k \leq u_k \leq U_k}{\min} \left[ \,
    c_k \cdot u_k + \underset{w_k \in \W_k}{\max} \left[~ h_k(x_k+u_k+w_k)
      + J_{k+1}^*\left(x_k+u_k+w_k\right) \, \right] \, \right] \quad ,
\end{align*}
where $J^{*}_{T+1}(x_{T+1}) \equiv 0$. By defining:
\begin{equation}
  \begin{aligned}
    y_k &\bydef x_k + u_k \\
    g_k\left(y_k\right) &\bydef 
    \underset{w_k \in \W_k}{\max} \left[~ h_k(y_k+w_k) + J_{k+1}^*\left(y_k+w_k\right)  
      ~\right] ~,
  \end{aligned}
  \label{eq:DP:gk_definition}
\end{equation}
we obtain the following solution to the Bellman recursion
(see Section \ref{sec:appendix:dynam-progr-solution} in 
the Appendix for the derivation):
\begin{align}
  u_k^*(x_k) &=
  \begin{cases}
    U_k, & ~\text{if}~  x_k < \vlow{y}{k} - U_k\\
    -x_k + y^*, & ~\text{otherwise} \\
    L_k, &  ~\text{if}~ x_k > \vup{y}{k} - L_k
  \end{cases} \label{eq:DP:uk_star} \\
  J_k^*(x_k) = c_k \cdot u_k^*(x_k) + g_k\left( x_k + u_k^*(x_k) \right) &=
  \begin{cases}
    c_k \cdot U_k + g_k(x_k+U_k), & ~\text{if}~  x_k < \vlow{y}{k} - U_k\\
    c_k\cdot (y^* - x_k) + g_k(y^*), & ~\text{otherwise} \\
    c_k \cdot L_k + g_k(x_k+L_k), &  ~\text{if}~ x_k > \vup{y}{k} - L_k ~,
  \end{cases} \label{eq:DP:Jk_star} 
\end{align}
where $y^* \in \minofgt{k}$, and $\minofgt{k}$ represents the (compact) set of 
minimizers of the convex function $c_k \cdot y + g_k(y)$.
A typical example of the optimal control law and the optimal value function 
is shown in Figure \ref{fig:uk_Jk_piecewise}.
\begin{figure}[h]
  \begin{center}
    
    \hspace{0.3cm}
    \includegraphics*[scale=0.75]{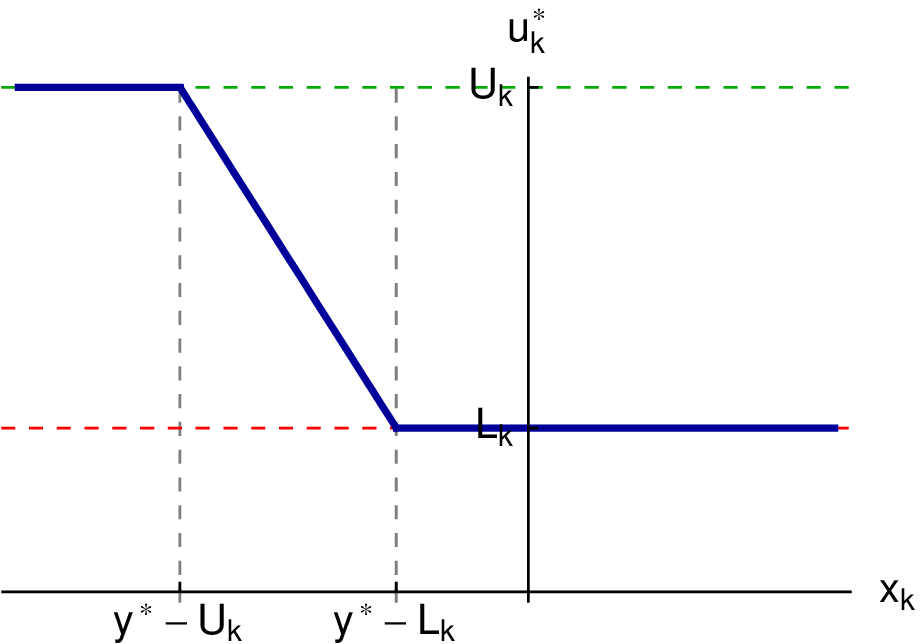}
    \hspace{0.4cm}
    \includegraphics*[scale=0.75]{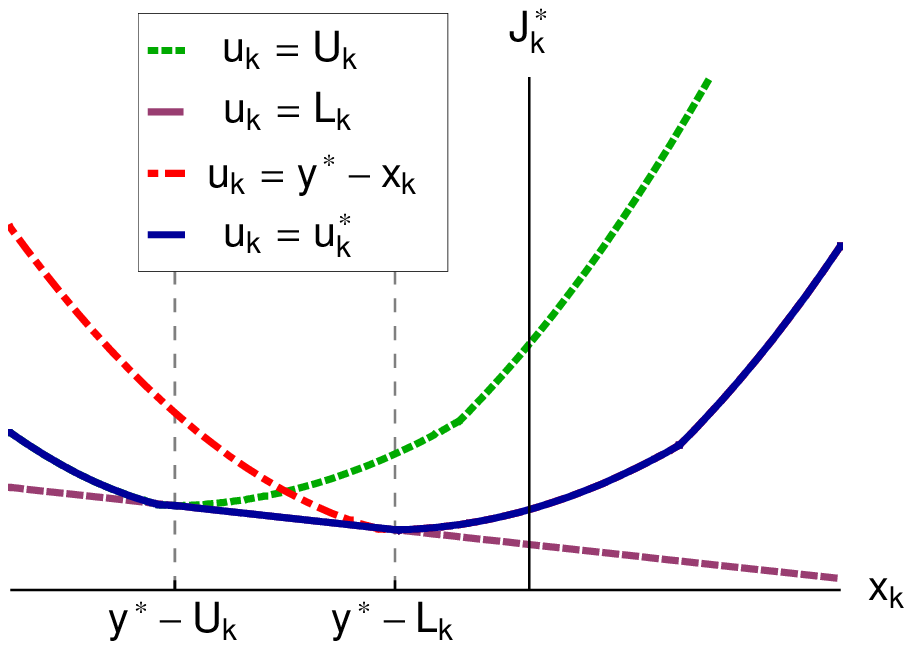}
    \caption{\label{fig:uk_Jk_piecewise} \small Optimal control law $u_k^*(x_k)$ 
      and optimal value function $J_k^*(x_k)$ at time $k$.}
  \end{center}
\end{figure}

The main properties of the solution that will be relevant 
for our later treatment are listed below:
\begin{itemize}
\item[\textbf{P1}] The optimal control law $u_k^*(x_k)$ is 
  piecewise affine, with 3 pieces, continuous and non-increasing.
\item[\textbf{P2}] The optimal value function, $J_k^*(x_k)$, and the function
  $g_k(y_k)$ are convex.
\item[\textbf{P3}] The difference in the values of the optimal control 
  law at two distinct arguments $s \leq t$ always satisfies:
  $u_k^*(s) - u_k^*(t) = - f_k \cdot (s-t)$, for some $f_k \in [0,1]$.
\item[\textbf{P4}] The optimal value function, $J_k^*(x_k)$, has a 
  subgradient at most $-c_k$ for $x_k < \vlow{y}{k}-U_k$, 
  exactly $-c_k$ in the interval $[\vlow{y}{k}-U_k,\vup{y}{k}-L_k]$, and at 
  least $-c_k$ for $x_k > \vup{y}{k}-L_k$.
\item[\textbf{P5}] The function $g_k(y_k)$ is decreasing, with a subgradient at most $-c_k$,
  in the interval $(-\infty,\vlow{y}{k}]$, and is increasing, with a sugradient 
  at least $-c_k$, in the interval $[\vup{y}{k},\infty)$.
\end{itemize}

\section{Optimality of Affine Policies in \texorpdfstring{$\mb{w}^k$}{wk}.}
\label{sec:optim-affine-polic}
In this section, we introduce our main contribution, namely a proof that 
policies that are affine in the disturbances $\mb{w}^k$ are, in fact, optimal for 
problem $(DP)$. Using the same notation as in Section \ref{sec:DP_formulation},
we can summarize our main result in the following theorem:
\begin{theorem}
  \label{thm:main_theorem}
  For every time step $k=1,\dots,T$, the following quantities exist:
  \begin{itemize}
  \item An affine control policy:
    \begin{align}
      q_k(\mb{w}^k) &\bydef 
      \aff{q}{k}{0} + \sum_{t=1}^{k-1} \aff{q}{k}{t} \cdot w_{t} \label{eq:affine_u}
    \end{align}
  \item An affine running cost:
    \begin{align}
      z_k(\mb{w}^{k+1}) &\bydef 
      \aff{z}{k}{0} + \sum_{t=1}^{k} \aff{z}{k}{t} \cdot w_{t} \label{eq:affine_cost}
    \end{align}
  \end{itemize}
  such that the following properties are obeyed:
  \begin{align}
    & L_k \leq q_k(\mb{w}^k) \leq U_k, \qquad
    \qquad \qquad \qquad \qquad \qquad \forall \, 
    \mb{w}^{k} \in \W_1\times \dots \times \W_{k-1}  
    \label{eq:constr:qk_robustly_feasible}\\
    & z_k(\mb{w}^{k+1}) \geq 
    h_k\left(x_1 + \sum_{t=1}^{k} \left( q_t(\mb{w}^t) + w_t \right) \right),
    \quad \quad \forall \, 
    \mb{w}^{k+1} \in \W_1\times \dots \times \W_{k}
    \label{eq:constr:affine_run_cost} \\
    & J_{mM} = \underset{w_1,\dots,w_k}{\max} \left[  \sum_{t=1}^k \left( c_t \cdot q_t(\mb{w}^t) +
        z_t(\mb{w}^{t+1}) \right) + 
      J_{k+1}^*\left(x_1 + \sum_{t=1}^{k} \left( q_t(\mb{w}^t) + w_t \right) \right)
    \right].
    \label{eq:constr:same_objective}
  \end{align}
\end{theorem}
Let us interpret the main 
statements and results in the theorem. Equation 
\eqref{eq:constr:qk_robustly_feasible}
confirms the existence of an affine policy $q_k(\mb{w}^k)$ 
that is robustly feasible, i.e. that obeys the control constraints, 
no matter what the realization of the disturbances may be. 
Equation \eqref{eq:constr:affine_run_cost} states 
the existence of an affine cost $z_k(\mb{w}^{k+1})$ 
that is always larger than the convex state cost $h_k(x_{k+1})$ 
incurred when the affine policies $\{q_t(\cdot)\}_{1 \leq t \leq k}$ are used. 
Equation \eqref{eq:constr:same_objective} guarantees that, 
despite using the (suboptimal) affine control law
$q_k(\mb{w}^k)$, and 
incurring a (potentially larger) affine stage cost 
$z_k(\mb{w}^{k+1})$, the overall objective function 
value $J_{mM}$ is, in fact, not increased.
This translates in the following two main results:
\begin{itemize}
\item \emph{Existential result}. Affine policies $q_k(\mb{w}^k)$ 
  are, in fact, optimal for Problem \ref{prob:initial_problem}.
\item \emph{Computational result}. When the convex costs $h_k(x_{k+1})$ are 
  piecewise affine, the optimal affine policies $\left\{q_k(\mb{w}^k)\right\}_{1 \leq k \leq T}$ 
  can be computed by solving a Linear Programming problem.
\end{itemize}
To see why the second implication would hold, suppose 
that $h_k(x_{k+1})$ is the maximum of $m_k$ affine functions:
\begin{align*}
  h_k(x_{k+1}) = \max \left( p_k^i \cdot x_{k+1} + p_{k,0}^i \right), ~
  i \in \{1,\dots,m_k\}.
\end{align*}
Then an optimal affine policy $q_k(\mb{w}^k)$ can be obtained by 
solving the following optimization problem (see \citet{BenBoy05,BenTal05}):
\begin{align*}
  \underset{
    \begin{array}{c} J; \{\aff{q}{k}{t}\}; \{\aff{z}{k}{t}\}
    \end{array}
  }
  {\min}  & \quad \quad J \\
  \text{s.t.} \quad & \quad \forall \, \mb{w} \in \W_1 \times \W_2 
  \times \dots \times \W_T, \quad \forall \, k = 1,\dots,T : \\
  (AARC) \qquad \qquad \qquad & J \geq \sum_{k=1}^T\left[ c_k\cdot \aff{q}{k}{0} + \aff{z}{k}{0} + 
    \sum_{t=1}^{k-1} \left( c_t \cdot \aff{q}{k}{t} + 
      \aff{z}{k}{t} \right) \cdot w_t + \aff{z}{k}{k} \cdot w_k \right] \\ 
  \forall \, i \in \{1,\dots,m_k\}: \quad  
  & \aff{z}{k}{0} + \sum_{t=1}^{k} \aff{z}{k}{t} \cdot w_t 
  \geq p_k^i \cdot \left[~ x_1 + \sum_{t=1}^{k} \left( \aff{q}{t}{0} + 
     \sum_{\tau=1}^{t-1} \aff{q}{t}{\tau} \cdot w_\tau + w_t \right) 
   ~\right] + p_{k,0}^i ~\\
  & L_k \leq \aff{q}{k}{0} + \sum_{t=1}^{k-1}\aff{q}{k}{t} \cdot w_{t} \leq U_k.
\end{align*}
Although Problem $(AARC)$ is still a semi-infinite LP (due to the requirement of 
robust constraint feasibility, $\forall \, \mb{w}$), since all the constraints are 
inequalities that are bi-affine in the decision variables and the uncertain quantities, 
a very compact reformulation of the problem is available. In particular, with a typical 
constraint in $(AARC)$ written as:
\begin{equation*}
  \lambda_0[\mb{x}] + \sum_{t=1}^T \lambda_t[\mb{x}]\cdot w_t 
  \leq 0, \qquad  \forall\, \mb{w} \in \W_1 \times \W_2 
  \times \dots \times \W_T ~,
\end{equation*}
where $\lambda[\mb{x}]$ denotes that $\lambda$ is an affine 
function of the decision variables $\mb{x}$, it can be shown 
(see \citet{BenTal02}, \citet{BenTal04} for details) that the previous condition 
is equivalent to:
\begin{align}
  \begin{cases}
    \lambda_0[\mb{x}] + 
    \sum_{t=1}^T \left( \lambda_t[\mb{x}] \cdot \frac{\vlow{w}{t}+\vup{w}{t}}{2} + 
      \frac{\vup{w}{t}-\vlow{w}{t}}{2} \cdot \xi_t \right) \leq 0 \\
    -\xi_t \leq \lambda_t[\mb{x}] \leq \xi_t, ~~ t =1,\dots,T ~,
  \end{cases}
  \label{eq:typical_constr}
\end{align}
which are linear constraints in the decision variables $\mb{x}, \mb{\xi}$. 
Therefore, $(AARC)$ can be reformulated as a Linear Program, with 
$O\left(T^2 \cdot \max_k m_k \right)$ variables and $O\left(T^2 \cdot \max_k 
  m_k \right)$ constraints, which can be solved very efficiently using 
commercially available software. 

We will conclude our observations by making one last remark related to an 
immediate extension of the results. Note that in the statement of Problem
\ref{prob:initial_problem}, there was no mention about constraints on the 
states $x_k$ of the dynamical system. In particular, one may want to incorporate 
lower or upper bounds on the states, as well:
\begin{align}
  L^x_k \leq x_k \leq U^x_k
  \label{eq:state_constraints}
\end{align}
We claim that, in case the mathematical problem including such constraints 
remains feasible\footnote{Such constraints may lead to infeasible 
problems. For example: 
$T=1, x_1=0, u_1 \in [0,1], w_1 \in [0,1], x_2 \in [5,10]$.},
then affine policies will, again, be optimal. The reason is that such 
constraints can always be simulated in our current framework, by adding 
suitable convex barriers to the stage costs $h_k(x_{k+1})$. In particular, by 
considering the modified, convex\footnote{The functions $\tilde{h}_k(\cdot)$ are 
convex in $x_{k+1}$, since $\mathbf{1}([L_{k+1},U_{k+1}])$, the indicator function of a 
convex set, is convex (see pages 28, 33 in \cite{Rock70} or  
Example 7.1.2 in \cite{Bkas03}), 
and the sum of convex functions remains convex.} stage costs:
\begin{align*}
  \tilde{h}_k(x_{k+1}) &\bydef h_k(x_{k+1}) + \mathbf{1}_{[L_{k+1}^x, U_{k+1}^x]}(x_{k+1}), ~
  \text{where}\\
  \mathbf{1}_{[L_{k+1}^x, U_{k+1}^x]}(x_{k+1}) &\bydef
  \begin{cases}
    0, &~\text{if}~ x_{k+1} \in [L_{k+1}^x, U_{k+1}^x] \\
    \infty, &~\text{otherwise} ~,
  \end{cases}
\end{align*}
it can be easily seen that the original problem, with convex stage costs 
$h_k(\cdot)$ and state constraints \eqref{eq:state_constraints},
is equivalent to a problem with the modified stage costs 
$\tilde{h}_k(\cdot)$ and no state constraints. And, since affine 
policies are optimal for the latter problem, the result is immediate.
Therefore, our decision to exclude such constraints from the original 
formulation was made only for sake of brevity and conciseness of the proofs,
but without loss of generality.

\section{Proof of Main Theorem.}
\label{sec:proof-main-theorem}
The current section contains the proof of Theorem \ref{thm:main_theorem}.
Before presenting the details, we would first like to give some 
intuition behind the strategy of the proof, and introduce the 
organization of the material.

Unlike most Dynamic Programming proofs, which utilize backward 
induction on the time-periods, we will 
have to proceed with a \emph{forward} induction. Section \ref{sec:induct_hypo} will 
present a test of the first step of the induction, and will 
then introduce a detailed analysis of the consequences of the induction hypothesis.

We will then separate the completion of the induction step into two parts.
In the first part, handled in Section \ref{sec:constr-affine-contr}, by 
exploiting the structure provided by the forward induction hypothesis, 
and making critical use of the properties of the 
optimal control law $u_k^*(x_k)$ and optimal value function $J_k^*(x_k)$
(the solutions to problem $(DP)$ in Section \ref{sec:DP_formulation}), 
we will be able to introduce a candidate affine policy $q_k(\mb{w}^k)$.
In Section \ref{sec:affine_controller_preserves_objective},
we will then prove that this policy is robustly feasible, 
and preserves the min-max value of the overall problem, $J_{mM}$, when 
used in conjunction with the original, convex state costs, $h_k(x_{k+1})$.

Similarly, for the second part of the inductive step, (Section \ref{sec:construction_affine_stage_cost}),
by re-analyzing the feasible sets of the optimization 
problems resulting after the use of the (newly computed) affine policy $q_k(\mb{w}^k)$,
we will determine a candidate affine cost $z_k(\mb{w}^{k+1})$, 
which we will prove (Section \ref{sec:proof_of_corectness_affine_cost_construction})
to be always larger than the original convex state costs, $h_k(x_{k+1})$. However, 
despite this fact, in Section \ref{sec:proof_of_corectness_affine_cost_construction} we 
will also show that when this affine cost is incurred, the overall min-max value $J_{mM}$
remains unchanged, which will complete the proof of the inductive step. 

Section \ref{sec:proof_of_main_theorem} will then conclude the proof 
of Theorem \ref{thm:main_theorem}, and will 
outline several counterexamples that prevent an immediate extension of the result 
to more general cases.

\subsection{Induction Hypothesis.}
\label{sec:induct_hypo} 
As mentioned before, the proof of the theorem will utilize a \emph{forward} induction 
on the time-step $k$. Let us begin by verifying the induction at $k=1$.

Using the same notation as in Section \ref{sec:DP_formulation},
by taking the affine control to be $q_1 \bydef u_1^*(x_1)$, we immediately get that $q_1$, which is simply 
a constant, is robustly feasible, so \eqref{eq:constr:qk_robustly_feasible} is obeyed.

Furthermore, since $q_1 = u_1^*(x_1)$ is optimal, we can write the overall optimal objective 
value, $J_{mM}$, as:
\begin{align}
  J_{mM} = J_1^*(x_1) &\overset{}{=} \underset{u_1 \in [L_1,U_1]}{\min} \left[\, c_1 \cdot u_1 + 
    g_1(x_1+u_1) ~\right] = c_1 \cdot q_1 + g_1\left(x_1+q_1\right)  
  \overset{\eqref{eq:apdx:DP:gk_value}}{=} \nonumber \\
  & \overset{\eqref{eq:apdx:DP:gk_value}}{=} c_1 \cdot q_1 + {\max} \{ \,
  \left(h_1 + J_2^*\right)\left(x_1+q_1+\vlow{w}{1} \right), ~
  \left(h_1 + J_2^*\right)\left(x_1+q_1+\vup{w}{1} \right) \, \}.
  \label{eq:min_prob_at_time_1}
\end{align}
Next, we introduce the affine cost 
$z_1(w_1) \bydef \aff{z}{1}{0} + \aff{z}{1}{1} \cdot w_1$, 
where we constrain the coefficients $\aff{z}{1}{i}$ to satisfy 
the following system of linear equations:
\begin{align*}
  \begin{cases}
    \aff{z}{1}{0} + \aff{z}{1}{1} \cdot \vlow{w}{1} =& h_1(x_1+q_1+\vlow{w}{1}) \\
    \aff{z}{1}{0} + \aff{z}{1}{1} \cdot \vup{w}{1} =& h_1(x_1+q_1+\vup{w}{1}) ~.
  \end{cases}
\end{align*}
Note that for fixed $x_1$ and $q_1$, the function $z_1(w_1)$ 
is nothing but a linear interpolation of the mapping $w_1 \mapsto h_1(x_1+q_1+w_1)$, 
matching the value at points $\{\vlow{w}{1},\vup{w}{1}\}$. 
Since $h_1$ is convex, the linear interpolation defined above 
clearly dominates it, so condition \eqref{eq:constr:affine_run_cost} 
is readily satisfied. 
Furthermore, by (\ref{eq:min_prob_at_time_1}),
$J_{mM}$ is achieved for $w_1 \in \{\vlow{w}{1},\vup{w}{1}\}$, so 
condition \eqref{eq:constr:same_objective} is also obeyed.
  
Having checked the induction at time $k=1$, let us now assume that 
the statements of Theorem \ref{thm:main_theorem} are true for times $t=1,\dots,k$. Equation  \eqref{eq:constr:same_objective} written for stage $k$ then yields:
\begin{align}
  J_{mM} & = \underset{w_1,\dots,w_k}{\max} \left[ \,
    {\sum_{t=1}^k \left( c_t \cdot q_t(\mb{w}^t) + 
        z_t\left(\mb{w}^{t+1}\right) \right)} + 
    J_{k+1}^*\left({x_1 + \sum_{t=1}^k \left( q_t(\mb{w}^t) + w_t\right)} \right) \, \right]
  = \nonumber\\
  &= \underset{\left(\theta_1,\theta_2\right) \in \Theta}{\max} 
  \left[ \, \theta_1 + J_{k+1}^*(\theta_2) \, \right] ,
  \label{eq:Jstar_with_theta1_theta2} \\
  \text{where} \quad \Theta &\bydef \left\{ \left(\theta_1,\theta_2\right) : 
    \theta_1 \bydef \sum_{t=1}^k \left( c_t \cdot q_t(\mb{w}^t) + 
      z_t\left(\mb{w}^{t+1}\right) \right), ~~
    \theta_2 \bydef x_1 + \sum_{t=1}^k \left( q_t(\mb{w}^t) + w_t\right) \right\}.
  \label{eq:set_for_theta1_theta2}
\end{align}

Since $\left\{q_t\right\}_{1 \leq t \leq k}$ and $\left\{z_t\right\}_{1 \leq t \leq k}$
are affine functions, this implies that, although the uncertainties $w_1,\dots,w_k$ 
lie in a set with $2^k$ vertices 
(the hyper-rectangle $\W_1 \times \dots \times \W_k$), 
they are only able to affect the objective $J_{mM}$ through affine combinations taking values in the 
set $\Theta$. Such a polyhedron, arising as a $2$-dimensional affine 
projection of a $k$-dimensional hyper-rectangle, is called a \textit{zonogon}. 
It belongs to a larger class of polytopes, known as \textit{zonotopes}, 
whose combinatorial structure and properties are well documented 
in the discrete and computational geometry literature. 
The interested reader is referred to Chapter 7 of \citet{Ziegl03} for a very nice 
and accessible introduction. 
\begin{figure*}[h]
  \begin{center}
    \begin{psfrags}      
      \psfrag{v0}[l][][1]{$\mb{v}_0=\mb{v}^-$ \small \textbf{[000000]}}
      \psfrag{v1}[l][][1]{$\mb{v}_1$ \small \textbf{[100000]}}
      \psfrag{v2}[l][][1]{$\mb{v}_2\,$\small \textbf{[110000]}}
      \psfrag{v3}[l][][1]{$\mb{v}_3$ \small \textbf{[111000]}}
      \psfrag{v4}[l][][1]{$\mb{v}_4$ \small \textbf{[111100]}}
      \psfrag{v5}[l][][1]{$\mb{v}_5$ \small \textbf{[111110]}}
      \psfrag{v6}[l][][1]{$\mb{v}_6=\mb{v}^+$ \small \textbf{[111111]}}
      \psfrag{vj}[r][][1]{$\mb{v}_j$}
      \psfrag{vjbar}[c][][1]{$\mb{v}_j^{\#}$}
      \psfrag{t1}[][][1]{$\theta_1$}
      \psfrag{t2}[][][1]{$\theta_2$}
      \includegraphics[scale=0.45]{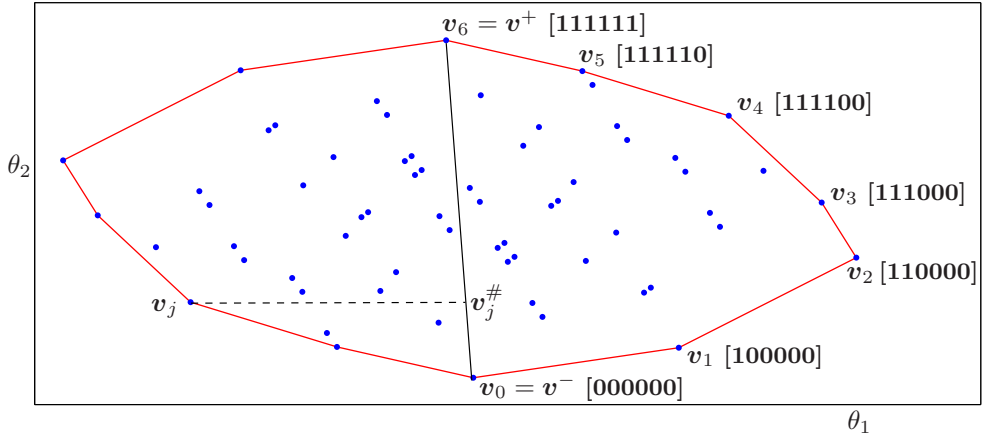}
    \end{psfrags}
    \caption{\small Zonogon obtained from projecting a hypercube in $\R^6$.}
    \label{fig:zonotope_R6}
  \end{center}
\end{figure*}

The main properties of a zonogon that we are interested in 
are summarized in Lemma \ref{lem:zonotope_properties}, found in the Appendix.
In particular, the set $\Theta$ is centrally 
symmetric, and has at most $2k$ vertices (see Figure \ref{fig:zonotope_R6} 
for an example). Furthermore, by numbering the vertices of $\Theta$ 
in counter-clockwise fashion, starting at:
\begin{align}
  \mb{v}_0 \equiv \mb{v}^- &\bydef 
  \argmax_{\theta_1} \left\{ 
    \argmin_{\theta_2} \left\{
      \mb{\theta} \in \Theta \right\} \right\}
  \label{eq:vertex_theta2_min} ~, 
\end{align}
we can establish the following result concerning the 
points of $\Theta$ that are relevant in our problem:
\begin{lemma}
  \label{lem:theta_restricted_region}
  The maximum value in optimization problem \eqref{eq:Jstar_with_theta1_theta2} is 
  achieved for $(\theta_1,\theta_2) \in 
  \left\{ \mb{v}_0,\mb{v}_1,\dots, \mb{v}_k\right\}$.
\end{lemma}
\begin{proof}
  The optimization problem described in \eqref{eq:Jstar_with_theta1_theta2}
  and \eqref{eq:set_for_theta1_theta2} is a maximization of a convex 
  function over a convex set. Therefore (see Section 32 of \citet{Rock70}),
  the maximum is achieved at the extreme points of the set $\Theta$, namely on the set 
  $\left \{\mb{v}_0, \mb{v}_1, \dots, \mb{v}_{2p-1},\mb{v}_{2p} \equiv \mb{v}_0 \right\}$,
  where $2p$ is the number of vertices of $\Theta$.
  Letting $\mb{O}$ denote the center of $\Theta$, by part (iii) of Lemma 
  \ref{lem:zonotope_properties} in the Appendix, we have that the vertex 
  symmetrically opposed to $\mb{v}^-$, 
  namely $\mb{v}^+ \bydef 2 \mb{O} - \mb{v}^-$, satisfies $\mb{v}^+ = \mb{v}_{p}$.

  Consider any vertex $\mb{v}_j$ with $j \in \{p+1,\dots,2p-1\}$. 
  From the definition of $\mb{v}^-,\mb{v}^+$,
  for any such vertex, there exists a point $\mb{v}_j^{\#} \in [\mb{v}^-,\mb{v}^+]$, 
  with the same $\theta_2$-coordinate as $\mb{v}_j$, but with 
  a $\theta_1$-coordinate larger than $\mb{v}_j$ (refer to Figure 
  \ref{fig:zonotope_R6}). Since such a point will have an objective in 
  problem \eqref{eq:Jstar_with_theta1_theta2} at least as large as 
  $\mb{v}_j$, and $\mb{v}_j^{\#} \in [\mb{v}_0,\mb{v}_p]$,
  we can immediately conclude that the maximum of problem \eqref{eq:Jstar_with_theta1_theta2}
  is achieved on the set 
  $\left\{\mb{v}_0,\dots,\mb{v}_p \right\}$. Since $2p \leq 2k$ 
  (see part (ii) of Lemma \ref{lem:zonotope_properties}), we can 
  immediately arrive at the conclusion of the lemma.
\end{proof}

Since the argument presented in the lemma will be recurring throughout several of our 
proofs and constructions, we will end this subsection by introducing two useful 
definitions, and generalizing the previous result. 

Consider the system of coordinates $(\theta_1,\theta_2)$ in $\R^2$, 
and let $\mathcal{S} \subset \R^2$ denote an arbitrary, finite set of points and 
$\mathcal{P}$ denote any (possibly non-convex) polygon such 
that its set of vertices is exactly $\mathcal{S}$. With 
$\mb{y}^- \bydef \argmax_{\theta_1} \left\{ \argmin_{\theta_2} \left\{
    \mb{\theta} \in \mathcal{S} \right\} \right\} , ~
\mb{y}^+ \bydef \argmax_{\theta_1} \left\{ \argmax_{\theta_2} \left\{ 
    \mb{\theta} \in \mathcal{S} \right\} \right\}$,
by numbering the vertices of the convex hull of $\mathcal{S}$ in a counter-clockwise fashion, 
starting at $\mb{y}_0 \bydef \mb{y}^-$, and with $\mb{y}_m = \mb{y}^+$, we define the 
\emph{right side} of $\mathcal{P}$ and the \emph{zonogon hull} of $\mathcal{S}$ as follows:
\begin{definition}
  \label{def:right_side}
  The \textbf{right side} of an arbitrary polygon $\mathcal{P}$ is:
  \begin{align}
    \rside \left( \mathcal{P} \right) &\bydef \left\{ \mb{y}_0,\mb{y}_1,\dots,\mb{y}_m \right\}.
    \label{eq:right_side_definition}
  \end{align}
\end{definition}
\begin{definition}
  \label{def:zonogon_hull}
  The \textbf{zonogon hull} of a set of points $\mathcal{S}$ is:
  \begin{align}
    \zhull \left( \mathcal{S} \right) &\bydef \left\{ \mb{y} \in \R^2 ~:~ 
      \mb{y} = \mb{y}_0 + \sum_{i=1}^m w_i \cdot \left( \mb{y}_i - \mb{y}_{i-1} \right) ~,~
      0 \leq w_i \leq 1 \right\}.
    \label{eq:zonogon_hull_definition}
  \end{align}
\end{definition}

\begin{figure*}[h]
  \begin{center}
    \begin{psfrags}      
      \psfrag{ymin}[l][][0.7]{$\mb{y}_0=\mb{y}^-$}
      \psfrag{y1}[l][][0.7]{$\mb{y}_1$}
      \psfrag{y2}[l][][0.7]{$\mb{y}_2$}
      \psfrag{y3}[l][][0.7]{$\mb{y}_3$}
      \psfrag{y4}[l][][0.7]{$\mb{y}_4$}
      \psfrag{y5}[l][][0.7]{$\mb{y}_5$}
      \psfrag{ymax}[l][][0.7]{$\mb{y}_m=\mb{y}^+$}
      \psfrag{t1}[][][0.7]{$\theta_1$}
      \psfrag{t2}[][][0.7]{$\theta_2$}
      \includegraphics[scale=0.275]{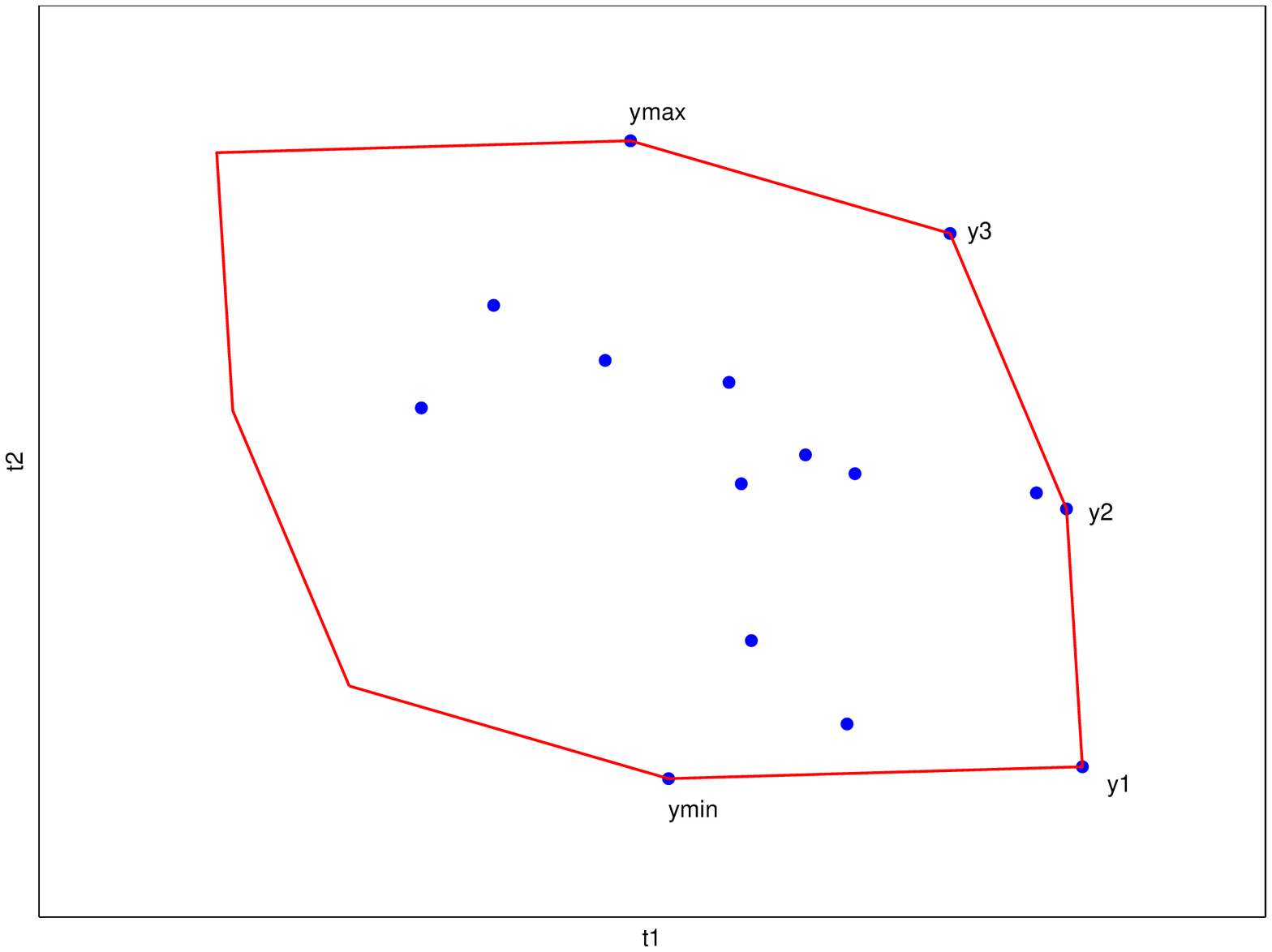}
      \includegraphics[scale=0.275]{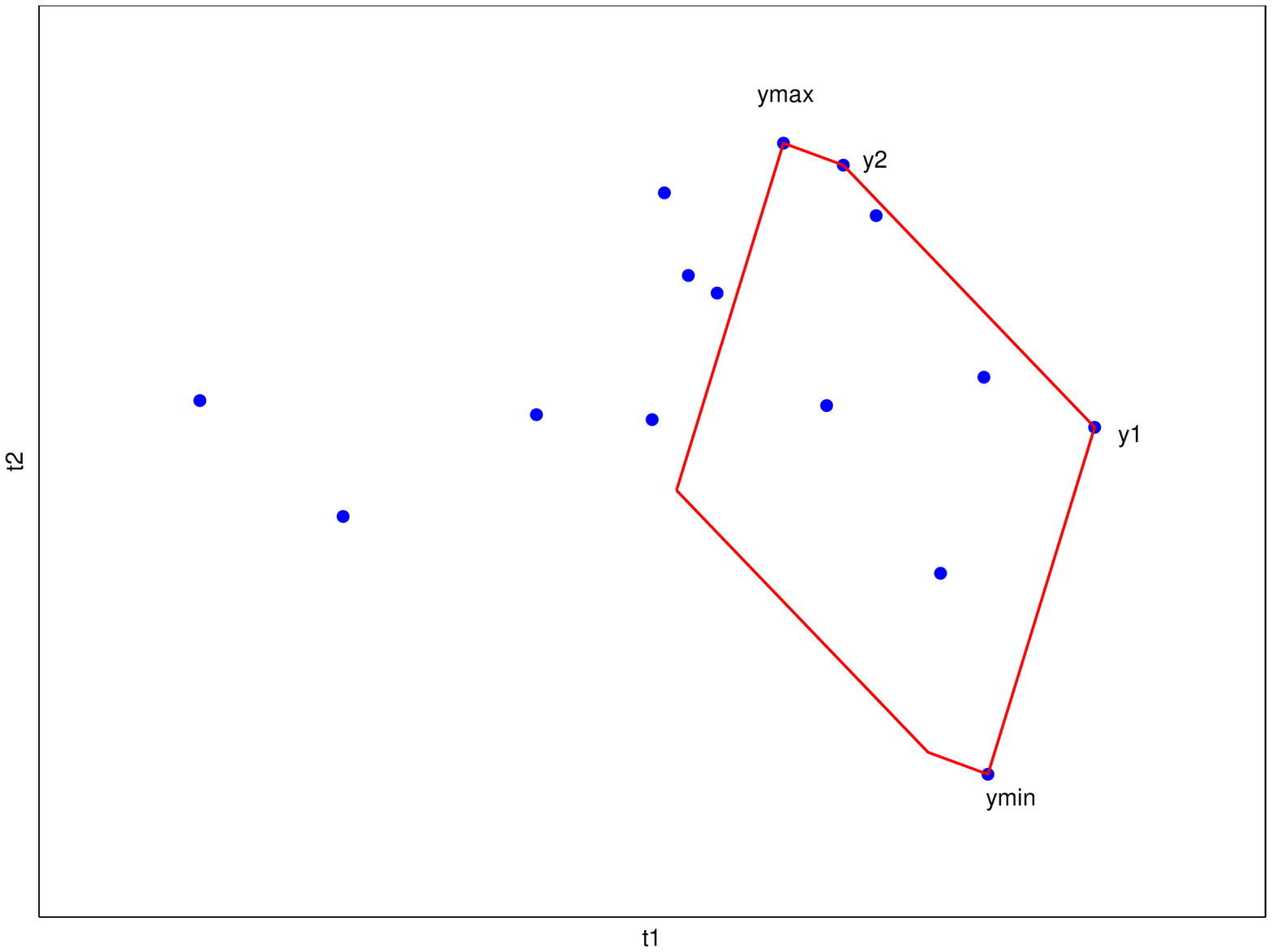}
      \includegraphics[scale=0.275]{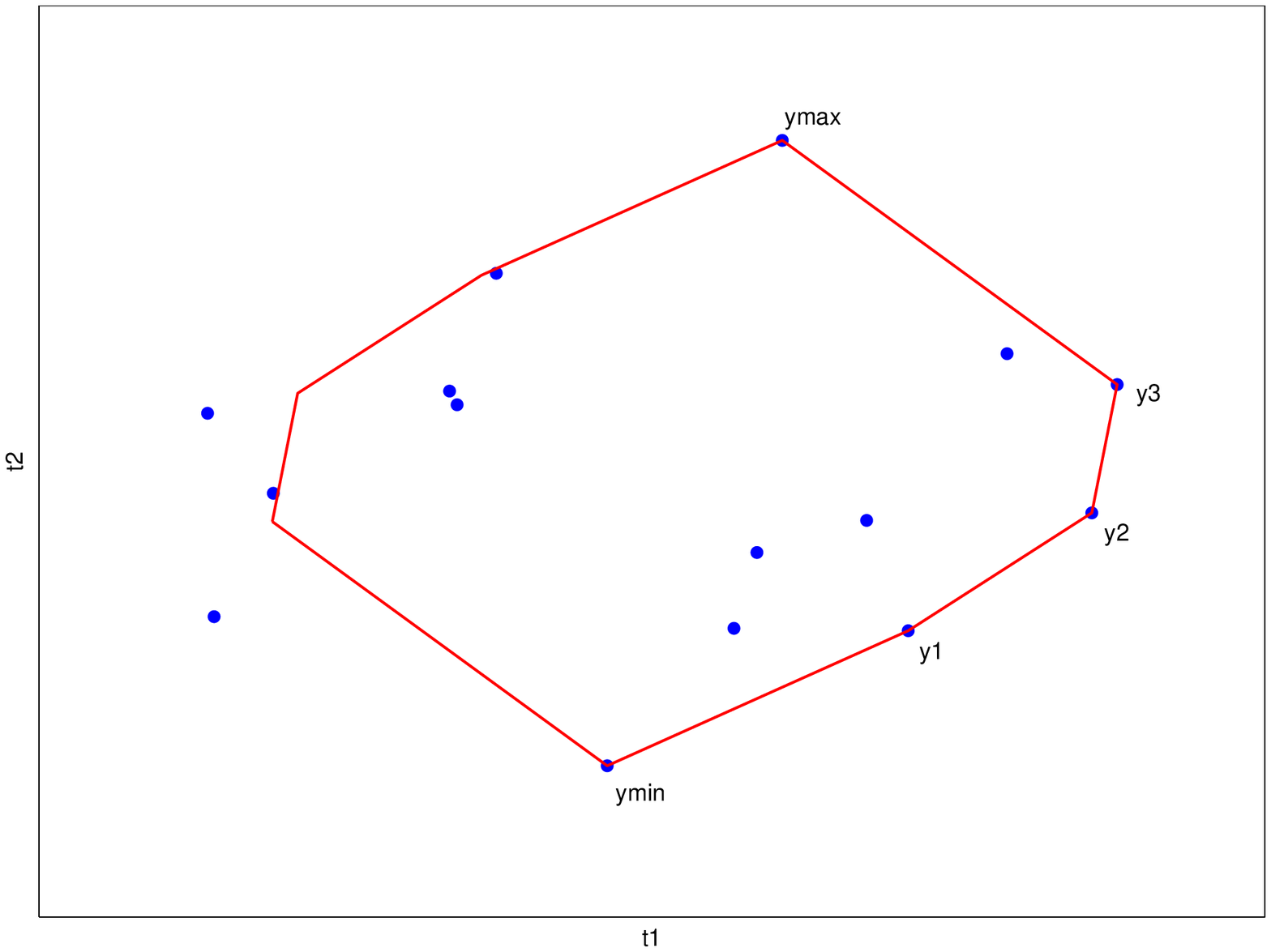}
    \end{psfrags}
    \caption{\small Examples of zonogon hulls for different sets $\mathcal{S} \in \R^2$.}
    \label{fig:zonogon_hulls}
  \end{center}
\end{figure*}

Intuitively, $\rside(\mathcal{P})$ represents exactly what the names hints at, i.e. the vertices 
found on the right side of $\mathcal{P}$. An equivalent definition using 
more familiar operators could be the following:
\begin{align*}
  \rside(\mathcal{P}) \equiv \ext \left( 
    \cone \left( 
      \left[ 
        \begin{array}{c}
          0 \\ -1
        \end{array} \right]  
    \right) + \conv\left( \mathcal{P} \right)
  \right), 
\end{align*}
where $\cone(\cdot)$ and $\conv(\cdot)$ represent the conic and convex hull, 
respectively, and $\ext(\cdot)$ denotes the set of extreme points. 

Using Definition \ref{def:zonotope} in Section \ref{sec:zonotope-properties} of 
the Appendix, one can see that the \emph{zonogon hull} of a set $\mathcal{S}$ is
simply a zonogon that has exactly the same vertices on the right side as the convex 
hull of $\mathcal{S}$, i.e. $\rside\left(\zhull\left(\mathcal{S}\right)\right) = 
\rside\left(\conv\left(\mathcal{S}\right)\right)$. Some examples of zonogon hulls 
are shown in Figure \ref{fig:zonogon_hulls} (note that the initial points in 
$\mathcal{S}$ do not necessarily fall inside the zonogon hull, and, as such, there is no general 
inclusion relation between the zonogon hull and the convex hull). The reason for 
introducing this object is that it allows for the following immediate generalization of Lemma
\ref{lem:theta_restricted_region}:
\begin{corollary}
  \label{corol:max_on_zonogon_hull}
  If $\mathcal{P}$ is any polygon in $\R^2 \equiv (\theta_1,\theta_2)$ with a finite 
  set $\mathcal{S}$ of vertices, and $f : \R \rightarrow \R$ is any convex function,
  then the following string of equalities holds:
  \begin{align*}
    \underset{(\theta_1,\theta_2) \in \mathcal{P}}{\max} 
    \left[ \, \theta_1 + f(\theta_2) \, \right] &=\,
    \underset{(\theta_1,\theta_2) \in \conv(\mathcal{P})}{\max} 
    \left[ \, \theta_1 + f(\theta_2) \, \right] \,=\,
    \underset{(\theta_1,\theta_2) \in \mathcal{S}}{\max} 
    \left[ \, \theta_1 + f(\theta_2) \, \right] \,=\, 
    \underset{(\theta_1,\theta_2) \in \rside(\mathcal{P})}{\max} 
    \left[ \, \theta_1 + f(\theta_2) \, \right] \\
    &= \,
    \underset{(\theta_1,\theta_2) \in \zhull \left( \mathcal{S} \right)}{\max} 
    \left[ \, \theta_1 + f(\theta_2) \, \right] \,=\,
    \underset{(\theta_1,\theta_2) \in \rside \left( \zhull \left( \mathcal{S} \right) 
      \right)}{\max} 
    \left[ \, \theta_1 + f(\theta_2) \, \right].
  \end{align*}
\end{corollary}
\begin{proof}
  The proof is identical to that of Lemma \ref{lem:theta_restricted_region}, and is omitted 
  for brevity.
\end{proof}

Using this result, whenever we will be faced with a maximization of a convex 
function $\theta_1 + f(\theta_2)$, we will be able to switch between different feasible sets, 
without affecting the overall optimal value of the optimization problem.

In the context of Lemma \ref{lem:theta_restricted_region}, the above result allows 
us to restrict attention from a potentially 
large set of relevant points (the $2^k$ vertices of the hyper-rectangle 
$\W_1 \times \dots \times \W_k$), to the $k+1$ vertices found on the right side
of the zonogon $\Theta$, which also gives insight into why 
the construction of an affine controller $q_{k+1}(\mb{w}^{k+1})$ with $k+1$ degrees of freedom, 
yielding the same overall objective function value $J_{mM}$, might actually be possible.

In the remaining part of Section \ref{sec:induct_hypo}, we would like to further narrow 
down this set of relevant points, by using the structure and properties of the optimal 
control law $u_{k+1}^*(x_{k+1})$ and optimal value function $J_{k+1}^*(x_{k+1})$, derived in 
Section \ref{sec:DP_formulation}. Before proceeding, however, we will first reduce 
the notational clutter by introducing several simplifications and assumptions.

\subsubsection{Simplified Notation and Assumptions.}
\label{sec:simplified-notation}
To start, we will omit writing the time subscripts ($k$ or $k+1$) whenever possible, and will take:
\begin{align}
  \theta_1(\mb{w}) \bydef a_0 + \sum_{i=1}^k a_i \cdot w_i;
  \quad \theta_2(\mb{w}) \bydef b_0 + \sum_{i=1}^k b_i \cdot w_i; 
  \quad q_{k+1}(\mb{w}) \equiv q(\mb{w}) 
  \bydef q_0 + \sum_{i=1}^k q_i \cdot w_i ~, \label{eq:simple_notation_theta12}
\end{align}
where $\mb{a} = (a_1,\dots,a_k)$ and $\mb{b} = (b_1,\dots,b_k)$ are the 
\emph{generators} of the zonogon $\Theta$.
We will use the same counter-clockwise numbering of the 
vertices of $\Theta$ as introduced earlier in Section \ref{sec:induct_hypo}:
\begin{align}
  \mb{v}_0 \bydef \mb{v}^{-} , \dots, \mb{v}_p \bydef \mb{v}^+, \dots ,
  \mb{v}_{2p} = \mb{v}^{-} ~,
  \label{eq:v0_vk_counterclockwise_def}
\end{align}
where $2p$ is the number of vertices of $\Theta$, and we will use the following 
overloaded notation for $\theta_1$ and $\theta_2$: 
\begin{itemize}
\item $\theta_i(\mb{w})$, for $\mb{w} \in \R^k$, 
  will denote the $\theta_i$ value assigned 
  by the affine projection to $\mb{w}$, via 
  (\ref{eq:simple_notation_theta12}).
\item $\theta_i\left[\mb{v}\right]$, applied to $\mb{v} \in \R^2$,
  will be used to denote the $\theta_i$-coordinate of the point $\mb{v}$.
\end{itemize}
Also, since $\theta_2 \equiv x_{k+1}$, instead of referring to $J_{k+1}^*(x_{k+1})$ 
and $u_{k+1}^*(x_{k+1})$, we will use $J^*(\theta_2)$ and $u^*(\theta_2)$, 
and we will occasionally use the short-hand notations $u^*(\mb{v_i})$, $J^*(\mb{v_i})$ and 
$g\left(\mb{v}_i\right)$, instead of $u^*\left(\theta_2[\mb{v}_i] \right)$,
$J^*\left(\theta_2[\mb{v}_i] \right)$ and 
$g \left(\theta_2[\mb{v_i}] + u^*(\theta_2[\mb{v}_i])\right)$, respectively.

Since many of the sets of interest will lie in $\R^2$, for a system with coordinates
$(x,y)$, we define the following convenient notation for the cotangent of 
the angle formed by an oriented line segment $[\mb{M},\mb{N}]$ with the $x$-axis:
\begin{align}
  \cotan{\mb{M}}{\mb{N}} \bydef \frac{x_N-x_M}{y_N-y_M}, 
  \quad \text{where}~ \mb{M} = (x_M,y_M) \in \R^2, ~\mb{N} = (x_N,y_N) \in \R^2.
  \label{eq:cotan_definition}
\end{align}

We will also make the following simplifying assumptions:
\begin{assumption}
  \label{as:hypercube}
  The uncertainty vector at time $k$, $\mb{w}^k = (w_1,\dots,w_k)$, belongs to
  the unit hypercube $\H_k$ of $\R^k$, i.e. $\vlow{w}{i}=0,\vup{w}{i}=1,\, \forall\, i = 1,\dots,k$. 
\end{assumption}
\begin{assumption}
  \label{as:zonogon_max_vertices}
  The zonogon $\Theta$ has a maximal number of vertices, i.e. $p=k$.
\end{assumption}
\begin{assumption}
  \label{as:consecutive_vertices_zeros}
  The vertex of the hypercube projecting to  $\mb{v}_i, ~ i \in \{0,\dots,k\}$, 
  is exactly $[1,1,\dots,1,0,\dots,0]$, i.e. $1$ in the first $k$ components and $0$ thereafter 
  (see Figure \ref{fig:zonotope_R6}).
\end{assumption}

These assumptions are made only to facilitate the exposition, and 
result in no loss of generality. To see this, note that the conditions of Assumption \ref{as:hypercube} 
can always be achieved by adequate translation and scaling of the generators $\mb{a}$ and 
$\mb{b}$ (refer to Section \ref{sec:zonotope-properties} of the Appendix for more details), 
and Assumption \ref{as:consecutive_vertices_zeros} can be satisfied by renumbering 
the coordinates of the hyper-rectangle, i.e. renumbering the disturbances $w_1,\dots,w_k$.
As for Assumption \ref{as:zonogon_max_vertices}, we argue that an extension 
of our construction to the degenerate case $p<k$ is immediate (one could also remove 
the degeneracy by applying an infinitesimal perturbation to the generators $\mb{a}$ 
or $\mb{b}$, with infinitesimal cost implications). 

\subsubsection{Further Analysis of the Induction Hypothesis.}
\label{sec:further_underst_induc_hypo}
With this simplified notation, by using (\ref{eq:DP:Jk_star}) to express 
$J^*(\cdot)$ as a function of $u^*(\cdot)$ and $g(\cdot)$, 
we rewrite equation (\ref{eq:Jstar_with_theta1_theta2}) as follows:
\begin{equation}
  \begin{aligned}
    (OPT) \quad 
    J_{mM} &=\underset{\left(\tilde{\gamma}_1,\tilde{\gamma}_2\right) 
      \in \tilde{\Gamma}}{\max} 
    \left[~ \tilde{\gamma}_1 + g\left(\tilde{\gamma}_2\right)~\right] ~,\\
    \text{where} \quad \tilde{\Gamma} &\bydef 
    \left\{ ~(\tilde{\gamma}_1,\tilde{\gamma}_2) ~:~ 
      \tilde{\gamma}_1 \bydef \theta_1 + c \cdot u^*(\theta_2), ~~
      \tilde{\gamma}_2 \bydef \theta_2 + u^*(\theta_2), ~~
      (\theta_1,\theta_2) \in \Theta \right\}.
  \end{aligned}
  \label{eq:set_for_gamma1tilde_gamma2tilde}
\end{equation}
A characterization for the set $\tilde{\Gamma}=(\tilde{\gamma}_1,\tilde{\gamma}_2)$
can be obtained by replacing the optimal, piecewise affine control 
law $u^*(\theta_2)$, given by\footnote{For simplicity, 
  we focus on the case when 
  $g(\cdot)$ has a unique minimizer, such that $\vlow{y}{}=\vup{y}{}=y^*$
  in (\ref{eq:DP:uk_star}), (\ref{eq:DP:Jk_star}).
  The analysis could be immediately extended to the set of 
  minimizers, $[\vlow{y}{},\vup{y}{}]$.} (\ref{eq:DP:uk_star})
in equation (\ref{eq:set_for_gamma1tilde_gamma2tilde}):
\begin{align}
  \tilde{\Gamma} ~:~ 
  (\tilde{\gamma}_1,\tilde{\gamma}_2) &=
  \begin{cases}
    ~ (\theta_1+ c \cdot U, ~\theta_2 + U)~, & ~\text{if}~  \theta_2 < y^* - U \\
    ~(\theta_1 - c \cdot \theta_2 + c \cdot y^*,~ y^*)~, & ~\text{otherwise} \\
    ~ (\theta_1+ c \cdot L, ~\theta_2 + L)~, & ~\text{if}~  \theta_2 > y^* - L \\
  \end{cases}
  \label{eq:mapping_for_gamma_tilde_ustar_explicit}
\end{align}
Using the same overloaded notation\footnote{To avoid confusion, in addition to different 
types of parentheses, we will consistenly use tilded arguments to denote points $\tilde{\mb{v}}_i \in \tilde{\Gamma}$, and non-tilded arguments for $\mb{v} \in \Theta$.} 
for $\tilde{\gamma}_1, \tilde{\gamma}_2$:
\begin{itemize}
\item $\tilde{\gamma}_i(\mb{v})$, applied to $\mb{v} \in \Theta \subset \R^2$,
  will refer to the $\tilde{\gamma}_i$ value assigned by the mapping
  (\ref{eq:mapping_for_gamma_tilde_ustar_explicit}) to $\mb{v}$
\item $\tilde{\gamma}_i\left[\tilde{\mb{v}}\right]$, applied to $\tilde{\mb{v}} \in 
  \Gamma \subset \R^2$, will denote the $\tilde{\gamma}_i$ coordinate of the 
  point $\tilde{\mb{v}}$,
\end{itemize}
we can now provide a compact characterization for the maximizers in 
problem $(OPT)$ from (\ref{eq:set_for_gamma1tilde_gamma2tilde}):
\begin{lemma}
  \label{lem:maximum_in_OPT}
  The maximum in problem $(OPT)$ over $\tilde{\Gamma}$ is reached on 
  the right side of:
  \begin{align}
    \Delta_\Gamma &\bydef \conv\left(\left\{\tilde{\mb{v}}_0,
      \dots,\tilde{\mb{v}}_{k}\right\} \right)~,
    \label{eq:delta_gamma_convex_hull_defn}
  \end{align}
  where:
  \begin{align}
    \tilde{\mb{v}}_i &\bydef \left(~ \tilde{\gamma}_1(\mb{v}_i), ~ 
      \tilde{\gamma}_2(\mb{v}_i) ~\right) = 
    \left(~ \theta_1[\mb{v}_i] + c \cdot u^*(\mb{v}_i),~ 
      \theta_2[\mb{v}_i] + u^*(\mb{v}_i) ~\right), \quad 
    i \in \{0,\dots,k\}.
    \label{eq:vi_tilde_definition}
  \end{align}
\end{lemma}
\begin{proof}
  By Lemma \ref{lem:theta_restricted_region}, the maximum 
  in (\ref{eq:Jstar_with_theta1_theta2}) is reached at one of the vertices 
  $\mb{v}_0,\mb{v}_1,\dots,\mb{v}_k$ of the zonogon $\Theta$. Since this problem is equivalent 
  to problem $(OPT)$ in (\ref{eq:set_for_gamma1tilde_gamma2tilde}), 
  written over $\tilde{\Gamma}$, we can immediately conclude that the maximum 
  of the latter problem is reached at the points $\{\tilde{\mb{v}}_i\}_{1 \leq i \leq k}$ 
  given by (\ref{eq:vi_tilde_definition}). Furthermore, since $g(\cdot)$ is convex
  (see Property \textbf{P2} of the optimal DP solution, in Section \ref{sec:DP_formulation}), 
  we can apply Corollary \ref{corol:max_on_zonogon_hull}, and replace 
  the points $\tilde{\mb{v}}_i$ with the right side of their convex hull, 
  $\rside\left( \Delta_\Gamma \right)$, without changing the result of the optimization problem, 
  which completes the proof.
\end{proof}

Since this result will be central to our future 
construction and proof, we will spend the remaining part of the 
subsection discussing some of the properties of the main object of interest,
the set of points on the right side of $\Delta_\Gamma$, $\rside(\Delta_\Gamma)$.
To understand the geometry of the 
set $\Delta_\Gamma$ and its connection with the optimal control law, 
note that the mapping (\ref{eq:mapping_for_gamma_tilde_ustar_explicit}) 
from $\Theta$ to $\tilde{\Gamma}$ will discriminate points
$\mb{\theta} = (\theta_1,\theta_2) \in \Theta$ depending on their position relative 
to the horizontal band:
\begin{align}
  \B \bydef \left\{~ (\theta_1,\theta_2) \in \R^2 ~:~ 
    \theta_2 \in [y^*-U,y^*-L] ~\right\}.
  \label{eq:BLU_definition}
\end{align}

In particular, from (\ref{eq:mapping_for_gamma_tilde_ustar_explicit}) and the definition 
of $\mb{v}_0, \dots, \mb{v}_k$ in (\ref{eq:v0_vk_counterclockwise_def}) and
(\ref{eq:vertex_theta2_min}), we can distinguish a total of four distinct cases.
The first three, shown in Figure \ref{fig:zono_cases_trivial}, are very easy to analyze:
\begin{figure*}[h]
  \begin{center}
    \begin{psfrags}      
      \psfrag{v0}[l][][0.8]{$\mb{v}_0=\mb{v}^-$}
      \psfrag{v1}[l][][0.8]{$\mb{v}_1$}
      \psfrag{v2}[l][][0.8]{$\mb{v}_2\,$}
      \psfrag{v3}[l][][0.8]{$\mb{v}_3$}
      \psfrag{v4}[l][][0.8]{$\mb{v}_4$}
      \psfrag{vkminus1}[l][][0.8]{$\mb{v}_{k-1}$}
      \psfrag{vk}[l][][0.8]{$\mb{v}_k=\mb{v}^+$}
      \psfrag{t1}[][][0.8]{$\theta_1$}
      \psfrag{t2}[][][0.8]{$\theta_2$}
      \psfrag{band}[][][0.8]{$\B$}
      \psfrag{ysminusL}[l][][0.8]{$y^*-L$}
      \psfrag{ysminusU}[l][][0.8]{$y^*-U$}
      \includegraphics[scale=0.28]{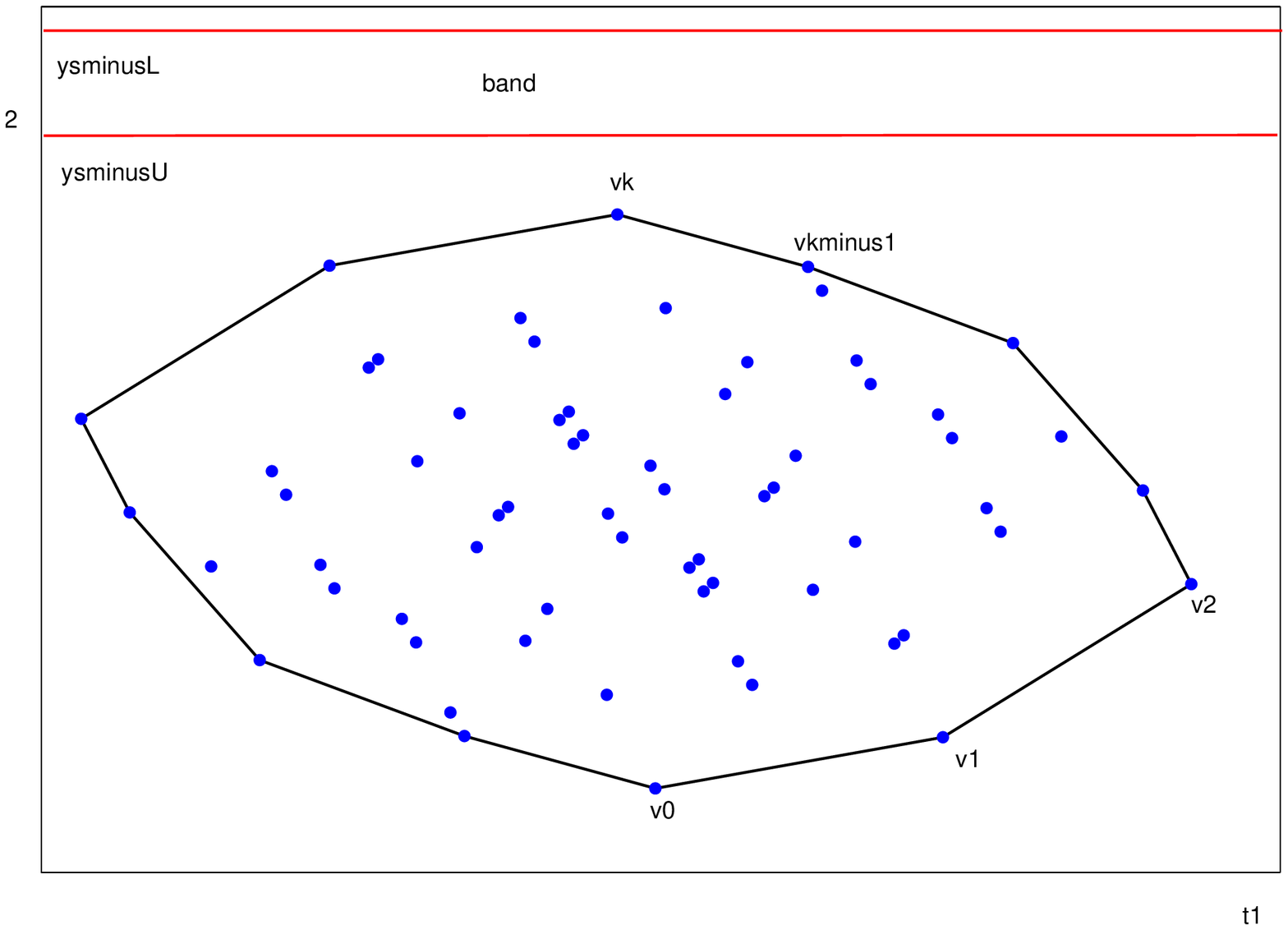}
      \includegraphics[scale=0.28]{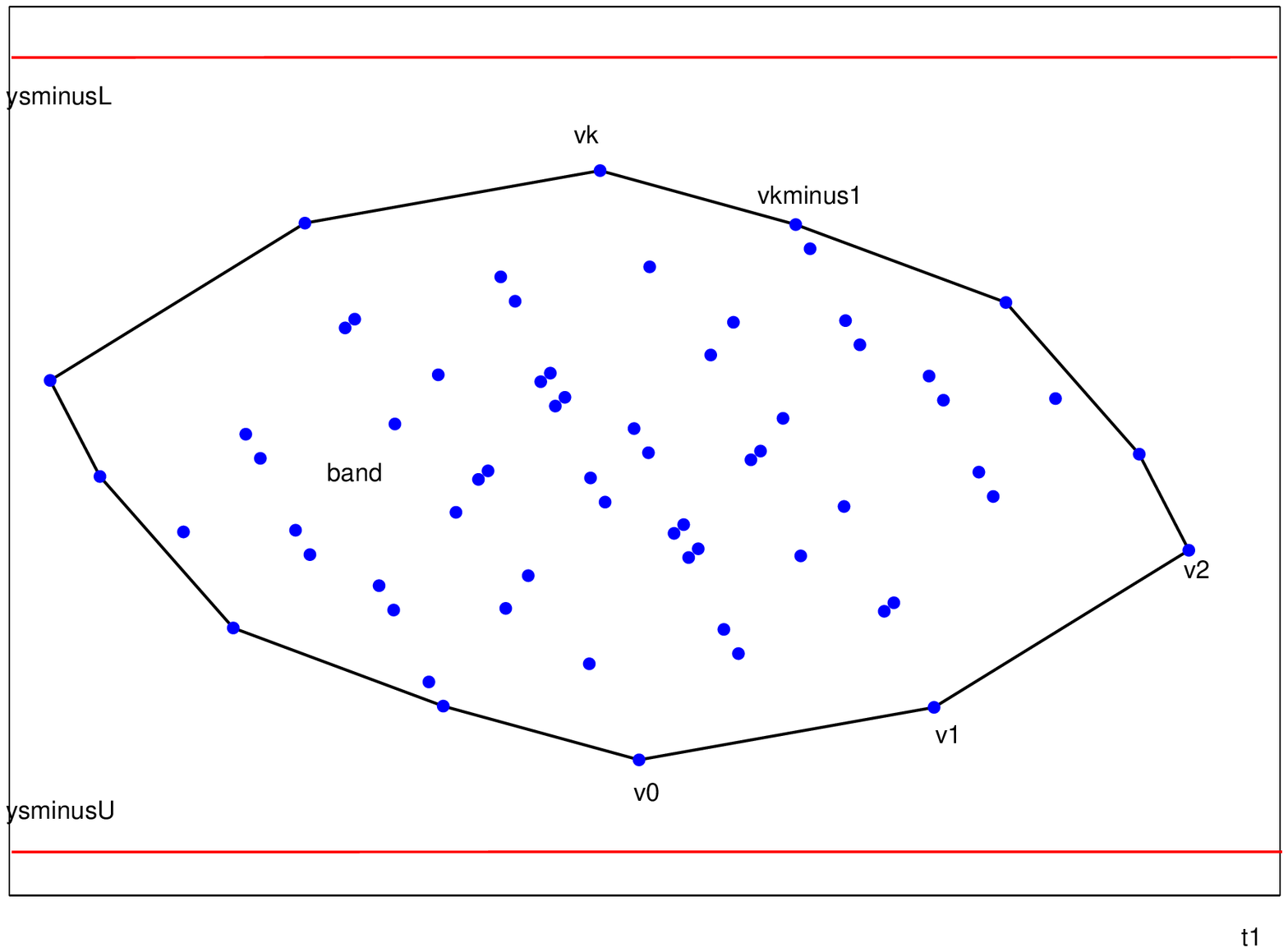}
      \includegraphics[scale=0.28]{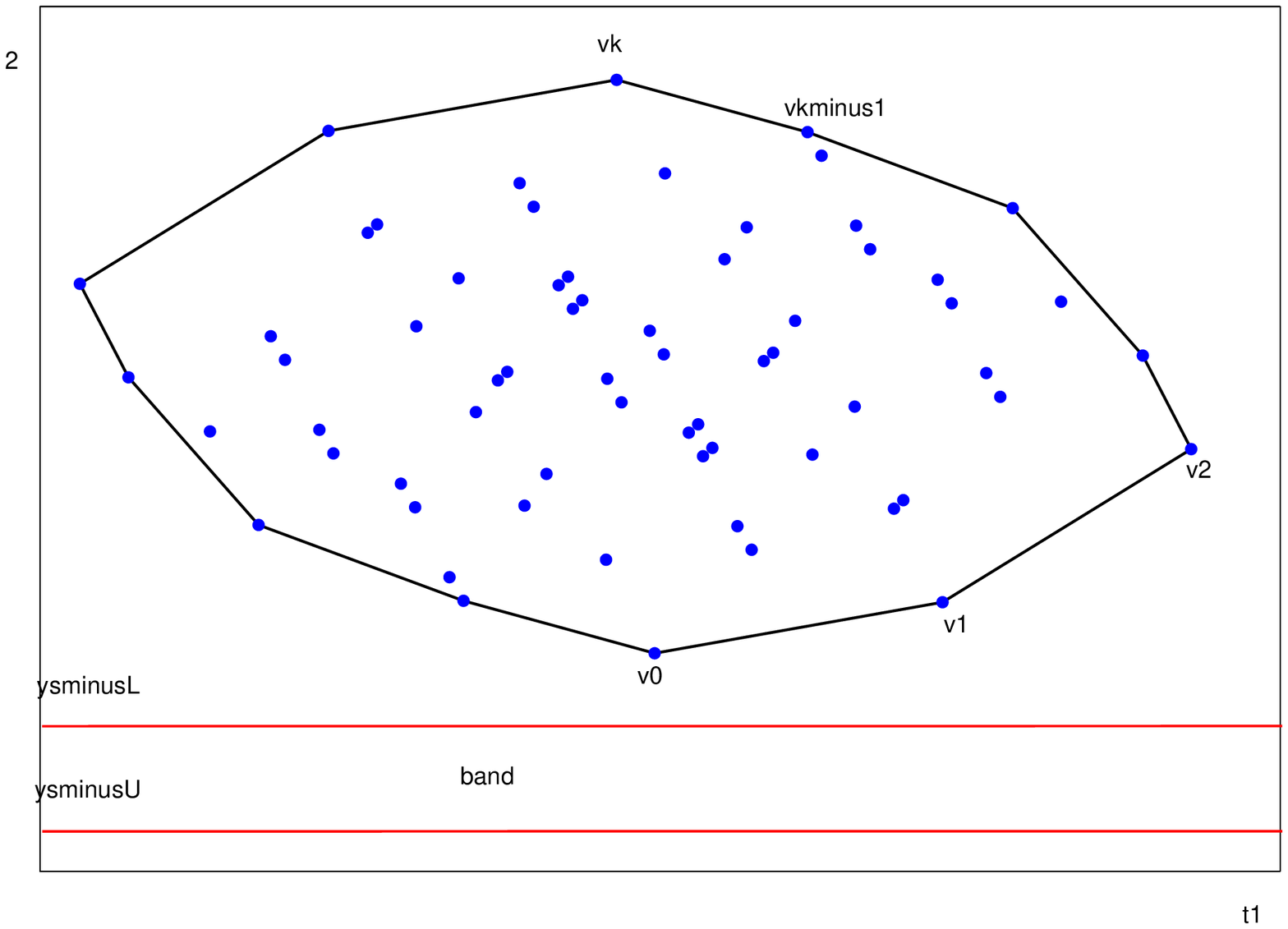}
    \end{psfrags}
    \caption{\small Trivial cases, when zonogon $\Theta$ lies entirely  
      \textbf{[C1]} below, \textbf{[C2]} inside, or \textbf{[C3]} above 
      the band $\B$.}
    \label{fig:zono_cases_trivial}
  \end{center}
\end{figure*}
\begin{itemize}
\item[\textbf{[C1]}] If the entire zonogon $\Theta$ falls below the band $\B$,
  i.e. $\theta_2\left[ \mb{v}_k \right] < y^*-U$, then 
  $\tilde{\Gamma}$ is simply a translation of $\Theta$, by 
  $(c \cdot U, U)$, so that 
  $\rside \left(\Delta_\Gamma \right) = 
  \{\tilde{\mb{v}}_0,\tilde{\mb{v}}_1,\dots,\tilde{\mb{v}}_k \}$.
\item[\textbf{[C2]}] If $\Theta$ lies inside the band $\B$, i.e. 
  $y^*-U \leq \theta_2\left[ \mb{v}_0 \right] \leq 
  \theta_2\left[ \mb{v}_k \right] \leq y^*-L$, then 
  all the points in $\tilde{\Gamma}$ will have 
  $\tilde{\gamma}_2 = y^*$, so $\tilde{\Gamma}$ will be 
  a line segment, and $\left| \rside\left(\Delta_\Gamma\right) \right| = 1$.
\item[\textbf{[C3]}] If the entire zonogon $\Theta$ falls above the band $\B$,
  i.e. $\theta_2\left[ \mb{v}_0 \right] > y^*-L$, then 
  $\tilde{\Gamma}$ is again a translation of $\Theta$, by  
  $(c \cdot L, L)$, so again 
  $\rside\left(\Delta_\Gamma\right) = \{\tilde{\mb{v}}_0,\tilde{\mb{v}}_1,\dots,\tilde{\mb{v}}_k \}$.
\end{itemize}

The remaining case, \textbf{[C4]}, is when $\Theta$ intersects the horizontal band $\B$ 
in a nontrivial fashion. We can separate this situation in the 
three sub-cases shown in Figure \ref{fig:zono_cases_nontrivial}, 
\begin{figure*}[h]
  \centering
  \begin{psfrags}
    \psfrag{vmin}[l][][0.8]{$\mb{v}_0=\mb{v}^-$}
    \psfrag{v1}[l][][0.8]{$\mb{v}_1$}
    \psfrag{v2}[l][][0.8]{$\mb{v}_2\,$}
    \psfrag{v3}[l][][0.8]{$\mb{v}_3$}
    \psfrag{v4}[l][][0.8]{$\mb{v}_4$}
    \psfrag{v5}[l][][0.8]{$\mb{v}_5$}
    \psfrag{v6}[l][][0.8]{$\mb{v}_6$}
    \psfrag{v7}[l][][0.8]{$\mb{v}_7$}
    \psfrag{vt}[l][][0.8]{$\mb{v}_t$}
    \psfrag{vmax}[l][][0.8]{$\mb{v}_k=\mb{v}^+$}
    \psfrag{y*minL}[l][][0.8]{$y^*-L$}
    \psfrag{y*minU}[l][][0.8]{$y^*-U$}
    \psfrag{band}[l][][0.8]{$\B$}
    \psfrag{t1}[][][0.8]{$\theta_1$}
    \psfrag{t2}[][][0.8]{$\theta_2$}

    \psfrag{g1}[][][0.8]{$\tilde{\gamma}_1$}
    \psfrag{g2}[][][0.8]{$\tilde{\gamma}_2$}
    \psfrag{tv0}[l][][0.8]{$\tilde{\mb{v}}_0$}
    \psfrag{tv1}[l][][0.8]{$\tilde{\mb{v}}_1$}
    \psfrag{tv2}[l][][0.8]{$\tilde{\mb{v}}_2$}
    \psfrag{tv3}[l][][0.8]{$\tilde{\mb{v}}_3$}
    \psfrag{tv4}[l][][0.8]{$\tilde{\mb{v}}_4$}
    \psfrag{tv5}[l][][0.8]{$\tilde{\mb{v}}_5$}
    \psfrag{tv6}[l][][0.8]{$\tilde{\mb{v}}_6$}
    \psfrag{tv7}[l][][0.8]{$\tilde{\mb{v}}_7$}
    \psfrag{tv8}[l][][0.8]{$\tilde{\mb{v}}_k$}
    \psfrag{tvt}[l][][0.8]{$\tilde{\mb{v}}_t$}
    \psfrag{y*}[l][][0.8]{$y^*$}
    \includegraphics[scale=0.32]{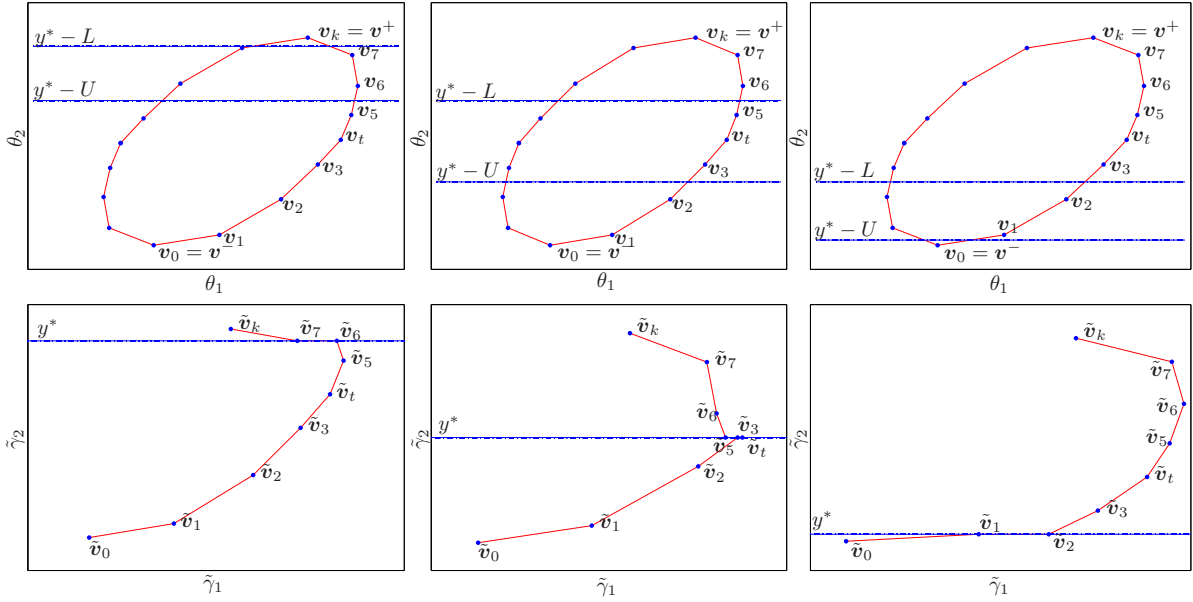}
  \end{psfrags}
  \caption{\small Case \textbf{[C4]}. Original zonogon $\Theta$ (first row) and the 
    set $\tilde{\Gamma}$ (second row) when $\mb{v}_t$ falls (a) under,
    (b) inside or (c) above the band $\B$.}
  \label{fig:zono_cases_nontrivial}       
\end{figure*}
depending on the position of the vertex $\mb{v}_t \in \rside(\Theta)$, where the 
index $t$ relates the per-unit control cost, $c$, with the geometrical 
properties of the zonogon:
\begin{align}
  t &\bydef
  \begin{cases}
    ~0~, &~\text{if}~ \frac{a_1}{b_1} \leq c \\
    ~\max \left\{i \in \{1,\dots,k\} : \frac{a_i}{b_i} > c \right\}~,
    &~\text{otherwise} ~.
  \end{cases}
  \label{eq:t_index_definition}
\end{align}
We remark that the definition of $t$ is consistent, since, by 
the simplifying Assumption \ref{as:consecutive_vertices_zeros}, the generators 
$\mb{a},\mb{b}$ of the zonogon $\Theta$ always satisfy:
\begin{align}
  \begin{cases}
    \frac{a_1}{b_1} > \frac{a_2}{b_2} > \dots > \frac{a_k}{b_k} \\
    b_1,b_2,\dots,b_k \geq 0.
  \end{cases}
  \label{eq:conditions_generators_theta}
\end{align}
An equivalent characterization of $\mb{v}_t$ can be obtained 
as the result of an optimization problem:
\begin{align}
  \label{eq:vt_as_result_of_optimization}
  \mb{v}_t &\equiv \argmin_{\theta_2} \left\{ \argmax_{(\theta_1,\theta_2) \in \Theta}
  \left\{ \theta_1 - c \cdot \theta_2 \right\} \right\}.
\end{align}

The following lemma summarizes all the relevant geometrical properties 
corresponding to this case:
\begin{lemma}
  \label{lem:rside_delta_gamma}
  When the zonogon $\Theta$ has a non-trivial intersection with 
  the band $\B$ (case \textup{\textbf{[C4]}}), 
  the convex polygon $\Delta_\Gamma$ and the set of 
  points on its right side, $\rside(\Delta_\Gamma)$, satisfy the following
  properties:
  \begin{enumerate}
  \item  $\rside(\Delta_\Gamma)$ is the union of two sequences of consecutive 
    vertices (one starting at $\tilde{\mb{v}}_0$, and one ending at 
    $\tilde{\mb{v}}_k$), and possibly an additional vertex, $\tilde{\mb{v}}_t$:
    \begin{equation*}
      \rside(\Delta_\Gamma) = \left\{\tilde{\mb{v}}_0,\tilde{\mb{v}}_1,\dots,
        \tilde{\mb{v}}_{s} \right\} \cup \{\tilde{\mb{v}}_t\} 
      \cup \left\{ \tilde{\mb{v}}_{r},\tilde{\mb{v}}_{r+1}
        \dots,\tilde{\mb{v}}_{k}\right\}
      ,~ \text{for some}~ 
      s \leq r \in \{0,\dots,k\}.
    \end{equation*}
  \item With $\cotan{\cdot}{\cdot}$ given by \eqref{eq:cotan_definition} applied 
    to the $(\tilde{\gamma}_1,\tilde{\gamma}_2)$ coordinates, we have that:
    \begin{align}
      \begin{cases}
        \cotan{\tilde{\mb{v}}_{s}}{\tilde{\mb{v}}_{\min(t,r)}}
        \geq \frac{a_{s+1}}{b_{s+1}}, & \textup{whenever $t > s$}\\
        \cotan{\tilde{\mb{v}}_{\max(t,s)}}{\tilde{\mb{v}}_{r}} \leq \frac{a_{r}}{b_{r}},
        & \textup{whenever $t < r$}.
      \end{cases}
      \label{eq:properties_KL_KU}
    \end{align}
  \end{enumerate}
\end{lemma}
While the proof of the lemma is slightly technical, which is why we have decided to 
leave it for Section \ref{sec:technical-lemmas} of the Appendix, its implications are more straightforward.
In conjuction with Lemma \ref{lem:maximum_in_OPT}, 
it provides a compact characterization of the points 
$\tilde{\mb{v}}_i \in \tilde{\Gamma}$ which are potential maximizers of 
problem $(OPT)$ in (\ref{eq:set_for_gamma1tilde_gamma2tilde}), which immediately 
narrows the set of relevant points $\mb{v}_i \in \Theta$ in optimization 
problem (\ref{eq:Jstar_with_theta1_theta2}), and, implicitly, the set 
of disturbances $\mb{w} \in \H_k$ that can achieve the min-max cost $J_{mM}$.

\subsection{Construction of the Affine Control Law.}
\label{sec:constr-affine-contr}
Having analyzed the consequences that result from using the induction 
hypothesis of Theorem \ref{thm:main_theorem}, we now return to the task of 
completing the inductive proof, which amounts to constructing 
an affine control law $q_{k+1}(\mb{w}^{k+1})$ and an affine cost 
$z_{k+1}(\mb{w}^{k+2})$ that verify conditions (\ref{eq:constr:qk_robustly_feasible}),
(\ref{eq:constr:affine_run_cost}) and (\ref{eq:constr:same_objective})
in Theorem \ref{thm:main_theorem}. We will separate this task into 
two parts. In the current section, 
we will exhibit an affine control law $q_{k+1}(\mb{w}^{k+1})$ that 
is robustly feasible, i.e. satisfies constraint (\ref{eq:constr:qk_robustly_feasible}), 
and that leaves the overall min-max cost $J_{mM}$ unchanged, when used at time $k+1$
in conjunction with the original convex state cost, $h_{k+1}(x_{k+2})$. 
The second part of the induction, i.e. the 
construction of the affine costs $z_{k+1}(\mb{w}^{k+2})$ satisfying 
(\ref{eq:constr:affine_run_cost}) and (\ref{eq:constr:same_objective}), 
will be left for Section \ref{sec:construction_affine_stage_cost}.

In the simplified notation introduced earlier, the problem we would like to solve 
is to find an affine control law $q(\mb{w})$ such that:
\begin{align*}
  J_{mM} &= \underset{\mb{w} \in [0,1]^k}{\max} 
  \left[~ \theta_1 + c \cdot q(\mb{w}) + 
    g\left(\theta_2+q(\mb{w})\right)~\right] \\
  L &\leq q(\mb{w}) \leq U ~, \quad \forall \, \mb{w} \in \H_k.
\end{align*}

The maximization represents the problem solved by the disturbances, 
when the affine controller, $q(\mb{w})$, is used instead of the 
optimal controller, $u^*(\theta_2)$. As such, the first equation 
amounts to ensuring that the overall objective function $J_{mM}$ remains unchanged, 
and the inequalities are a restatement of the robust feasibility condition.
The system can be immediately rewritten as:
\begin{align}
  \begin{aligned}
    (AFF) \quad 
    J_{mM} &=\underset{\left(\gamma_1,\gamma_2\right) 
      \in \Gamma}{\max} 
    \left[~ \gamma_1 + g\left(\gamma_2\right)~\right]  \\
    L &\leq q(\mb{w}) \leq U \quad ,
  \end{aligned}
  \label{eq:aff_problem}
\end{align}
where:
\begin{align}
  \Gamma &\bydef 
  \left\{ ~(\gamma_1,\gamma_2) ~:~ 
    \gamma_1 \bydef \theta_1 + c \cdot q(\mb{w}), ~~
    \gamma_2 \bydef \theta_2 + q(\mb{w}), ~~
    (\theta_1,\theta_2) \in \Theta \right\}.
  \label{eq:set_for_gamma1_gamma2}
\end{align}

With this reformulation, all our decision variables, i.e. the affine coefficients of $q(\mb{w})$,
have been moved to the feasible set $\Gamma$ of the maximization problem $(AFF)$ in 
(\ref{eq:aff_problem}). Note that, with an affine controller 
$q(\mb{w}) = q_0 + \scprod{q}{w}$, and $\theta_1,\theta_2$ affine in $\mb{w}$  
(\ref{eq:simple_notation_theta12}), the feasible set $\Gamma$ will represent a new zonogon 
in $\R^2$, with generators given by $\mb{a} + c \cdot \mb{q}$ and $\mb{b} + \mb{q}$. 
Furthermore, since the function $g$ is convex, the optimization problem $(AFF)$ over 
$\Gamma$ is of the exact same nature as that in (\ref{eq:Jstar_with_theta1_theta2}), defined 
over the zonogon $\Theta$. Thus, in perfect analogy with our discussion 
in Section \ref{sec:induct_hypo} (Lemma \ref{lem:theta_restricted_region} and 
Corollary \ref{corol:max_on_zonogon_hull}), we can 
conclude that the maximum in $(AFF)$ will occur at a vertex of $\Gamma$ found in 
$\rside(\Gamma)$.

In a different sense, note that optimization problem $(AFF)$ is also very similar to 
problem $(OPT)$ in (\ref{eq:set_for_gamma1tilde_gamma2tilde}), which 
was the problem solved by the uncertainties $\mb{w}$ when 
the optimal control law, $u^*(\theta_2)$, was used at time $k+1$. 
Since the optimal value of the latter problem 
is exactly equal to the overall min-max value, $J_{mM}$, we will interpret the 
equation in (\ref{eq:aff_problem}) as comparing the optimal values in the 
two optimization problems, $(AFF)$ and $(OPT)$. 

As such, note that the same convex 
objective function, $\xi_1 + g(\xi_2)$, is maximized in both problems, but over different feasible 
sets, $\tilde{\Gamma}$ for $(OPT)$ and $\Gamma$ for $(AFF)$, 
respectively. From Lemma \ref{lem:maximum_in_OPT} in Section \ref{sec:further_underst_induc_hypo},
the maximum of problem $(OPT)$ is reached on the set $\rside(\Delta_\Gamma)$, where 
$\Delta_{\Gamma} = \conv \left( \left\{ \tilde{\mb{v}}_0, \tilde{\mb{v}}_1, \dots, 
\tilde{\mb{v}}_k \right\} \right)$. From the discussion in the 
previous paragraph, the maximum in problem $(AFF)$ occurs on $\rside(\Gamma)$. 
Therefore, in order to compare the 
two results of the maximization problems, we must relate the sets 
$\rside(\Delta_\Gamma)$ and $\rside(\Gamma)$. 

In this context, we introduce the central idea behind the construction of the affine 
control law, $q(\mb{w})$. Recalling the concept of a \emph{zonogon hull} introduced in Definition 
\ref{def:zonogon_hull}, we argue that, if the affine coefficients of the 
controller, $q_0,\mb{q}$, were computed 
in such a way that the zonogon $\Gamma$ actually corresponded to the zonogon hull of the 
set $\left\{ \tilde{\mb{v}}_0, \tilde{\mb{v}}_1, \dots, \tilde{\mb{v}}_k \right\}$, then, 
by using the result in Corollary \ref{corol:max_on_zonogon_hull}, we could 
immediately conclude that the optimal values in $(OPT)$ and $(AFF)$ are the same.

To this end, we introduce the following procedure for computing 
the affine control law $q(\mb{w})$:
\begin{algorithm}[H]
  \caption{Compute affine controller $q(\mb{w})$}
  \label{alg1:findq}
  \begin{algorithmic}[1]
    \REQUIRE $\theta_1(\mb{w}), \theta_2(\mb{w}), g(\cdot), u^*(\cdot)$
    \IF{($\Theta$ falls below $\B$) \textbf{or} ($\Theta \subseteq \B$) 
      \textbf{or} 
      ($\Theta$ falls above $\B$)}
    \STATE Return $q(\mb{w}) = u^*(\theta_2(\mb{w}))$.
    \ELSE
    \STATE Apply the mapping \eqref{eq:mapping_for_gamma_tilde_ustar_explicit}
    to obtain the points $\tilde{\mb{v}}_i, ~i \in \{0,\dots,k\}$.
    \STATE Compute the set 
    $\Delta_{\Gamma} = 
    \conv \left( \left\{\tilde{\mb{v}}_0,\dots,\tilde{\mb{v}}_{k}\right\} \right)$.
    \STATE Let $\rside(\Delta_\Gamma) 
    = \{\tilde{\mb{v}}_{0},\tilde{\mb{v}}_{1},\dots,\tilde{\mb{v}}_{s}\} \cup \{\tilde{\mb{v}}_{t}\}
    \cup \{\tilde{\mb{v}}_{r},\dots,\tilde{\mb{v}}_{k}\}$ 
    be the set of points on the right side of $\Delta_\Gamma$.
    \STATE Solve the following system for $q_0,\dots,q_k$ and $K_U,K_L$:
    \begin{equation}
      (S) \qquad 
      \left\{
        \begin{aligned}
          q_0 + \dots + q_{i} &= u^*\left(\mb{v}_{i} \right), 
          && \forall\, \tilde{\mb{v}}_i \in \rside(\Delta_\Gamma) 
          && \text{(matching)}  \\
          \frac{a_i+c \cdot q_i}{b_i + q_i} &= K_{U}, 
          && \forall\, i \in \{s+1,\dots,\min(t,r)\} 
          && \text{(alignment below $t$)}  \\
          \frac{a_i+c \cdot q_i}{b_i + q_i} &= K_{L}, 
          && \forall\, i \in \{\max(t,s)+1,\dots,r\}
          && \text{(alignment above $t$)}
        \end{aligned}
      \right.
      \label{eq:system_for_q_coefficients}
    \end{equation}
    \STATE Return $q(\mb{w}) = q_0 + \sum_{i=1}^k q_i w_i$.
    \ENDIF
  \end{algorithmic}
\end{algorithm}

Before proving that the construction is well-defined and does produce the expected result, 
we will give some intuition for the constraints in system (\ref{eq:system_for_q_coefficients}).
In order to have the zonogon $\Gamma$ be the same as the zonogon hull of 
$\left\{\tilde{\mb{v}}_0,\dots,\tilde{\mb{v}}_{k}\right\}$, 
we must ensure that the vertices on the right side of $\Gamma$ exactly 
correspond to the points on the right side of 
$\Delta_\Gamma = \conv \left( \left\{\tilde{\mb{v}}_0,\dots,\tilde{\mb{v}}_{k}\right\} \right)$.
This is achieved in two stages. First, we ensure that vertices $\mb{w}_i$ of 
the hypercube $\H_k$ that are mapped by the optimal control law $u^*(\cdot)$ into
points $\tilde{\mb{v}}_i \in \rside(\Delta_\Gamma)$ (through the succession of mappings 
$\mb{w}_i \overset{\tiny \eqref{eq:set_for_theta1_theta2}}{\mapsto} 
\mb{v}_i \in \rside(\Theta) \overset{\tiny \eqref{eq:vi_tilde_definition}}{\mapsto}
\tilde{\mb{v}}_i \in \rside(\Delta_\Gamma)$), will be mapped by the affine 
control law, $q(\mb{w}_i)$, into the same point $\tilde{\mb{v}}_i$ 
(through the mappings $\mb{w}_i \overset{\tiny \eqref{eq:set_for_theta1_theta2}}{\mapsto} 
\mb{v}_i \in \rside(\Theta) \overset{\tiny \eqref{eq:set_for_gamma1_gamma2}}{\mapsto}
\tilde{\mb{v}}_i \in \rside(\Delta_\Gamma)$). This is done in the first set of 
constraints, by \emph{matching} the value of the optimal control law 
at any such points. Second, we ensure that any such matched points $\tilde{\mb{v}}_i$ 
actually correspond to the vertices on the right side of the zonogon 
$\Gamma$. This is done in the second and third set of constraints in 
(\ref{eq:system_for_q_coefficients}), by computing the affine coefficients $q_j$ 
in such a way that the resulting segments in the generators of the zonogon $\Gamma$, namely 
$\left(\begin{array}{c}a_j + c \cdot q_j \\ b_j +  q_j\end{array}\right)$,
are all \emph{aligned}, i.e. have the same cotangent, given by 
the $K_U,K_L$ variables. Geometrically, this exactly corresponds to 
the situation shown in Figure \ref{fig:matching_affine_control} below.

\begin{figure}[h]
  \centering
  \begin{psfrags}
    \psfrag{vmin}[l][][0.8]{$\mb{v}_0=\mb{v}^-$}
    \psfrag{v1}[l][][0.8]{$\mb{v}_1$}
    \psfrag{v2}[l][][0.8]{$\mb{v}_2$}
    \psfrag{v3}[l][][0.8]{$\mb{v}_3$}
    \psfrag{v4}[l][][0.8]{$\mb{v}_4=\mb{v}_t$}
    \psfrag{v5}[l][][0.8]{$\mb{v}_5$}
    \psfrag{v6}[l][][0.8]{$\mb{v}_6$}
    \psfrag{v7}[l][][0.8]{$\mb{v}_7$}
    \psfrag{vmax}[l][][0.8]{$\mb{v}_k=\mb{v}^+$}
    \psfrag{y*minL}[l][][0.8]{$y^*-L$}
    \psfrag{y*minU}[l][][0.8]{$y^*-U$}
    \psfrag{band}[l][][0.8]{$\B$}
    \psfrag{bigM}[l][][0.8]{}
    \psfrag{bigN}[l][][0.8]{}
    \psfrag{t1}[][][0.8]{$\theta_1$}
    \psfrag{t2}[][][0.8]{$\theta_2$}
    \psfrag{xlabel}[c][][0.8]{Original zonotope $\Theta$.}
    
    \psfrag{g1}[][][0.8]{$\tilde{\gamma}_1$}
    \psfrag{g2}[][][0.8]{$\tilde{\gamma}_2$}
    \psfrag{tv0}[l][][0.8]{$\tilde{\mb{v}}_0=\mb{y}_0$}
    \psfrag{tv1}[l][][0.8]{$\tilde{\mb{v}}_s=\mb{y}_s$}
    \psfrag{tv2}[l][][0.8]{$\tilde{\mb{v}}_2$}
    \psfrag{tv3}[l][][0.8]{$\tilde{\mb{v}}_3$}
    \psfrag{tv4}[l][][0.8]{$\tilde{\mb{v}}_4$}
    \psfrag{tv5}[l][][0.8]{$\tilde{\mb{v}}_5$}
    \psfrag{tv6}[l][][0.8]{$\tilde{\mb{v}}_6$}
    \psfrag{tv7}[l][][0.8]{$\tilde{\mb{v}}_r=\mb{y}_r$}
    \psfrag{tvt}[l][][0.8]{$\tilde{\mb{v}}_t=\mb{y}_t$}
    \psfrag{tvk}[l][][0.8]{$\tilde{\mb{v}}_k=\mb{y}_k$}
    \psfrag{tM}[l][][0.8]{}
    \psfrag{tN}[l][][0.8]{}
    \psfrag{CDG}[l][][0.8]{$\tilde{\Gamma}$}
    \psfrag{PVT}[l][][0.8]{$\tilde{\mb{v}}_i$}
    \psfrag{RSD}[l][][0.8]{$\rside\left(\Delta_\Gamma \right)$}
    \psfrag{MVT}[l][][0.8]{$\mb{y}_j \in \rside(\Gamma)$}
    \psfrag{y2}[l][][0.8]{$\mb{y}_2$}
    \psfrag{y3}[l][][0.8]{$\mb{y}_3$}
    \psfrag{y5}[l][][0.8]{$\mb{y}_5$}
    \psfrag{y6}[l][][0.8]{$\mb{y}_6$}
    \psfrag{GammaTil}[c][][0.8]{Set $\tilde{\Gamma}$ and points $\mb{y}_j \in \rside(\Gamma)$.}
    \psfrag{y*}[l][][0.8]{$y^*$}
    \includegraphics[scale=0.37]{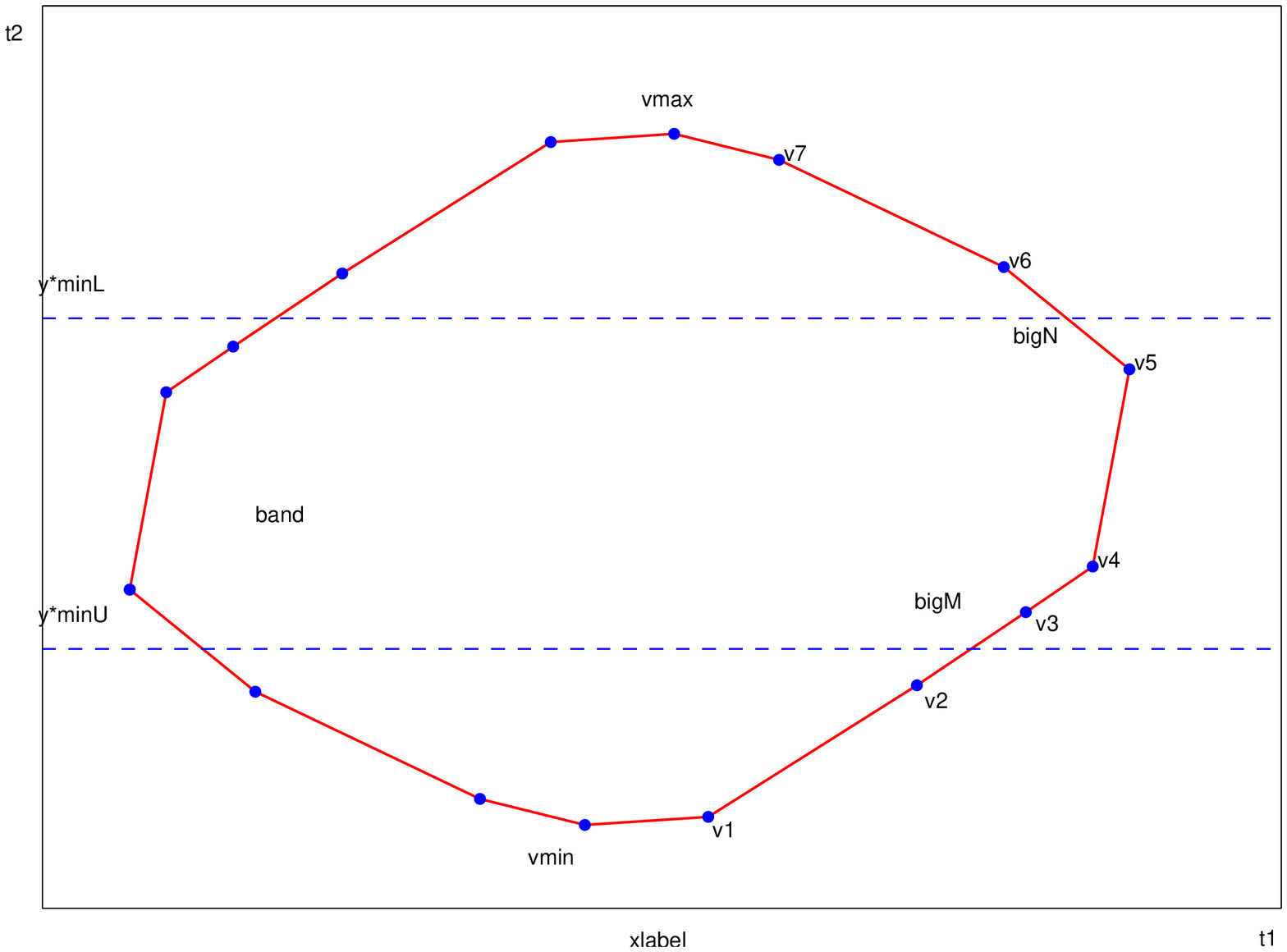}
    \includegraphics[scale=0.37]{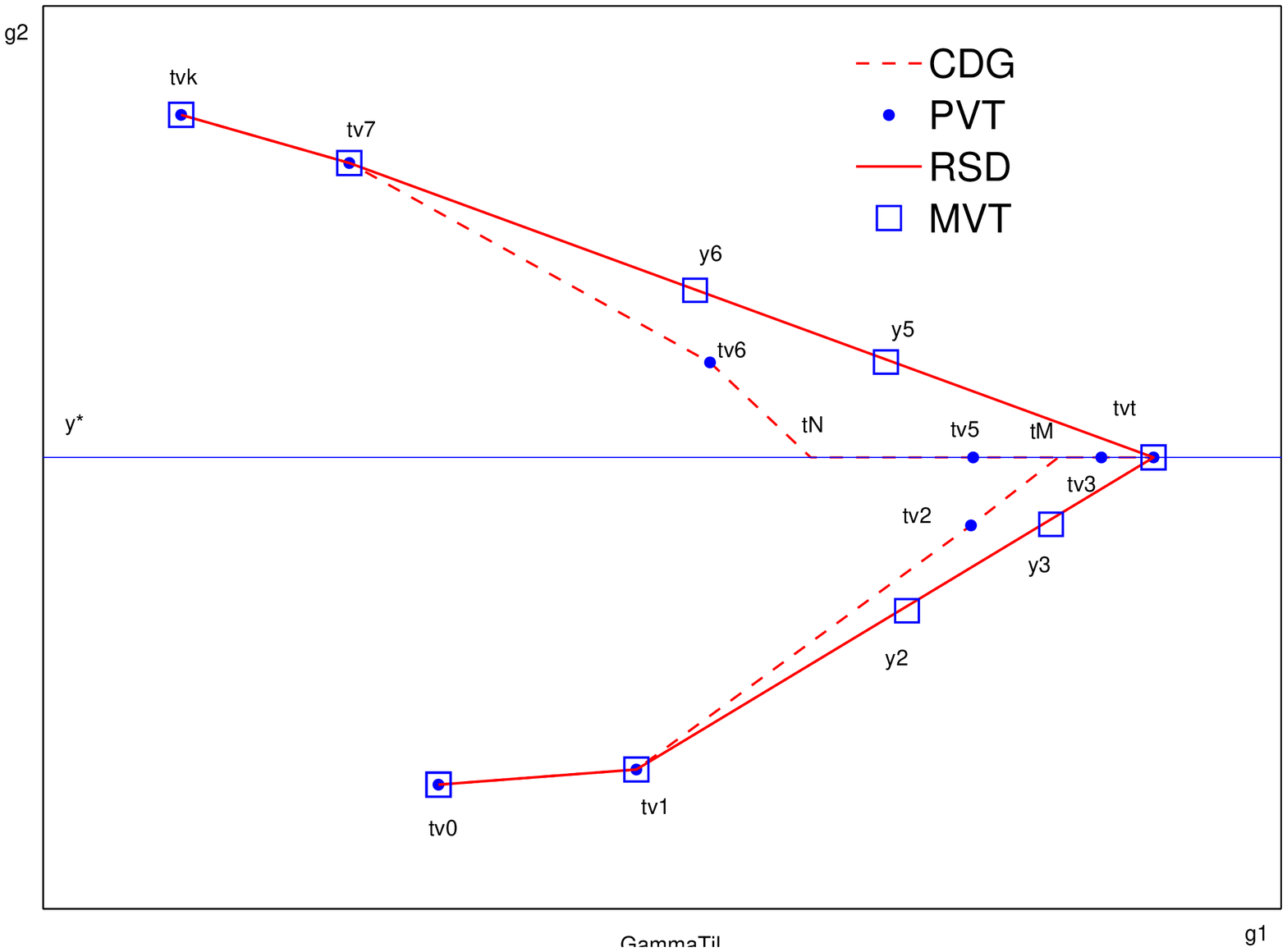}
    \caption{\label{fig:matching_affine_control} Outcomes from the matching and alignment performed in Algorithm 1.}
  \end{psfrags}
\end{figure}

We remark that the above algorithm does not explicitly require 
that the control $q(\mb{w})$ be robustly feasible, i.e. second condition 
in (\ref{eq:aff_problem}). However, we will soon 
prove that this is a direct result of the way matching and alignment 
are performed in Algorithm 1. 

\subsubsection{Affine Controller Preserves Overall Objective and Is Robust.}
\label{sec:affine_controller_preserves_objective}
In this section, we prove that the affine control law $q(\mb{w})$ 
produced by Algorithm 1 
satisfies the requirements of (\ref{eq:aff_problem}), i.e. it is robustly feasible, 
and it preserves the overall objective function $J_{mM}$ 
when used in conjunction with the original convex state costs, $h(\cdot)$. 
With the exception of Corollary \ref{corol:max_on_zonogon_hull}, all 
the key results that we will be using are contained in Section 
\ref{sec:further_underst_induc_hypo} (Lemmas \ref{lem:maximum_in_OPT} and  
\ref{lem:rside_delta_gamma}). Therefore, we will preserve the 
same notation and case discussion as initially introduced there.

First consider the condition on line 1 of Algorithm 1, and 
note that this corresponds to the three 
trivial cases \textbf{[C1]}, \textbf{[C2]} and \textbf{[C3]} of Section 
\ref{sec:further_underst_induc_hypo}. In particular, since $\theta_2 \equiv x_{k+1}$, 
we can use (\ref{eq:DP:uk_star}) to conclude that in these cases, the optimal control law $u^*(\cdot)$ is 
actually affine:
\begin{itemize}
\item[\textbf{[C1]}]  If $\Theta$ falls below the band $\B$, then the upper bound 
  constraint on the control at time $k$ is always active, i.e. 
  the optimal control is $u^*(\theta_2(\mb{w}))= U, ~\forall\, 
  \mb{w} \in \H_k$. 
\item[\textbf{[C2]}] If $\Theta \subseteq \B$, then the constraints on the 
  control at time $k$ are never active, i.e. 
  $u^*(\theta_2(\mb{w}))=y^*-\theta_2(\mb{w})$, hence affine in $\mb{w}$, 
  since $\theta_2$ is affine in $\mb{w}$, by (\ref{eq:simple_notation_theta12}).
\item[\textbf{[C3]}]  If $\Theta$ falls above the band $\B$, then the lower bound 
  constraint on the control is always active, i.e. 
  $u^*(\theta_2(\mb{w}))= L, ~\forall\, \mb{w} \in \H_k$. 
\end{itemize}
Therefore, with the assignment in line 2 of Algorithm 1, 
we obtain an affine control law that is always feasible and also optimal.

When none of the trivial cases holds, we are in case \textbf{[C4]} of 
Section \ref{sec:further_underst_induc_hypo}. Therefore, we can invoke the 
results from Lemma \ref{lem:rside_delta_gamma} to argue that the 
right side of the set $\Delta_\Gamma$ is exactly 
the set on line 7 of the algorithm, $\rside(\Delta_\Gamma) 
= \{\tilde{\mb{v}}_{0},\dots,\tilde{\mb{v}}_{s}\} \cup 
\{\tilde{\mb{v}}_{t}\} \cup \{\tilde{\mb{v}}_{r},\dots,\tilde{\mb{v}}_{k}\}$.
In this setting, we can now formulate the first claim about system 
(\ref{eq:system_for_q_coefficients}) and its solution:

\begin{lemma}
  \label{lem:affine_control_coefficients_robust}
  System (\ref{eq:system_for_q_coefficients}) is always feasible, and the 
  solution satisfies:
  \begin{enumerate}
  \item $-b_i \leq q_i \leq 0, ~\forall \, i \in \{1,\dots,k\}$.
  \item $ L \leq q(\mb{w}) \leq U, ~\forall \, \mb{w} \in \H_k$.
  \end{enumerate}
\end{lemma}
\begin{proof}
  Note first that system (\ref{eq:system_for_q_coefficients}) has exactly 
  $k+3$ unknowns, two for the cotangents $K_U,K_L$, and one for 
  each coefficient $q_i, 0 \leq i \leq k$. Also, since $\left| 
    \rside(\Delta_\Gamma) \right| \leq 
  \left| \ext(\Delta_\Gamma) \right| \leq k+1$, and there are exactly 
  $\left| \rside(\Delta_\Gamma) \right|$ 
  matching constraints, and $k+3-\left| \rside(\Delta_\Gamma) \right|$ 
  alignment constraints, it can be immediately 
  seen that the system is always feasible. 
  
  Consider any $q_i$ with $i \in \{1,\dots,s\} \cup \{r+1,\dots,k\}$. 
  From the matching conditions, we have that 
  $q_i = u^*(\mb{v}_i) - u^*(\mb{v}_{i-1})$.
  By property \textbf{P3} from Section \ref{sec:DP_formulation},  the difference 
  in the values of the optimal control law $u^*(\cdot)$ satisfies:
  \begin{align*}
    u^*(\mb{v}_i) - u^*(\mb{v}_{i-1}) &\bydef u^*(\theta_2[\mb{v}_i]) 
    - u^*(\theta_2[\mb{v}_{i-1}]) \\
    (\text{by \textbf{P3}})~& \,= - f \cdot \left( \theta_2[\mb{v}_i] - 
      \theta_2[\mb{v}_{i-1}] \right) \\
    &\overset{\tiny \eqref{eq:simple_notation_theta12}}{=} - f \cdot b_i, 
    \qquad \text{where}~ f \in [0,1].
  \end{align*}
  Since, by (\ref{eq:conditions_generators_theta}), $b_j \geq 0, \, \forall \, j \in \{1,\dots,k\}$, 
  we immediately obtain $-b_i \leq q_i \leq 0$, for $i \in \{1,\dots,s\} \cup 
  \{r+1,\dots,k\}$.
  
  Now consider any index $i \in \{s+1,\dots, t \wedge r \}$, where 
  $t \wedge r \equiv \min(t,r)$. From 
  the conditions in system (\ref{eq:system_for_q_coefficients}) for alignment below $t$, 
  we have $q_i = \frac{a_i- K_U \cdot b_i}{K_U-c}$.
  By summing up all such relations, we obtain:
  \begin{align*}
    \sum_{i=s+1}^{t \wedge r} q_i &= \frac{\sum_{i=s+1}^{t \wedge r} a_i- K_U \cdot 
      \sum_{i=s+1}^{t \wedge r} b_i}{K_U-c} \quad \Leftrightarrow 
    \qquad (\text{using the matching}) \\
    u^*(\mb{v}_{t \wedge r}) - u^*(\mb{v}_{s})&=
    \frac{\sum_{i=s+1}^{t \wedge r} a_i- K_U \cdot 
      \sum_{i=s+1}^{t \wedge r} b_i}{K_U-c} \quad \Leftrightarrow \\
    K_U &= \frac{\sum_{i=s+1}^{t \wedge r} a_i + c \cdot 
      \left( u^*(\mb{v}_{t \wedge r}) - u^*(\mb{v}_{s}) \right)}
    {\sum_{i=s+1}^{t \wedge r} b_i + u^*(\mb{v}_{t \wedge r}) - u^*(\mb{v}_{s})} \\
    &= \frac{\left[ \sum_{i=0}^{t \wedge r} a_i + c \cdot 
        u^*(\mb{v}_{t \wedge r}) \right] - 
      \left[ \sum_{i=0}^{s} a_i + c \cdot u^*(\mb{v}_{s}) \right]}
    {\left[ \sum_{i=0}^{t \wedge r} b_i + u^*(\mb{v}_{t \wedge r}) \right] - 
      \left[ \sum_{i=0}^{s} b_i + u^*(\mb{v}_{s}) \right]} \\
    &\overset{\tiny \eqref{eq:vi_tilde_definition}}{=} 
    \frac{\tilde{\gamma}_1[\tilde{\mb{v}}_{t \wedge r}] - 
      \tilde{\gamma}_1[\tilde{\mb{v}}_{s}]}{\tilde{\gamma}_2[\tilde{\mb{v}}_{t \wedge r}] - 
      \tilde{\gamma}_2[\tilde{\mb{v}}_{s}]} 
    \overset{\tiny \eqref{eq:cotan_definition}}{=} 
    \cotan{\tilde{\mb{v}}_{s}}{\tilde{\mb{v}}_{t \wedge r}}.
  \end{align*}
  In the first step, we have used the fact that both $\tilde{\mb{v}}_{s}$
  and $\tilde{\mb{v}}_{\min(t,r)}$ are matched, hence the intermediate coefficients $q_i$ 
  must sum to exactly the difference of the values of $u^*(\cdot)$ 
  at $\mb{v}_{\min(t,r)}$ and $\mb{v}_{s}$ respectively. In this context, we can 
  see that $K_U$ is simply the cotangent of the angle formed 
  by the segment $[\tilde{\mb{v}}_{s},\tilde{\mb{v}}_{\min(t,r)}]$ with the 
  horizontal (i.e. $\tilde{\gamma}_1$) axis. In this case, we can immediately
  recall result (\ref{eq:properties_KL_KU}) from Lemma \ref{lem:rside_delta_gamma},
  to argue that $K_U \geq \frac{a_{s+1}}{b_{s+1}}$. Combining with (\ref{eq:t_index_definition})
  and (\ref{eq:conditions_generators_theta}), we obtain:
  \begin{align*}
    K_U \geq \frac{a_{s+1}}{b_{s+1}} 
    \overset{\tiny \eqref{eq:conditions_generators_theta}}{\geq \dots \geq} 
    \frac{a_{\min(t,r)}}{b_{\min(t,r)}} \geq \frac{a_t}{b_t} 
    \overset{\tiny \eqref{eq:t_index_definition}}{>} c.
  \end{align*}
  Therefore, we immediately have that for any $i \in \{s+1,\dots,\min(t,r)\}$,
  \begin{align*}
    \left\{
      \begin{aligned}
        a_i - K_U \cdot b_i &\leq 0 \\
        K_U - c &> 0
      \end{aligned} \right.
    & ~ \Rightarrow ~ q_i = \frac{a_i- K_U \cdot b_i}{K_U-c} \leq 0 ~, & \quad 
    \left\{
      \begin{aligned}
        a_i - c \cdot b_i &> 0 \\
        q_i + b_i &= \frac{a_i - c \cdot b_i}{K_U - c}
      \end{aligned} \right.
    & ~\Rightarrow  ~ q_i + b_i \geq 0.
  \end{align*}
  The argument for indices $i \in \{\max(t,s)+1,\dots,r\}$ proceeds in exactly 
  the same fashion, by recognizing that $K_L$ defined in the algorithm
  is the same as $\cotan{\tilde{\mb{v}}_{\max(t,s)}}{\tilde{\mb{v}}_{r}}$,
  and then applying (\ref{eq:properties_KL_KU}) to argue that 
  $ K_L < \frac{a_{r}}{b_{r}} \leq 
  \frac{a_{\max(t,s)+1}}{b_{\max(t,s)+1}} \leq \frac{a_{t+1}}{b_{t+1}} \leq c$.
  This will allow us to use the same reasoning as above, completing the proof 
  of part $(i)$ of the claim.

  To prove part $(ii)$, consider any $\mb{w} \in \H_k \bydef [0,1]^k$. Using part $(i)$,
  we obtain:
  \begin{align*}
    q(\mb{w}) &\bydef q_0 + \sum_{i=1}^k q_i \cdot w_i ~ \leq ~
    (\text{since $w_i \in [0,1], q_i \leq 0$}) ~
    \leq q_0 \overset{\tiny (**)}{=} u^*(\mb{v}_0) \leq U ~ ,\\
    q(\mb{w}) &\geq q_0 + \sum_{i=1}^k q_i \cdot 1 \overset{\tiny (**)}{=}
    u^*(\mb{v}_k) \geq L.
  \end{align*}
  Note that in step $(**)$, we have critically used the result from Lemma 
  \ref{lem:rside_delta_gamma} that, when $\Theta \nsubseteq \B$,
  the points $\tilde{\mb{v}}_0,\tilde{\mb{v}}_k$ are always among the 
  points on the right side of $\Delta_\Gamma$, and, therefore, we will always 
  have the equations $q_0 = u^*(\mb{v}_0),  q_0 + \sum_{i=1}^k q_i = u^*(\mb{v}_k)$ 
  among the matching equations of system 
  (\ref{eq:system_for_q_coefficients}). For the last arguments, we have simply 
  used the fact that the optimal control law, $u^*(\cdot)$, is always 
  feasible, hence $L \leq u^*(\cdot) \leq U$. 
\end{proof}

This completes our first goal, namely proving that the affine controller $q(\mb{w})$ is 
always robustly feasible. To complete the construction, we introduce the following 
final result:
\begin{lemma}
  \label{lem:affine_control_preserves_objective}
  Using the affine control law $q(\mb{w})$ computed in Algorithm 1 
  satisfies the first equation in (\ref{eq:aff_problem}).
\end{lemma}
\begin{proof}
  From (\ref{eq:set_for_gamma1_gamma2}), the affine controller $q(\mb{w})$ 
  induces the generators $\mb{a} + c \cdot \mb{q}$ and $\mb{b} + \mb{q}$
  for the zonogon $\Gamma$. This implies that $\Gamma$ will be the 
  Minkowski sum of the following segments in $\R^2$:
  \begin{align}
    \left[\begin{array}{c}
        a_1+ c \cdot q_1 \\
        b_1+q_1 
      \end{array} \right], \dots,
    \left[\begin{array}{c}
        a_{s}+ c \cdot q_{s}  \\
        b_{s}+ q_{s}
      \end{array} \right], 
    \left[\begin{array}{c}
        K_U\cdot \left(b_{s+1}+ q_{s+1}\right)  \\
        b_{s+1}+ q_{s+1}
      \end{array} \right], \dots,
    \left[\begin{array}{c}
        K_U \cdot \left(b_{\min(t,r)}+ q_{\min(t,r)}\right)  \\
        b_{\min(t,r)}+ q_{\min(t,r)}
      \end{array} \right], \nonumber \\  
    \left[\begin{array}{c}
        K_L \cdot \left(b_{\max(t,s)+1}+ q_{\max(t,s)+1}\right)  \\
        b_{\max(t,s)+1}+ q_{\max(t,s)+1}
      \end{array} \right], \dots,
    \left[\begin{array}{c}
        K_L \cdot \left(b_{r}+ q_{r}\right)  \\
        b_{r}+ q_{r}
      \end{array} \right],
    \left[\begin{array}{c}
        a_{r+1}+ c \cdot q_{r+1} \\
        b_{r+1}+q_{r+1} 
      \end{array} \right], \dots,
    \left[\begin{array}{c}
        a_{k}+ c \cdot q_{k}  \\
        b_{k}+ q_{k}
      \end{array} \right].
    \label{eq:final_zonotope_generators}
  \end{align}
  From Lemma \ref{lem:affine_control_coefficients_robust}, we have that 
  $q_i+b_i \geq 0, \forall \, i \in \{1,\dots,k\}$. Therefore,
  if we consider the points in $\R^2$:
  \begin{align*}
    \mb{y}_i &= \left(~ \sum_{j=0}^i (a_j+ c \cdot q_j),~~
      \sum_{j=0}^i (b_j+ q_j) ~\right), \quad \forall \, i \in \{0,\dots,k\},
  \end{align*}
  we can make the following simple observations:
  \begin{itemize}
  \item For any vertex $\mb{v}_i \in \Theta, \, i \in \{0,\dots,k\}$, 
    that is matched, i.e.
    $\tilde{\mb{v}}_i \in \rside(\Delta_\Gamma)$, if we let $\mb{w}_i$ represent the 
    unique\footnote{This vertex is unique due to our standing Assumption 
      \ref{as:zonogon_max_vertices} that the number 
      of vertices in $\Theta$ is $2k$ (also see part (iv) of Lemma \ref{lem:zonotope_properties}
      in the Appendix).} vertex of the hypercube $\H_k$ projecting onto $\mb{v}_i$,
    i.e. $\mb{v}_i = \left(\theta_1(\mb{w}_i), \theta_2(\mb{w}_i) \right)$, then 
    we have:
    \begin{align*}
      \mb{y}_i \overset{\tiny \eqref{eq:set_for_gamma1_gamma2}}{=} 
      \left(\gamma_1(\mb{w}_i),\gamma_2(\mb{w}_i)\right)
      \overset{\tiny \eqref{eq:system_for_q_coefficients}}{=}
      \left(\tilde{\gamma}_1(\mb{v}_i),  \tilde{\gamma}_2(\mb{v}_i)\right)
      \overset{\tiny \eqref{eq:vi_tilde_definition}}{=}
      \tilde{\mb{v}}_i.
    \end{align*}
    The first equality follows from the definition of the mapping 
    that characterizes the zonogon $\Gamma$. The second equality 
    follows from the fact that for any matched vertex $\mb{v}_i$, 
    the coordinates in $\Gamma$ and $\tilde{\Gamma}$ are exactly the 
    same, and the last equality is simply the definition of 
    the point $\tilde{\mb{v}}_i$.
  \item For any vertex $\mb{v}_i \in \Theta, \, i \in \{0,\dots,k\}$,
    that is \emph{not} matched, we have:
    \begin{align*}
      \mb{y}_i &\in [\mb{y}_{s},\mb{y}_{\min(t,r)}],
      \quad \forall\, i \in \{s+1,\dots,\min(t,r)-1\} \\
      \mb{y}_i &\in [\mb{y}_{\max(t,s)},\mb{y}_{r}],
      \quad \forall\, i \in \{\max(t,s)+1,\dots,r-1\}.
    \end{align*}
    This can be seen directly from
    (\ref{eq:final_zonotope_generators}), since the segments in $\R^2$
    given by $[\mb{y}_{s},\mb{y}_{s+1}],
    \dots,[\mb{y}_{\min(t,r)-1},\mb{y}_{\min(t,r)}]$ are always
    aligned (with common cotangent, given by $K_U$), and, similarly,
    the segments $[\mb{y}_{\max(t,s)},\mb{y}_{\max(t,s)+1}]$, $\ldots,
    [\mb{y}_{r-1},\mb{y}_{r}]$ are also aligned (with common cotangent
    $K_L$).
  \end{itemize}

  This exactly corresponds to the situation shown earlier in Figure \ref{fig:matching_affine_control}. 
  By combining the two observations, it can be seen that the points 
  $\left\{\mb{y}_{0},\mb{y}_{1},\dots,\mb{y}_{s},\mb{y}_{\max(t,s)}, 
    \mb{y}_{\min(t,r)},\mb{y}_{r},\dots,\mb{y}_{k}\right\}$ will satisfy the following properties:
  \begin{align*}
    \mb{y}_i &= \tilde{\mb{v}}_i, ~\forall \, \tilde{\mb{v}}_i \in 
    \rside(\Delta_\Gamma) ~, \\
    \cotan{\mb{y}_0}{\mb{y}_1} &\geq \cotan{\mb{y}_1}{\mb{y}_2}
    \geq \dots \geq \cotan{\mb{y}_{s-1}}{\mb{y}_{s}} \geq 
    \cotan{\mb{y}_{s}}{\mb{y}_{\min(t,r)}}  \geq \nonumber \\
    & \geq  \cotan{\mb{y}_{\max(t,s)}}{\mb{y}_{r}} \geq 
    \cotan{\mb{y}_{r}}{\mb{y}_{r+1}} \geq \dots \geq
    \cotan{\mb{y}_{k-1}}{\mb{y}_{k}},
  \end{align*}
  where the second relation follows simply because the points 
  $\tilde{\mb{v}}_i \in \rside(\Delta_\Gamma)$ are extreme points on the right side 
  of a convex hull, and thus satisfy the same string of inequalities.
  This immediately implies that this set of $\mb{y}_i$ (corresponding to 
  $\tilde{\mb{v}}_i \in \rside(\Delta_\Gamma)$)
  exactly represent the right side of the zonogon $\Gamma$, 
  which, in turn, implies that $\Gamma \equiv \zhull\left( 
    \left\{ \tilde{\mb{v}}_0, \tilde{\mb{v}}_1, \dots, \tilde{\mb{v}}_{s},
      \tilde{\mb{v}}_{\max(t,s)}, \tilde{\mb{v}}_{\min(t,r)},
      \tilde{\mb{v}}_{r}, \tilde{\mb{v}}_{r+1}, \dots, \tilde{\mb{v}}_{k} 
    \right\} \right)$.
  But then, by Corollary \ref{corol:max_on_zonogon_hull}, the maximum value 
  of problem $(OPT)$ in (\ref{eq:set_for_gamma1tilde_gamma2tilde}) is equal to 
  the maximum value of problem $(AFF)$ 
  in (\ref{eq:aff_problem}), and, since the former is always $J_{mM}$,
  so is that latter.
\end{proof}

This concludes the construction of the affine control law $q(\mb{w})$. 
We have shown that the policy computed by Algorithm 1 
satisfies all the conditions in (\ref{eq:aff_problem}), i.e. is robustly feasible
(by Lemma \ref{lem:affine_control_coefficients_robust}) and, 
when used in conjunction with the original convex state costs, 
preserves the overall optimal min-max value $J_{mM}$ 
(Lemma \ref{lem:affine_control_preserves_objective}).

\subsection{Construction of the Affine State Cost.}
\label{sec:construction_affine_stage_cost}
Note that we have essentially completed the first part of the induction step. 
For the second part, we would still need to show how an affine 
stage cost can be computed, such that constraints \eqref{eq:constr:affine_run_cost}
and \eqref{eq:constr:same_objective} are satisfied. We will return temporarily 
to the notation containing time indices, so as to put the current state of the 
proof into perspective.

In solving problem $(AFF)$ of (\ref{eq:aff_problem}), we have shown that there exists 
an affine $q_{k+1}(\mb{w}^{k+1})$ such that:
\begin{align*}
  J_{mM} &~= \underset{\mb{w}^{k+1} \in \H_k}{\max} 
  \left[ \theta_1(\mb{w}^{k+1}) + c_{k+1} \cdot q_{k+1}(\mb{w}^{k+1}) + 
    g_{k+1}\left(\theta_2(\mb{w}^{k+1})+
      q_{k+1}(\mb{w}^{k+1})\right) \right] \\
  & \overset{\tiny \eqref{eq:set_for_gamma1_gamma2}}{=} 
  \underset{\mb{w}^{k+1} \in \H_k}{\max} \left[ \gamma_1(\mb{w}^{k+1})
  + g_{k+1}\left(\gamma_2(\mb{w}^{k+1})\right) \right].
\end{align*}
Using the definition of $g_{k+1}(\cdot)$ from \eqref{eq:DP:gk_definition}, 
we can write the above (only retaining the second term) as:
\begin{align*}
  J_{mM} &\,= \underset{\mb{w}^{k+1} \in \H_k}{\max} 
  \left[\, \gamma_1(\mb{w}^{k+1}) + 
    \max_{w_{k+1} \in \W_{k+1}} \left[\, h_{k+2}(\gamma_2(\mb{w}^{k+1}) + w_{k+2})
      + J^*_{k+2}(\gamma_2(\mb{w}^{k+1}) + w_{k+2}) \right] ~\right] \\
  &\bydef \underset{\mb{w}^{k+2} \in \H_{k+1}}{\max} 
  \left[~ \pi_1(\mb{w}^{k+2}) + 
    h_{k+2}\left(\pi_2(\mb{w}^{k+2})\right)
    + J^*_{k+2}\left(\pi_2(\mb{w}^{k+2})\right) ~\right],
\end{align*}
where $\pi_1(\mb{w}^{k+2}) \bydef \gamma_1(\mb{w}^{k+1})$,
and $\pi_2(\mb{w}^{k+2}) \bydef \gamma_2(\mb{w}^{k+1}) + w_{k+2}$.
Is is easy to note that:
\begin{align}
  \Pi \bydef \left( \pi_1(\mb{w}^{k+2}),
    \pi_2(\mb{w}^{k+2}) \right)    
  \label{eq:definition_zonotope_affine_cost}
\end{align}
represents yet another zonogon, obtained by projecting a hyper-rectangle 
$\H_{k+1} \subset \R^{k+1}$ into $\R^2$. It has a particular shape 
relative to the zonogon 
$\Gamma = (\gamma_1,\gamma_2)$, since the generators of $\Pi$ are 
simply obtained by appending a $0$ and a $1$, respectively, to the generators 
of $\Gamma$, which implies that $\Pi$ is the convex 
hull of two translated copies of $\Gamma$, where the translation is occuring 
on the $\pi_2$ axis. As it turns out, this fact will bear little importance 
for the discussion to follow, so we include it here only for completeness.

In this context, the problem we would like to solve is to replace the 
convex function 
$h_{k+2}(\pi_2(\mb{w}^{k+2}))$ with an affine function 
$z_{k+2}(\mb{w}^{k+2})$, such that the analogues of conditions 
\eqref{eq:constr:affine_run_cost} and \eqref{eq:constr:same_objective} are obeyed:
\begin{align*}
  z_{k+2}(\mb{w}^{k+2}) &\geq h_{k+2}(\pi_2(\mb{w}^{k+2})),
  \qquad \forall \, \mb{w}^{k+2} \in \H_{k+1} \\
  J_{mM} &= \underset{\mb{w}^{k+2} \in \H_{k+1}}{\max} 
  \left[~ \pi_1(\mb{w}^{k+2}) + 
    z_{k+2}(\mb{w}^{k+2}) + J^*_{k+2}(\pi_2(\mb{w}^{k+2})) ~\right].
\end{align*}
We can now switch back to the simplified notation, where all 
the time subscripts or superscripts are removed. Furthermore, to preserve 
as much of the familiar notation from Section \ref{sec:simplified-notation}, 
we will denote the generators of zonogon $\pi$ by 
$\mb{a},\mb{b}$, so that we have:
\begin{align}
  \pi_1(\mb{w}) &= a_0 + \sum_{i=1}^{k+1} a_i \cdot w_i \quad ,&
  \pi_2(\mb{w}) &= b_0 + \sum_{i=1}^{k+1} b_i \cdot w_i
  \label{eq:pi1_pi2_definition}
\end{align}
In perfect analogy to our discussion in Section \ref{sec:induct_hypo}, we can introduce:
\begin{align}
  & \mb{v}^- \bydef \underset{\pi_1}{\argmax} \left\{ 
    \underset{\pi_2}{\argmin} \left\{
      \mb{\pi} \in \Pi \right\} \right\}; \qquad 
  \mb{v}^+ \bydef 2{\mb{O}} - \mb{v}^- 
  \qquad \left( \mb{O} ~\text{is the center of}~ \Pi \right)
  \label{eq:vertex_Pi_min_max} \\
  & \mb{v}_0 \bydef \mb{v}^{-} , \dots, \mb{v}_{p_1} 
  \bydef \mb{v}^+, \dots , \mb{v}_{2p_1} = \mb{v}^{-} 
  \qquad \left(\text{counter-clockwise numbering of the vertices of $\Pi$}\right).
  \nonumber
\end{align}
Without loss of generality, we will, again, work under Assumptions \ref{as:hypercube},
 \ref{as:zonogon_max_vertices}, 
and \ref{as:consecutive_vertices_zeros}, i.e. we will analyze the case when 
$\H_{k+1} = [0,1]^{k+1}$, $p_1=k+1$ (the zonogon $\Pi$ has a maximal number of vertices), 
and $\mb{v}_i=[1,1,\dots,1,0,\dots,0]$ (ones in the first $i$ positions). 
Furthermore, we will again use $\pi_{1,2}(\mb{w})$ for $\mb{w} \in \H_{k+1}$ to denote the mapping 
from $\H_{k+1} \mapsto \Pi \subset \R^2$, and $\pi_{1,2}[\mb{v}]$ 
to denote the coordinates of the point $\mb{v} \in \R^2$, and will use 
the shorthand notations $h(\mb{v}_i), J^*(\mb{v}_i)$ instead of 
$h(\pi_2[\mb{v}_i])$ and $J^*(\pi_2[\mb{v}_i])$, respectively.

With the simplified notation, the goal is to find $z(\mb{w})$ such that:
\begin{align}
  z(\mb{w}) &\geq h(\pi_2(\mb{w})),
  \qquad \forall \, \mb{w} \in \H_{k+1} \label{eq:simple_notation_robust_h_cost} \\
  \max_{(\pi_1,\pi_2) \in \Pi} \left[\, \pi_1 + 
    h(\pi_2) + J^*(\pi_2) \,\right]& = 
  \max_{\mb{w} \in \H_{k+1}} \left[\, \pi_1(\mb{w}) + 
    z(\mb{w}) + J^*\left(\pi_2(\mb{w})\right) \,\right]
  \label{eq:simple_notation_max_problem_h_cost}
\end{align}
In (\ref{eq:simple_notation_max_problem_h_cost}), the first maximization problem 
corresponds to the problem solved by the uncertainties, $\mb{w}$, when the original 
convex state cost, $h(\pi_2)$, is incurred. As such, the result of the 
maximization is always exactly equal to $J_{mM}$, the overall min-max value. The 
second maximization corresponds to the problem solved by the uncertainties 
when the affine cost, $z(\mb{w})$, is incurred instead of the convex cost.
Requiring that the two optimal values be equal thus amounts to preserving the overall min-max value,
$J_{mM}$, under the affine cost.

Since $h$ and $J^*$ are convex (see property 
\textbf{P2} in Section \ref{sec:DP_formulation}), we can immediately
use Lemma \ref{lem:theta_restricted_region} to conclude that 
the optimal value in the first maximization problem in (\ref{eq:simple_notation_max_problem_h_cost}) 
is reached at one of the vertices $\mb{v}_0,\dots,\mb{v}_{k+1}$ found in $\rside(\Pi)$.
Therefore, by introducing the points:
\begin{align}
  \tilde{\mb{v}}_i \bydef \left( \,
    \pi_1[\mb{v}_i] + h(\mb{v}_i), ~ \pi_2[\mb{v}_i] \, \right), \, 
  \forall\, i \in \{0,\dots,k+1\},
  \label{eq:vtilde_definition_for_affine_cost}
\end{align}
we can immediately conclude the following result:
\begin{lemma}
  \label{lem:maximum_in_OPT_affine_cost}
  The maximum in problem:
  \begin{equation}
    \begin{aligned}
      &(OPT) \quad  \underset{\left(\tpi_1,\tpi_2\right)}{\max} 
      \left[~ \tpi_1 + J^*\left(\tpi_2\right)~\right] ~, \\
      \text{where} \quad & \tpi_1 \bydef \pi_1 + h(\pi_2), \quad
      \tpi_2 \bydef \pi_2, \quad
      (\pi_1,\pi_2) \in \Pi,
    \end{aligned}
    \label{eq:OPT_problem_for_affine_cost}
  \end{equation}
  is reached on the right side of:
  \begin{align}
    \Delta_\Pi \bydef \conv \left( \left\{\tilde{\mb{v}}_0,
      \dots,\tilde{\mb{v}}_{k+1}\right\} \right).
    \label{eq:delta_pi_convex_hull_defn_affine_cost}
  \end{align}
\end{lemma}
\begin{proof}
  The result is analogous to Lemma \ref{lem:maximum_in_OPT}, and the proof 
  is a rehashing of similar ideas. In particular,
  first note that problem $(OPT)$ is a rewriting of 
  the first maximization problem in (\ref{eq:simple_notation_max_problem_h_cost}). 
  Therefore, since the maximum of the latter problem 
  is reached at the vertices $\mb{v}_i, \, i \in \{0,\dots,k+1\},$ of 
  zonogon $\Pi$, by using \eqref{eq:vtilde_definition_for_affine_cost}
  we can conclude that the maximum in problem $(OPT)$ must be reached 
  on the set $\{\tilde{\mb{v}}_0,\dots,\tilde{\mb{v}}_{k+1}\}$.
  Noting that the function maximized in $(OPT)$ is convex, this  
  set of points can be replaced with its convex hull, $\Delta_\Pi$, without affecting the 
  result \citep[see Section 32 of][]{Rock70}. Furthermore, since the function 
  maximized is of the form $\xi_1 + f(\xi_2)$, with $f$ convex, by 
  applying the results in Corollary \ref{corol:max_on_zonogon_hull},
  and replacing the set by the right-side of its convex hull, $\rside(\Delta_\Pi)$,
  the result of the optimization would remain unchanged.
\end{proof}

Continuing the analogy with the construction in Section \ref{sec:constr-affine-contr}, 
we rewrite the second optimization in 
\eqref{eq:simple_notation_max_problem_h_cost} as:
\begin{equation}
  \begin{aligned}
    &(AFF) \quad \underset{\left(\hpi_1,\hpi_2\right) \in \hat{\Pi}}{\max} 
    \left[~ \hpi_1 + J^*\left(\hpi_2\right)~\right] ~,\\
    \text{where} \quad &\hat{\Pi} \bydef \left\{~ (\hpi_1,\hpi_2) ~:~ 
    \hpi_1(\mb{w}) \bydef \pi_1(\mb{w}) + z(\mb{w}), ~~
    \hpi_2(\mb{w}) \bydef \pi_2(\mb{w}), ~ \mb{w} \in \H_{k+1} ~\right\}.
  \end{aligned}
  \label{eq:AFF_problem_for_affine_cost}
\end{equation}
In order to examine the maximum in problem ($AFF$), we remark 
that its feasible set, $\hat{\Pi} \subset \R^2$,
also represents a zonogon, with generators given by 
$\mb{a} + \mb{z}$ and $\mb{b}$, respectively. Therefore, by 
Lemma \ref{lem:theta_restricted_region}, the maximum of problem ($AFF$) 
is reached at one of the vertices on the right side of $\hat{\Pi}$. 

Using the same key idea from the construction of the affine control law,
we now argue that, if the coefficients of the affine cost, $z_i$, were computed 
in such a way that $\hat{\Pi}$ represented the \emph{zonogon hull} of the set 
of points $\left\{\tilde{\mb{v}}_0,\dots,\tilde{\mb{v}}_{k+1}\right\}$, then 
(by Corollary \ref{corol:max_on_zonogon_hull}), the maximum value 
of problem $(AFF)$ would be the same as the maximum value of problem $(OPT)$.

To this end, we introduce the following procedure for computing the affine 
cost $z(\mb{w})$:
\begin{algorithm}[H]
  \caption{Compute affine stage cost $z(\mb{w})$}
  \label{alg2:find_z_cost}
  \begin{algorithmic}[1]
    \REQUIRE $\pi_1(\mb{w}), \pi_2(\mb{w}), h(\cdot), J^*(\cdot)$.
    \STATE Apply the mapping \eqref{eq:vtilde_definition_for_affine_cost} to obtain 
    $\tilde{\mb{v}}_i, \, \forall\, i \in \{0,\dots,k+1\}$.
    \STATE Compute the set $\Delta_\Pi = \conv\left(\{\tilde{\mb{v}}_0,\dots,
      \tilde{\mb{v}}_{k+1}\}\right)$.
    \STATE Let $\rside(\Delta_\Pi) \bydef \{\tilde{\mb{v}}_{s(1)},\dots,\tilde{\mb{v}}_{s(n)}\}$,
    where $s(1) \leq s(2) \leq \dots \leq s(n) \in \{0,\dots,k+1\}$ are the sorted indices of 
    points on the right side of $\Delta_\Pi$.
    \STATE Solve the following system for $z_j, \, (j \in \{0,\dots,k+1\})$, and 
    $K_{s(i)}, \, (i \in \{2,\dots,n\})$:
    \begin{equation}
      \left\{
        \begin{aligned}
          z_0 + z_1 + \dots + z_{s(i)} &= h\left(\mb{v}_{s(i)} \right), 
          && \forall\, \tilde{\mb{v}}_{s(i)} \in \rside(\Delta_\Pi) && \text{(matching)}  \\
          \frac{z_j+a_j}{b_j} &= K_{s(i)}, &&
          \forall\, j \in \{s(i-1)+1,\dots,s(i)\},
          ~ \forall \, i \in \{2,\dots,n\},
          && \text{(alignment)} 
        \end{aligned}
      \right.
      \label{eq:system_for_z_coefficients}
    \end{equation} 
    \STATE Return $z(\mb{w}) = z_0 + \sum_{i=1}^{k+1} z_i \cdot w_i$.
  \end{algorithmic}
\end{algorithm}

To visualize how the algorithm is working, an extended example is included 
in Figure \ref{fig:algorithm_2_iterations}.
\begin{figure}[h]
  \centering
  \begin{psfrags}
    \psfrag{vmin}[l][][0.8]{$\mb{v}_0=\mb{v}^-$}
    \psfrag{v1}[l][][0.8]{$\mb{v}_1$}
    \psfrag{v2}[l][][0.8]{$\mb{v}_2$}
    \psfrag{v3}[l][][0.8]{$\mb{v}_3$}
    \psfrag{v4}[l][][0.8]{$\mb{v}_4$}
    \psfrag{v5}[l][][0.8]{$\mb{v}_5$}
    \psfrag{v6}[l][][0.8]{$\mb{v}_6$}
    \psfrag{v7}[l][][0.8]{$\mb{v}_7$}
    \psfrag{vmax}[l][][0.8]{$\mb{v}_{k+1}=\mb{v}^+$}
    \psfrag{bigM}[l][][0.8]{}
    \psfrag{bigN}[l][][0.8]{}
    \psfrag{t1}[][][0.8]{$\pi_1$}
    \psfrag{t2}[][][0.8]{$\pi_2$}
    \psfrag{xlabel}[c][][0.8]{Original zonotope $\Pi$.}
    
    \psfrag{g1}[][][0.8]{$\tilde{\pi}_1$}
    \psfrag{g2}[][][0.8]{$\tilde{\pi}_2$}
    \psfrag{tv0}[l][][0.8]{$\tilde{\mb{v}}_0=\tilde{\mb{v}}_{s(1)} = \mb{y}_{s(1)}$}
    \psfrag{tv1}[l][][0.8]{$\tilde{\mb{v}}_1$}
    \psfrag{tv2}[l][][0.8]{$\tilde{\mb{v}}_2$}
    \psfrag{tv3}[l][][0.8]{$\tilde{\mb{v}}_3$}
    \psfrag{tv4}[l][][0.8]{$\tilde{\mb{v}}_{s(2)} = \mb{y}_{s(2)}$}
    \psfrag{tv5}[l][][0.8]{$\tilde{\mb{v}}_5$}
    \psfrag{tv6}[l][][0.8]{$\tilde{\mb{v}}_6$}
    \psfrag{tv7}[l][][0.8]{$\tilde{\mb{v}}_{s(3)} = \mb{y}_{s(3)}$}
    \psfrag{tv8}[r][][0.8]{$\tilde{\mb{v}}_{k+1} = \tilde{\mb{v}}_{s(n)}=\mb{y}_{s(n)}$}
    \psfrag{TPI}[l][][0.8]{}
    \psfrag{PVT}[l][][0.8]{$\tilde{\mb{v}}_i$}
    \psfrag{RSD}[l][][0.8]{$\hat{\Pi} = \zhull\left(\{\tilde{\mb{v}}_i\} \right)$}
    \psfrag{MVT}[l][][0.8]{$\mb{y}_j \in \rside(\hat{\Pi})$}
    \psfrag{y1}[l][][0.8]{$\mb{y}_1$}
    \psfrag{y2}[l][][0.8]{$\mb{y}_2$}
    \psfrag{y3}[l][][0.8]{$\mb{y}_3$}
    \psfrag{y5}[l][][0.8]{$\mb{y}_5$}
    \psfrag{y6}[l][][0.8]{$\mb{y}_6$}
    \psfrag{PiTil}[c][][0.8]{Points $\tilde{\mb{v}}_i$ and $\mb{y}_i \in \rside(\hat{\Pi})$.}
    \psfrag{y*}[l][][0.8]{$y^*$}
    \includegraphics[scale=0.37]{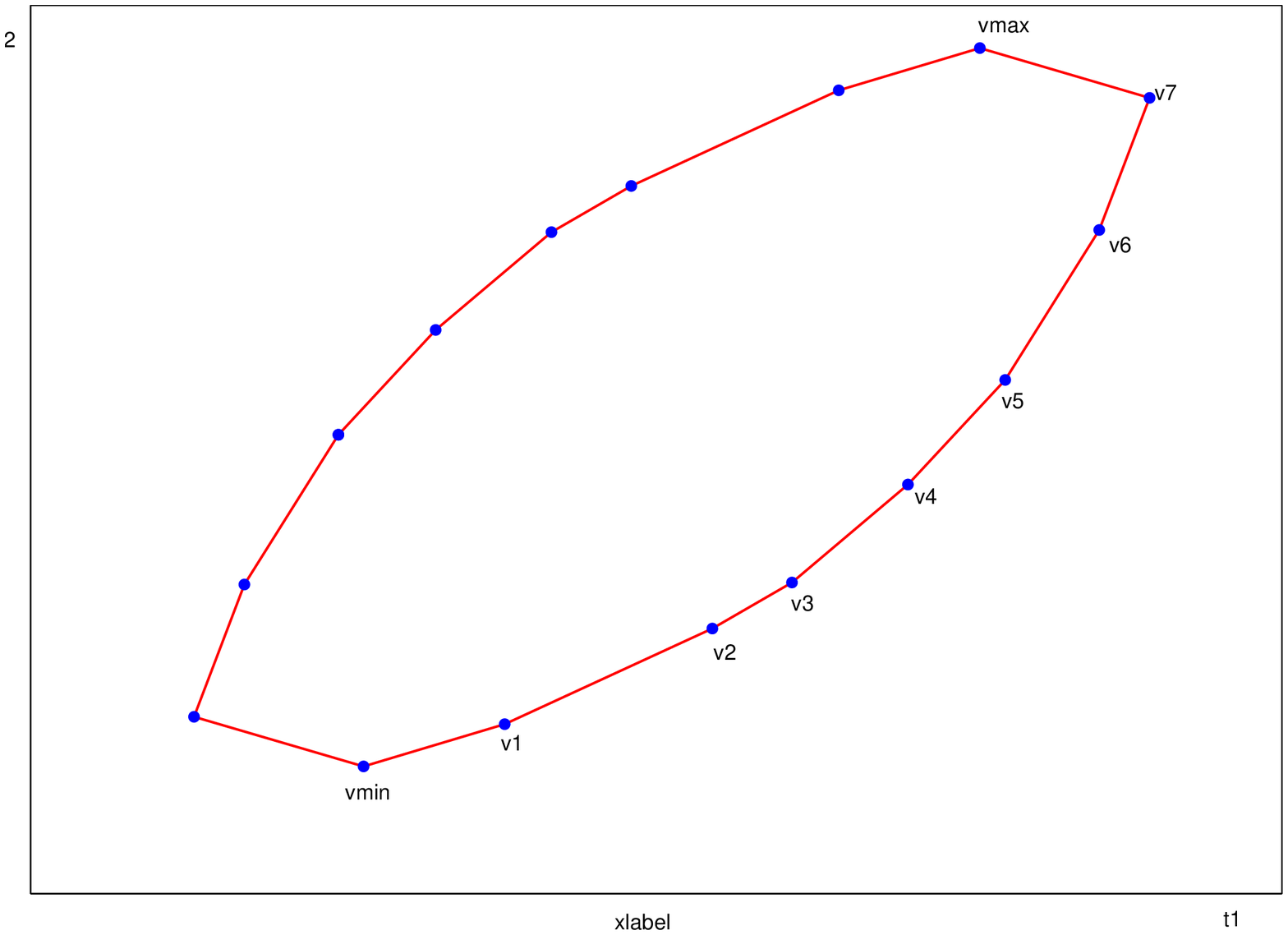}
    \includegraphics[scale=0.37]{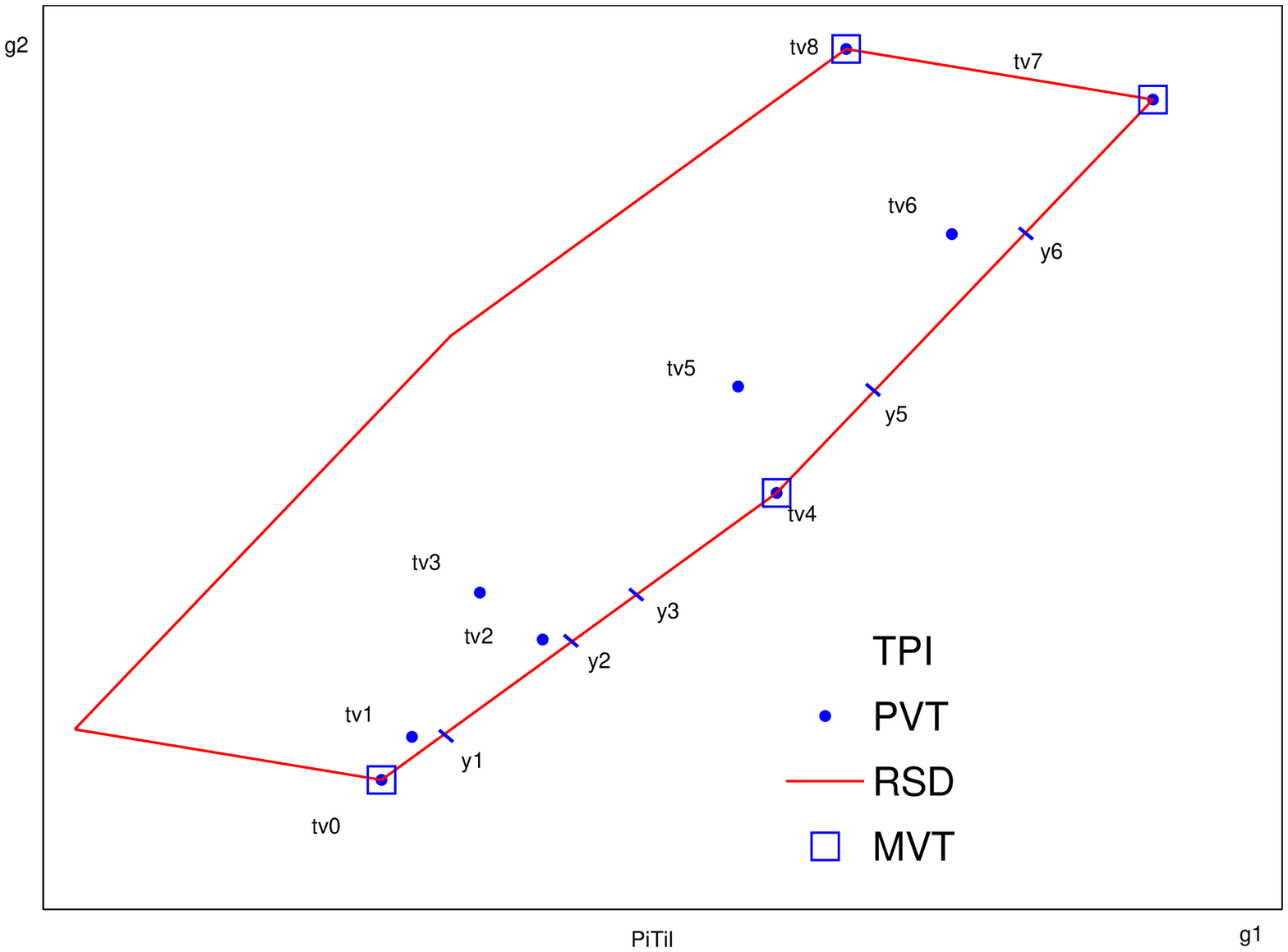}
    \caption{\label{fig:algorithm_2_iterations} Matching and alignment performed in Algorithm 2.}
  \end{psfrags}
\end{figure}
  
The intuition behind the construction is the same as that presented in Section 
\ref{sec:constr-affine-contr}. In particular, the \emph{matching} constraints in 
system \eqref{eq:system_for_z_coefficients} ensure that for any vertex $\mb{w}$ of the 
hypercube $\H_{k+1}$ that corresponds to a potential maximizer in problem $(OPT)$ 
(through $\mb{w} \in \H_{k+1}
\overset{\tiny \eqref{eq:pi1_pi2_definition}}{\mapsto}
\mb{v}_{i} \in \Pi \overset{\tiny \eqref{eq:vtilde_definition_for_affine_cost}}{\mapsto}
\tilde{\mb{v}}_{i} \in \rside(\Delta_\Pi)$), the value of the affine cost 
$z(\mb{w})$ is equal to the value of the initial convex cost, 
$h(\mb{v}_{i}) \equiv h(\pi_2(\mb{w}))$, implying that the value in problem 
$(AFF)$ of \eqref{eq:AFF_problem_for_affine_cost} at 
$\left( \hat{\pi}_1(\mb{w}), \hat{\pi}_2(\mb{w}) \right)$ is equal to the value in problem 
$(OPT)$ of \eqref{eq:OPT_problem_for_affine_cost} at $\tilde{\mb{v}}_{i}$.
The \emph{alignment} constraints in system \eqref{eq:system_for_z_coefficients}
ensure that any such \emph{matched} points, 
$\left( \hat{\pi}_1(\mb{w}), \hat{\pi}_2(\mb{w}) \right)$,
actually correspond to the vertices on the right side of the zonogon $\hat{\Pi}$,
$\rside(\hat{\Pi})$, which implies that, as desired, $\hat{\Pi} \equiv 
\zhull \left( \left\{\tilde{\mb{v}}_0,\dots,\tilde{\mb{v}}_{k+1}\right\} \right)$.

We conclude our preliminary remarks by noting that system (\ref{eq:system_for_z_coefficients})
does not directly impose the robust domination constraint (\ref{eq:simple_notation_robust_h_cost}).
However, as we will soon see, this result is a byproduct of the way the 
matching and alignment are performed in Algorithm 2. 

\subsubsection{Affine Cost \texorpdfstring{$z(\cdot)$}{z()} Dominates Convex Cost \texorpdfstring{$h(\cdot)$}{h()} and Preserves Overall Objective.}
\label{sec:proof_of_corectness_affine_cost_construction}

In this section, we will prove that the affine cost $z(\mb{w})$ computed in 
Algorithm 2 not only robustly dominates the original convex cost (\ref{eq:simple_notation_robust_h_cost}),
but also preserves the overall min-max value (\ref{eq:simple_notation_max_problem_h_cost}).

The following lemma summarizes the first main result:
\begin{lemma}
  \label{lem:affine_cost_preserves_JmM}
  System \eqref{eq:system_for_z_coefficients} is always feasible, and the 
  solution $z(\mb{w})$ always satisfies equation \eqref{eq:simple_notation_max_problem_h_cost}.
\end{lemma}
\begin{proof}
  We first note that $s(1) = 0$ and $s(n)=k+1$, i.e. $\tilde{\mb{v}}_0 ,
  \tilde{\mb{v}}_{k+1} \in \rside(\Delta_\Pi)$. To see why that is the case, 
  note that, by (\ref{eq:vertex_Pi_min_max}),
  $\mb{v}_0$ will always have the smallest $\pi_2$ coordinate in the zonogon $\Pi$. 
  Since the transformation (\ref{eq:vtilde_definition_for_affine_cost}) yielding 
  $\tilde{\mb{v}}_i$ leaves the second coordinate unchanged, it will always be true that 
  $\tilde{\mb{v}}_0 = \argmax_{\tilde{\pi}_1} \argmin_{\tilde{\pi}_2} 
  \left\{ \tilde{\mb{v}}_i, \, i \in \{0,\dots,k+1\}
  \right\}$, which will immediately imply that $\tilde{\mb{v}}_0 \in \rside(\Delta_\Pi)$.
  The proof for $\tilde{\mb{v}}_{k+1}$ follows in an identical matter, 
  since $\mb{v}_{k+1}$ has the largest $\pi_2$ coordinate in $\Pi$.

  It can then be checked that the following choice of $z_i$ always satisfies 
  system (\ref{eq:system_for_z_coefficients}):
  \begin{align*}
    z_0 &= h(\mb{v}_0); \quad 
    z_j = K_{s(i)} \cdot b_j - a_j, \quad  \forall\, j \in \{s(i-1)+1,\dots,s(i)\},
    \quad \forall \, i \in \{2,\dots,n\}, \\
    K_{s(i)} &= \frac{z_{s(i-1)+1}+\dots+z_{s(i)}+a_{s(i-1)+1}+\dots+a_{s(i)}}
    {b_{s(i-1)+1}+\dots+b_{s(i)}} = 
    \frac{h(v_{s(i)})-h(v_{s(i-1)})+a_{s(i-1)+1}+\dots+a_{s(i)}}
    {b_{s(i-1)+1}+\dots+b_{s(i)}}.
  \end{align*}
  
  The proof of the second part of the lemma will be analogous to 
  that of Lemma \ref{lem:affine_control_preserves_objective}. To start, 
  consider the feasible set of problem 
  $(AFF)$ in (\ref{eq:AFF_problem_for_affine_cost}),
  namely the zonogon $\hat{\Pi}$, and note that, from (\ref{eq:pi1_pi2_definition}), 
  its generators are given by $\mb{a} + \mb{z}$, $\mb{b}$:
  \begin{align}
    \left[\begin{array}{c}
        \mb{a} + \mb{z} \\ \mb{b}
      \end{array} \right] = 
    \left[\begin{array}{cccccc}
        a_1 + z_1& \dots & a_{s(i)}+z_{s(i)} & a_{s(i)+1}+z_{s(i)+1}
        &\dots & a_{k+1} + z_{k+1} \\
        b_1 & \dots & b_{s(1)} & b_{s(1)+1} &\dots & b_{k+1}
      \end{array} \right].\label{eq:final_zonotope_generators_affine_cost}
  \end{align}
  By introducing the following points in $\R^2$:
  \begin{align*}
    \mb{y}_i &= \left( \sum_{j=0}^i (a_j+z_j), ~ \sum_{j=0}^i b_j \right),
  \end{align*}
  we have the following simple claims:
  \begin{itemize}
  \item For any $\mb{v}_i \in \Pi$ that is \emph{matched}, i.e. 
    $\tilde{\mb{v}}_{i} \in \rside(\Delta_\Pi)$, with 
    $\mb{w}_i=[1,1,\dots,1,0,\dots,0]$ 
    denoting the unique\footnote{Recall that we are, again, working under Assumption
      \ref{as:zonogon_max_vertices}, which implies uniqueness by part (iv) of 
      Lemma \ref{lem:zonotope_properties} in the Appendix)} vertex of $\H_{k+1}$ satisfying 
    $\left(\pi_1(\mb{w}_i), \pi_2(\mb{w}_i) \right) = \mb{v}_i$,
    we have:
    \begin{align*}
      \mb{y}_i \overset{\tiny \eqref{eq:AFF_problem_for_affine_cost}}{=} 
      \left( \pi_1(\mb{w}_i) + z(\mb{w}_i), \pi_2(\mb{w}_i) \right) 
      \overset{\tiny \eqref{eq:system_for_z_coefficients}}{=} 
      \left( \pi_1[\mb{v}_i] + h(\mb{v}_i), \pi_2[\mb{v}_i] \right)
      \overset{\tiny \eqref{eq:vtilde_definition_for_affine_cost}}{=}
      \tilde{\mb{v}}_i.
    \end{align*}
    The first equality follows from the definition of the zonogon $\hat{\Pi}$,
    the second follows because any $\tilde{\mb{v}}_{i} \in \rside(\Delta_\Pi)$ is 
    \emph{matched} in system \eqref{eq:system_for_z_coefficients}, and the 
    third equality represents the definition of the points $\tilde{\mb{v}}_i$.
  \item For any vertex $\mb{v}_j \in \Pi$,
    that is \emph{not} matched, i.e. $\tilde{\mb{v}}_j \notin \rside(\Delta_\Pi)$,
    and $s(i) < j < s(i+1)$ for some $i$, we have
      $\mb{y}_j \in [\mb{y}_{s(i)},\mb{y}_{s(i+1)}]$.
    This can be seen by using the \emph{alignment} conditions in system 
    (\ref{eq:system_for_z_coefficients}) in conjunction with 
    (\ref{eq:final_zonotope_generators_affine_cost}),
    since the segments in $\R^2$ given by $[\mb{y}_{s(i)},\mb{y}_{s(i)+1}],
    [\mb{y}_{s(i)+1},\mb{y}_{s(i)+2}],\dots,[\mb{y}_{s(i+1)-1},\mb{y}_{s(i+1)}]$ 
    will always be parallel, with common cotangent given by $K_{s(i+1)}$.
  \end{itemize}
  For a geometric interpretation, the reader is referred 
  back to Figure \ref{fig:algorithm_2_iterations}.
  Corroborating these results with the fact that $\left\{ \tilde{\mb{v}}_{s(1)},\dots,
    \tilde{\mb{v}}_{s(n)} \right\} = \rside(\Delta_\Pi)$ must always satisfy:
  \begin{align}
    \cotan{\tilde{\mb{v}}_{s(1)}}{\tilde{\mb{v}}_{s(2)}} \geq 
    \cotan{\tilde{\mb{v}}_{s(2)}}{\tilde{\mb{v}}_{s(3)}} \geq \dots \geq 
    \cotan{\tilde{\mb{v}}_{s(n-1)}}{\tilde{\mb{v}}_{s(n)}},
    \label{eq:cotangents_right_side_affine_cost}
  \end{align}
  we immediately obtain that the points $\left\{ \mb{y}_{s(1)}, \mb{y}_{s(2)}, 
    \dots, \mb{y}_{s(n)} \right\}$ (corresponding to 
  $\tilde{\mb{v}}_i \in \rside(\Delta_\Pi)$)
  exactly represent the right side of the zonogon $\hat{\Pi}$, 
  which, in turn, implies that $\hat{\Pi} \equiv \zhull\left( 
    \left\{ \tilde{\mb{v}}_0, \tilde{\mb{v}}_1, \dots, \tilde{\mb{v}}_{k+1} 
    \right\} \right)$.
  But then, by Corollary \ref{corol:max_on_zonogon_hull}, the maximum value
  of problem $(OPT)$ in (\ref{eq:OPT_problem_for_affine_cost}) is equal to 
  the maximum value of problem $(AFF)$ 
  in (\ref{eq:AFF_problem_for_affine_cost}), and, since the former is always $J_{mM}$,
  so is that latter.
\end{proof}

In order to complete the second step of the induction, we must only show that the 
robust domination constraint (\ref{eq:simple_notation_robust_h_cost}) is 
also obeyed:
\begin{align*}
  z(\mb{w}) &\geq h\left(\pi_2(\mb{w})\right) ~\Leftrightarrow~
  z_0 + z_1 \cdot w_1 + \dots + z_{k+1}\cdot w_{k+1} \geq 
  h\left(b_0 + b_1 \cdot w_1 + \dots + b_{k+1}\cdot w_{k+1}\right), 
  ~ \forall \, \mb{w} \in \H_{k+1}.
\end{align*}
The following lemma will take us very close to the desired result:
\begin{lemma}
  \label{lem:affine_cost_dominates_convex_cost_at_extreme_points}
  The coefficients for the affine cost $z(\mb{w})$ computed in Algorithm 2
  will always satisfy the following property:
  \begin{align*}
    h\left(b_0 + b_{j(1)} + \dots + b_{j(m)}\right) \leq 
    z_0 + z_{j(1)} + \dots + z_{j(m)}, ~ \forall \, 
    j(1),\dots,j(m) \in \{1,\dots,k+1\}, \forall \, m \in \{1,\dots,k+1\}.
  \end{align*}
\end{lemma}
\begin{proof}
  Before proceeding with the proof, we will first list 
  several properties related to the construction of the affine cost.
  We claim that, upon termination, Algorithm 2 
  will produce a solution to the following system:
  \begin{align}
    &\begin{cases}
      z_0 &= h\left(\mb{v}_{s(1)}\right) \\
      z_0 + z_1 + \dots + z_{s(2)} &= h\left(\mb{v}_{s(2)} \right) \\
      \qquad \vdots    &\qquad \vdots \\
      z_0 + z_1 + \dots + z_{s(n)} &= h\left(\mb{v}_{s(n)} \right) \\
      \frac{z_1+a_1}{b_1} &= \dots = \frac{z_{s(2)}+a_{s(2)}}{b_{s(2)}} = K_{s(2)} \\
      \qquad \vdots  &\qquad \qquad \vdots \\
      \frac{z_{s(n-1)+1}+a_{s(n-1)+1}}{b_{s(n-1)+1}} &= \dots = 
      \frac{z_{s(n)}+a_{s(n)}}{b_{s(n)}} = K_{s(n)}
    \end{cases} \label{eq:system_coefs_full_blown} \\
    & K_{s(2)} \geq \dots \geq K_{s(n)} 
    \label{eq:K_increasing_on_zonotope} \\
    &\begin{cases}
      \frac{h\left(\mb{v}_j\right)-h\left(\mb{v}_{0}\right) + 
        a_{1}+\dots+a_j}{b_{1}+\dots+b_j} \leq 
      K_{s(2)} \leq \frac{h\left(\mb{v}_{s(2)}\right)-h\left(\mb{v}_{j}\right) + 
        a_{j+1}+\dots+a_{s(1)}}{b_{j+1}+\dots+b_{s(1)}}, \\
      \qquad \qquad \forall \, j \in \{1,\dots,s(2)-1\}\\
      \qquad \qquad \vdots \qquad \vdots \\ 
      \frac{h\left(\mb{v}_j\right)-h\left(\mb{v}_{s(n-1)}\right) + 
        a_{s(n-1)+1}+\dots+a_j}{b_{s(n-1)+1}+\dots+b_j} \leq
      K_{s(n)} \leq \frac{h\left(\mb{v}_{s(n)}\right)-h\left(\mb{v}_{j}\right) + 
        a_{j+1}+\dots+a_{s(n)}}{b_{j+1}+\dots+b_{s(n)}}, \\
      \qquad \qquad \forall \, j \in \{s(n-1)+1,\dots,s(n)-1\} ~.
    \end{cases}  \label{eq:cannot_match_intermediate_point}
  \end{align}
  Let us explain the significance of all the equations. 
  (\ref{eq:system_coefs_full_blown}) is simply a rewriting of 
  the original system (\ref{eq:system_for_z_coefficients}), which states 
  that at any vertex $\mb{v}_{s(i)}$, the value of the affine function 
  should exactly match the value assigned by the convex function $h(\cdot)$,
  and the coefficients $z_i$ between any two matched vertices should be 
  such that the resulting segments, $[z_j+a_j,b_j]$, are aligned (i.e. 
  the angle they form with the $\tilde{\pi}_1$ axis has the same cotangent, 
  specified by $K_{(\cdot)}$ variables). We note that we have explicitly used 
  the fact that $s(1)=0$, which we have shown in the first paragraph
  of the proof of Lemma \ref{lem:affine_cost_preserves_JmM}. 

  Equation (\ref{eq:K_increasing_on_zonotope}) is a simple restatement of 
  (\ref{eq:cotangents_right_side_affine_cost}), that the cotangents on the 
  right side of a convex hull must be decreasing.

  Equation (\ref{eq:cannot_match_intermediate_point}) is a direct consequence of 
  the fact that $\{\tilde{\mb{v}}_{s(1)},\tilde{\mb{v}}_{s(2)},
  \dots,\tilde{\mb{v}}_{s(n)}\}$ represent $\rside(\Delta_\Pi)$. To see why that is, consider 
  an arbitrary $j \in \{s(i)+1,\dots,s(i+1)-1\}$. Since 
  $\tilde{\mb{v}}_j \notin \rside(\Delta_\Pi)$, we have:
  \begin{align*}
    \cotan{\tilde{\mb{v}}_{s(i)}}{\tilde{\mb{v}}_{j}} & \leq 
    \cotan{\tilde{\mb{v}}_{j}}{\tilde{\mb{v}}_{s(i+1)}} ~
    \overset{\tiny \eqref{eq:pi1_pi2_definition},\eqref{eq:vtilde_definition_for_affine_cost}}
    {\Leftrightarrow} \\
    \frac{ a_{s(i)+1}+\dots+a_j + h\left(\mb{v}_j\right)-h\left(\mb{v}_{s(i)}\right)}
    {b_{s(i)+1}+\dots+b_j} &\leq \frac{a_{j+1}+\dots+a_{s(i+1)} +
      h\left(\mb{v}_{s(i+1)}\right)-h\left(\mb{v}_{j}\right)}{b_{j+1}+\dots+b_{s(i+1)}} ~
    \Leftrightarrow \\
    \frac{ a_{s(i)+1}+\dots+a_j + h\left(\mb{v}_j\right)-h\left(\mb{v}_{s(i)}\right)}
    {b_{s(i)+1}+\dots+b_j} &\leq K_{s(2)} \leq \frac{a_{j+1}+\dots+a_{s(i+1)} +
      h\left(\mb{v}_{s(i+1)}\right)-h\left(\mb{v}_{j}\right)}{b_{j+1}+\dots+b_{s(i+1)}},
  \end{align*}
  where, in the last step, we have used the mediant 
  inequality\footnote{If $b,d>0$ and$\frac{a}{b}\leq\frac{c}{d}$,
    then $\frac{a}{b} \leq \frac{a+c}{b+d} \leq \frac{c}{d}$.} and the 
  fact that, from (\ref{eq:system_coefs_full_blown}), 
  $K_{s(2)} = \cotan{\tilde{\mb{v}}_{s(i)}}{\tilde{\mb{v}}_{s(i+1)}} = 
  \frac{a_{s(i)+1}+\dots+a_{s(i+1)}+h\left(\mb{v}_{s(i+1)}\right)-h\left(\mb{v}_{s(i)}\right)}
  {b_{s(i)+1}+\dots+b_{s(i+1)}}$ (refer back to Figure \ref{fig:algorithm_2_iterations}
  for a geometrical interpretation).
  
  With these observations, we will now prove the claim of the lemma.
  The strategy of the proof will be to use induction on the size of 
  the subsets, $m$. First, we will show the property 
  for any subset of indices $j(1),\dots,j(m) \in \{s(1)=0,\dots,s(2)\}$, and will 
  then extend it to $j(1),\dots,j(m) \in \{s(i)+1,\dots,s(i+1)\}$ for any $i$, 
  and then to any subset of $\{1,\dots,k+1\}$. 

  The following implications of the conditions 
  (\ref{eq:system_coefs_full_blown}),
  (\ref{eq:K_increasing_on_zonotope}) and (\ref{eq:cannot_match_intermediate_point}),
  are stated here for convenience, since they will be used 
  throughout the rest of the proof:
  \begin{align}
    & h\left(\mb{v}_{s(1)}\right) = h(\mb{v}_0) = z_0; \quad h(\mb{v}_{s(2)}) = 
    z_0 + z_1 + \dots + z_{s(2)}.
    \label{eq:matching_v0_vs1} \\
    & h(\mb{v}_j) - h(\mb{v}_0) \leq z_1 + \dots + z_j, 
    \qquad \forall \, j \in \{1,\dots,s(2)-1\}.
    \label{eq:inequality_ok_at_partial_sums} \\
    & \frac{z_1}{b_1} \leq \dots \leq \frac{z_j}{b_j} \leq \dots \leq \frac{z_{s(2)}}{b_{s(2)}},
    \qquad \forall \, j \in \{1,\dots,s(2)-1\}.
    \label{eq:almost_increasing_increments}
  \end{align}
  Their proofs are straightforward.
  (\ref{eq:matching_v0_vs1}) follows directly from system 
  (\ref{eq:system_coefs_full_blown}), and:
  \begin{align*}
    &\frac{h(\mb{v}_j) - h(\mb{v}_0) + a_1+\dots + a_j}{b_1+\dots+b_j} 
    \overset{\tiny \eqref{eq:cannot_match_intermediate_point}}{\leq} 
    K_{s(2)} \overset{\tiny \eqref{eq:system_coefs_full_blown}}{=}
    \frac{z_1+\dots+z_j + a_1+\dots + a_j}{b_1+\dots+b_j} ~\Rightarrow 
    \eqref{eq:inequality_ok_at_partial_sums} ~\text{true}.\\
    & \begin{cases}
      \eqref{eq:system_coefs_full_blown}:~ 
      \frac{a_1+z_1}{b_1} = \dots = \frac{a_j+z_j}{b_j} = \dots =
      \frac{a_{s(2)}+z_{s(2)}}{b_{s(2)}}~  \\
      \Pi ~\text{zonogon}~ \Rightarrow 
      \frac{a_1}{b_1} > \dots > \frac{a_j}{b_j} > \dots 
      > \frac{a_{s(2)}}{b_{s(2)}}
    \end{cases} ~\Rightarrow~ \eqref{eq:almost_increasing_increments} ~\text{true}.
  \end{align*}
  We can now proceed with the proof, by checking the induction for $m=1$. 
  We would like to show that:
  \begin{align*}
    h\left(b_0 + b_{j}\right) \leq 
    z_0 + z_{j}, \quad \forall \, j \in \{1,\dots,s(2)\}
  \end{align*}
  Writing $b_0+b_j$ as:
  \begin{align*}
    b_0 + b_j &= (1-\lambda) \cdot b_0 + \lambda \cdot (b_0 + \dots + b_j) \\
    \lambda &= \frac{b_j}{b_1 + \dots + b_j},
  \end{align*}
  we obtain:
  \begin{align*}
    h(b_0 + b_j) &\leq (1-\lambda) \cdot h(b_0) + 
    \lambda \cdot \expl{\equiv h(\mb{v}_j)}
    {h(b_0 + \dots + b_j)} \\
    &= h(\mb{v}_0) + \frac{b_j}{b_1+\dots+b_j} \left[\, h(\mb{v}_j) - 
      h(\mb{v}_0) \,\right]  ~\leq~ (\text{by (\ref{eq:matching_v0_vs1}) 
      if $j=s(2)$ or 
      (\ref{eq:inequality_ok_at_partial_sums}) otherwise}) \\
    &\overset{}{\leq} z_0 + 
    \frac{b_j}{b_1+\dots+b_j} \left( z_1+\dots+z_j \right) ~\leq~
    (\text{by (\ref{eq:almost_increasing_increments}) and the mediant inequality})\\
    &\leq z_0 + z_j.
  \end{align*}
  Assume the property is true for any subsets of size $m$. Consider a 
  subset $j(1),\dots,j(m),j(m+1)$, and, without loss of generality, let
  $j(m+1)$ be the largest index. With the convex combination:
  \begin{align*}
    b^* &\bydef b_0 + b_{j(1)} + \dots + b_{j(m)} + b_{j(m+1)} \\
    &~= (1-\lambda) \cdot ( b_0 + b_{j(1)} + \dots + b_{j(m)} ) + \lambda \cdot 
    (b_0 + b_1 + \dots + b_{j(m+1)-1} + b_{j(m+1)} ), \\
    \text{where}~ \lambda &= \frac{b_{j(m+1)}}{(b_1+b_2+\dots+b_{j(m+1)})-( b_{j(1)} +
      b_{j(2)} + \dots + b_{j(m)})},
  \end{align*}
  we obtain:
  \begin{align*}
    h(b^*) &\leq (1-\lambda) \cdot h(b_0 + b_{j(1)} + \dots + b_{j(m)}) + \lambda \cdot 
    h\left( \mb{v}_{i(m+1)}\right) \, \leq \, (\text{by induction hypothesis and 
      (\ref{eq:matching_v0_vs1}), (\ref{eq:inequality_ok_at_partial_sums})})\\
    &\leq (1-\lambda) \cdot (z_0 + z_{j(1)} + \dots + z_{j(m)}) + \lambda \cdot 
    \left( z_0 + z_1 + \dots + z_{i(m+1)} \right) \\
    &= z_0 + z_{j(1)} + \dots + z_{j(m)} + \frac{b_{j(m+1)}}
    {(b_1+b_2+\dots+b_{j(m+1)})-( b_{j(1)} + b_{j(2)} + \dots + b_{j(m)})} \cdot \\
    & \qquad \cdot 
    \left[ (z_1+z_2+\dots+z_{j(m+1)})-( z_{j(1)} + z_{j(2)} + 
      \dots + z_{j(m)}) \right] \, \leq 
    \, (\text{by (\ref{eq:almost_increasing_increments}) and 
      mediant inequality})\\
    &\leq z_0 + z_{j(1)} + \dots + z_{j(m)} + z_{j(m+1)}.
  \end{align*}
  We claim that the exact same procedure can be repeated for a subset of indices 
  from $\{s(i)+1,\dots,s(i+1)\}$, for any index $i \in \{1,\dots,n-1\}$. 
  We would simply 
  be using the adequate inequality from (\ref{eq:cannot_match_intermediate_point}),
  and the statements equivalent to (\ref{eq:matching_v0_vs1}), 
  (\ref{eq:inequality_ok_at_partial_sums}) and 
  (\ref{eq:almost_increasing_increments}). 
  The following results would be immediate:
  \begin{align}
    &h\left( (b_0+b_1+\dots+b_{s(i)}) + b_{j(1)} + \dots + b_{j(m)} \right) \leq 
    \left(z_0+z_1+\dots+z_{s(i)}\right) + z_{j(1)} + \dots + z_{j(m)}, 
    \label{eq:inequalities_ok_in_any_interval}\\
    &\qquad \qquad 
    \forall \, i \in \{1,\dots,n\}, ~\forall \, 
    j(1),\dots,j(m) \in \{s(i)+1,\dots,s(i+1)\}.
    \nonumber
  \end{align}
  Note that instead of the term $b_0$ for the argument of $h(\cdot)$, we would 
  use the complete sum $b_0+b_1+\dots+b_{s(i)}$, and, similarly, instead of 
  $z_0$ we would have the complete sum $z_0+z_1+\dots+z_{s(i)}$. With these 
  results, we can make use of the increasing increments property of convex functions:
  \begin{align*}
    \frac{h(x_1+\Delta) - h(x_1)}{\Delta} \leq 
    \frac{h(x_2+\Delta) - h(x_2)}{\Delta}, \quad \forall \, \Delta > 0, x_1 \leq x_2 ~,
  \end{align*}
  to obtain the following result:
  \begin{align*}
    & h\left(b_0 + \expl{j(\cdot) \in \{1,\dots,s(2)\}}{b_{j(1)} + \dots + b_{j(m)}} + 
      \expl{i(\cdot) \in \{s(2)+1,\dots,s(3)\}}{b_{i(1)} + \dots + b_{i(l)}} \right)
    - h\left(b_0 + b_{j(1)} + \dots + b_{j(m)}\right) ~\leq~\\
    &\qquad \leq h\left(b_0 + \expl{\text{all indices in}~ \{1,\dots,s(2)\}}
      {b_1 + \dots + b_{s(2)}} + b_{i(1)} + \dots + b_{i(l)} \right)
    - \expl{\bydef h(\mb{v}_{s(2)})}
    {h\left(b_0 + {b_1 + \dots + b_{s(2)}} \right)} 
    \overset{\tiny \eqref{eq:matching_v0_vs1}, \eqref{eq:inequalities_ok_in_any_interval}}{\leq}  \\
    &\qquad \leq \left(z_0+z_1+\dots+z_{s(2)}\right) + z_{i(1)} + \dots + z_{i(l)} - 
    \left(z_0+z_1+\dots+z_{s(2)}\right) \\
    &\qquad =  z_{i(1)} + \dots + z_{i(l)} \qquad \Rightarrow \\
    &h\left(b_0 + {b_{j(1)} + \dots + b_{j(m)}} + {b_{i(1)} + \dots + b_{i(l)}} \right)
    \leq h\left(b_0 + b_{j(1)} + \dots + b_{j(m)}\right) +  z_{i(1)} + \dots + z_{i(l)} 
    \overset{\tiny \eqref{eq:inequalities_ok_in_any_interval}}{\leq} \\
    &\qquad \qquad \qquad \leq z_0 +  z_{j(1)} + \dots + z_{j(m)} +  z_{i(1)} + \dots + z_{i(l)}.
  \end{align*}
  We showed the property for indices drawn only from the first two intervals, 
  $\{s(1)+1,\dots,s(2)\}$ and $\{s(2)+1,\dots,s(3)\}$,
  but it should be clear how the argument can be immediately extended to any 
  collection of indices, drawn from any intervals. We omit the details for brevity,
  and conclude that the claim of the lemma is true.
\end{proof}

We are now ready for the last major result:
\begin{lemma}
  \label{lem:affine_cost_dominates_convex_cost}
  The affine cost $z(\mb{w})$ computed by Algorithm 2 
  always dominates the convex cost $h(\pi_2(\mb{w}))$:
  \begin{align*}
    h\left(b_0 + \sum_{i=1}^{k+1} b_i \cdot w_i \right)
    \leq z_0 + \sum_{i=1}^{k+1} z_i \cdot w_i, \quad 
    \forall \, \mb{w} \in \H_{k+1} = [0,1]^{k+1}.
  \end{align*}
\end{lemma}
\begin{proof}
  Note first that the function $f(\mb{w}) \bydef h\left( b_0 + 
    \sum_{i=1}^{k+1} b_i \cdot w_i \right) - (z_0 + \sum_{i=1}^{k+1} 
  z_i \cdot w_i)$ is a convex function of $\mb{w}$. Furthermore,
  the result of Lemma \ref{lem:affine_cost_dominates_convex_cost_at_extreme_points}
  can be immediately rewritten as:
  \begin{align*}
    h\left(b_0 + \sum_{i=1}^{k+1} b_i \cdot w_i \right)
    \leq z_0 + \sum_{i=1}^{k+1} z_i \cdot w_i, \,
    \forall \, \mb{w} \in \{0,1\}^{k+1} ~\Leftrightarrow~
    f(\mb{w}) \leq 0, \, \forall \, \mb{w} \in \{0,1\}^{k+1}.
  \end{align*}
  Since the maximum of a convex 
  function on a polytope occurs on the extreme points of the polytope, 
  and $\ext( \H_{k+1} ) = \{0,1\}^{k+1}$,
  we immediately have that:
  $\max_{\mb{w} \in \H_{k+1}} f(\mb{w}) = 
  \max_{\mb{w} \in \{0,1\}^{k+1}} f(\mb{w}) \leq 0$,
  which completes the proof of the lemma.
\end{proof}

We can now conclude the proof of correctness in the construction of the affine
stage cost, $z(\mb{w})$. With Lemma \ref{lem:affine_cost_dominates_convex_cost}, 
we have that the affine cost always dominates the convex cost $h(\cdot)$, thus 
condition (\ref{eq:simple_notation_robust_h_cost}) is obeyed. Furthermore, 
from Lemma \ref{lem:affine_cost_preserves_JmM},
the overall min-max cost remains unchanged even when incurring 
the affine stage cost, $z(\mb{w})$, hence condition 
(\ref{eq:simple_notation_max_problem_h_cost}) is also true.
This completes the construction of the affine cost, and hence also the 
full step of the induction hypothesis. 

\subsubsection{Proof of Main Theorem.}
\label{sec:proof_of_main_theorem}
To finalize the current section, we summarize 
the steps that have lead us to the result, thereby proving the main 
Theorem \ref{thm:main_theorem}.

\begin{proof}[Proof of Theorem \ref{thm:main_theorem}]
  In Section \ref{sec:induct_hypo}, we have verified the induction hypothesis 
  at time $k=1$. With the induction hypothesis assumed 
  true for times $t=1,\dots,k$, we have listed the initial 
  consequences in Lemma \ref{lem:theta_restricted_region} and 
  Corollary \ref{corol:max_on_zonogon_hull} of Section \ref{sec:simplified-notation}.
  By exploring the structure of the optimal control law, $u_{k+1}^*(x_{k+1})$, 
  and the optimal value function, $J_{k+1}^*(x_{k+1})$, in 
  Section \ref{sec:further_underst_induc_hypo}, we have finalized the 
  analysis of the induction hypothesis, and summarized our findings in 
  Lemmas \ref{lem:maximum_in_OPT} and \ref{lem:rside_delta_gamma}.
  
  Section \ref{sec:constr-affine-contr} then introduced the main 
  construction of the affine control law, $q_{k+1}(\mb{w}^{k+1})$, 
  which was shown to be robustly feasible 
  (Lemma \ref{lem:affine_control_coefficients_robust}). Furthermore, in
  Lemma  \ref{lem:affine_control_preserves_objective}, 
  we have shown that, when used in conjuction with the original convex state costs, 
  $h_{k+1}\left(x_{k+2}\right)$, this affine control preserves the min-max value 
  of the overall problem.
  
  In Section \ref{sec:construction_affine_stage_cost}, we have also introduced 
  an affine stage cost, $z_{k+2}(\mb{w}^{k+2})$, which, if incurred at time $k+1$,
  will always preserve the overall min-max value (Lemma \ref{lem:affine_cost_preserves_JmM}),
  despite being always larger than the original convex cost, $h_{k+1}\left(x_{k+2}\right)$
  (Lemma \ref{lem:affine_cost_dominates_convex_cost}). 
\end{proof}

\subsubsection{Counterexamples for potential extensions.}
\label{sec:counter_examples}
On first sight, one might be tempted to believe that the results in Theorem 
\ref{thm:main_theorem} could be immediately extended to more general problems. 
In particular, one could be tempted to ask one of the following natural 
questions:
\begin{enumerate}
\item Would both results of Theorem \ref{thm:main_theorem}
  (i.e. existence of affine control laws \emph{and} existence of affine stage costs)
  hold for a problem which also included linear constraints 
  coupling the controls $u_t$ across different time-steps? 
  (see \citet{BenTal05} for a situation when this might be of interest)
\item Would both results of Theorem \ref{thm:main_theorem} 
  hold for multi-dimensional linear systems? (i.e. problems 
  where $x_k \in \R^d, \, \forall \, k$, with $d \geq 2$)
\item Are affine policies in the disturbances optimal for 
  the two problems above?
\end{enumerate}

In the rest of the current section, we would like to show how these questions 
can all be answered negatively, using the following simple counterexample:
\begin{align*}
  & T=4, c_k = 1, ~ h_k(x_{k+1}) = \max\{ 18.5 \cdot x_{k+1}, \, -24 \cdot x_{k+1} \}, ~
  L_k =0, U_k = \infty, \, \forall \, k \in \{1,\dots,4\} ~,\\
 (CEx) \qquad &w_1 \in [-7,0], w_2 \in [-11,0], w_3 \in [-8,0], w_4 \in [-44,0] ~,\\
  &\sum_{i=1}^k u_i \leq 10 \cdot k ~, \forall \, k \in \{1,\dots,4\}.
\end{align*}
The first two rows describe a one-dimensional 
problem that fits the conditions of Problem 
\ref{prob:initial_problem} in Section \ref{sec:introduction}. 
The third row corresponds to a coupling constraint for controls 
at different times, so that the problem fits question (i) above. 
Furthermore, since the state in such a problem consists of 
two variables (one for $x_k$ and one for $\sum_{i=1}^k u_k$), 
the example also fits question (ii) above.

The optimal min-max value for the counterexample $(CEx)$ above can 
be found by solving a stochastic optimization problem (see \citet{BenTal05}), 
in which non-anticipatory decisions are computed at 
all the extreme points of the uncertainty set, i.e. for $\{\vlow{w}{1},\vup{w}{1}\}
\times \{\vlow{w}{2},\vup{w}{2}\} \times \{\vlow{w}{3},\vup{w}{3}\} \times 
\{\vlow{w}{4},\vup{w}{4}\}$.
The resulting model, which is a large linear program, can be solved 
to optimality, resulting in a corresponding value of approximately $838.493$ for 
problem ($CEx$).

To compute the optimal min-max objective obtained by using 
affine policies $q_k(\mb{w}^k)$ and incurring affine costs $z_k(\mb{w}^{k+1})$, 
one can ammend the model $(AARC)$ from Section \ref{sec:optim-affine-polic} 
by including constraints for the cumulative controls (see \citet{BenTal05} for details), 
and then using \eqref{eq:typical_constr} to rewrite the resulting model as a linear program. 
The optimal value of this program for counterexample ($CEx$) was approximately 
$876.057$, resulting in a gap of $4.4\%$, and thus providing a negative 
answer to questions (i) and (ii).

To investigate question (iii), we remark that the smallest objective achievable by using 
affine policies of the type $q_k(\mb{w}^k)$ can be found by solving 
another stochastic optimization problem, 
having as decision variables the affine coefficients $\{q_{k,t}\}_{0 \leq t < k \leq T}$, 
as well as (non-anticipatory) stage cost 
variables $z_k^{\mb{w}}$ 
for every time step $k \in \{1,\dots, T\}$ and every extreme point $\mb{w}$ 
of the uncertainty set. Solving 
the resulting linear program for instance $(CEx)$ gave an optimal value 
of $873.248$, so strictly larger than the (true) optimum ($838.493$), and strictly 
smaller than the optimal value of the model utilizing 
both affine control policies \emph{and} affine stage costs ($876.057$).

Thus, with question (iii) also answered negatively, we conclude that 
policies that are affine in the disturbances, $q_k(\mb{w}^k)$, are in general 
\emph{suboptimal} for problems 
with cumulative control constraints or multiple dimensions, and that replacing the 
convex state costs $h_k(x_{k+1})$ by (larger) affine costs $z_k(\mb{w}^{k+1})$
would, in general, result in even \emph{further} deterioration of the objective.

\section{An application in inventory management.}
\label{sec:invent_management_application}
In this section, we would like to explore our results in connection with a 
classical inventory problem. This idea was originally 
introduced by \citet{BenTal05}, in the context of a more general 
model: the retailer-supplier with flexible commitment contracts problem (RSFC). 
We will first describe a simplified version of the problem, and then 
draw a very interesting connection with our results. 

The setting is the following: consider a single-product,
single-echelon, multi-period supply chain, in which inventories 
are managed periodically over a planning horizon of $T$ periods. 
The unknown demands $w_t$ from customers arrive at the (unique) 
echelon, henceforth referred to as the \emph{retailer}, and are satisfied 
from the on-hand inventory, denoted by $x_t$ at the beginning 
of period $t$. The retailer can replenish the inventory by placing 
orders $u_t$, at the beginning of each period $t$, for 
a cost of $c_t$ per unit of product. These orders are 
immediately available, i.e. there is no lead-time in the system, but there 
are capacities on how much the retailer can order: $L_t \leq u_t \leq U_t$.
After the demand $w_t$ is realized, the retailer incurs holding 
costs $H_t \cdot \max\{0, x_t+u_t-w_t\}$ for all the 
amounts of supply stored on her premises, as well as penalties 
$B_t\cdot \max\{w_t-x_t-u_t, 0\}$, for any demand that is backlogged. 

In the spirit of robust optimization, we will assume that the only information
available about the demand at time $t$ is that it resides within a 
certain inverval centered around a \emph{nominal} (or mean) 
demand $\bar{d}_t$, which results in the uncertainty set
  $\W_t = \{ \, \abs{w_t-\bar{d}_t} \leq 
  \rho \cdot \bar{d}_t \, \}$,
where $\rho \in [0,1]$ can be interpreted as an \emph{uncertainty level}. 
As such, if we take the objective function to be minimized as the cost 
resulting in the worst-case scenario, we immediately obtain an instance 
of our original Problem \ref{prob:initial_problem}, with $\a_t= \b_t=1, 
\c_t=-1$, and the convex state costs $h_t(\cdot)$ denoting the Newsvendor 
cost, 
$h_t(x_{t+1}) = H_t \cdot \max\{x_t + u_t - w_t, 0\} + 
B_t \cdot \max\{ w_t - x_t - u_t, 0\}$.

Therefore, the results in Theorem \ref{thm:main_theorem} are immediately applicable to 
conclude that no loss of optimality is incurred when we restrict 
attention to affine order quantities $q_t$ that depend on the history of available 
demands at time $t$,
 $q_t(\mb{w}^t) = \aff{q}{t}{0} + \sum_{\tau=1}^{t-1}\aff{q}{t}{\tau}\cdot w_\tau$,
and when we replace the Newsvendor costs $h_t(x_{t+1})$ by some (potentially larger) 
affine costs $z_t(\mb{w}^{t+1})$. The main advantage is that, with these 
substitutions, the problem of finding the optimal affine policies becomes an LP 
(see the discussion in Section \ref{sec:optim-affine-polic} and the paper by 
\citet{BenTal05} for more details). 

The more interesting connection with our results comes if we recall the construction 
in Algorithm 1. In particular, we have the following simple claim:
\begin{proposition}
  \label{rmrk:decaying_memory_property}
  If the affine orders $q_t(\mb{w}^t)$ computed 
  in Algorithm 1 are implemented at every time step $t$, and we let:
  $x_k(\mb{w}^k) = x_1 + \sum_{t=1}^{k-1} \left( q_t(\mb{w}^t) - w_t \right)
  \bydef \aff{x}{t}{0} + \sum_{t=1}^{k-1} \aff{x}{k}{t} \cdot w_t$
  denote the affine dependency of the inventory $x_k$ on the history 
  of demands, $\mb{w}^k$, then:
  \begin{enumerate}
  \item If a certain demand $w_t$ is fully satisfied by time $k \geq t+1$, i.e. 
    $x_{k,t}=0$, then all the (affine) orders $q_{\tau}$ placed after time $k$ will not depend 
    on $w_t$.
  \item Every demand $w_t$ is at most satisfied by the future orders $q_k, \, k \geq t+1$,
    and the coefficient $q_{k,t}$ represents what fraction of the 
    demand $w_t$ is satisfied by the order $q_k$.
  \end{enumerate}
\end{proposition}
\begin{proof}
  To prove the first claim, recall that, in our notation from Section \ref{sec:simplified-notation}, 
  $x_k \equiv \theta_2 = b_0 + \sum_{t=1}^{k-1} b_t \cdot w_t$. 
  Applying part $(i)$ of Lemma \ref{lem:affine_control_coefficients_robust} in the 
  current setting\footnote{The signs of the inequalities are changed 
    because every disturbance, $w_t$, is entering the system dynamics with a coefficient $-1$,
    instead of $+1$, as was the case in the discussion from Section \ref{sec:simplified-notation}.}, 
  we have that $0 \leq q_{k,t} \leq -x_{k,t}$. Therefore, if $x_{k,t}=0$, then $q_{k,t}=0$, which 
  implies that $x_{k+1,t}=0$. By induction, we immediately get that 
  $q_{\tau,t}=0, \, \forall \, \tau \in \{k,\dots,T\}$.

  To prove the second part, note that any given demand, $w_t$, initially has an 
  affine coefficient of $-1$ in the state $x_{t+1}$, i.e. $x_{t+1,t}=-1$. By 
  part $(i)$ of Lemma \ref{lem:affine_control_coefficients_robust}, 
  $0 \leq q_{t+1,t} \leq -x_{t+1,t} = 1$, so that $q_{t+1,t}$ represents a 
  fraction of the demand $w_t$ satisfied by the order $q_{t+1}$. Furthermore,
  $x_{t+2,t}=x_{t+1,t}+q_{t+1,t} \in [-1,0]$, so, by induction, we immediately have 
  that $q_{k,t} \in [0,1], \, \forall\, k \geq t+1$, and 
  $\sum_{k=t+1}^{T} q_{k,t} \leq 1$.
\end{proof}

In view of this result, if we think of 
$\{q_k\}_{k \geq t+1}$ as future orders that are partially 
satisfying the demand $w_t$, then
every future order quantity $q_k(\mb{w}^k)$ will satisfy exactly a fraction of 
the demand $w_t$ (since the coefficient for $w_t$ in $q_k$ 
will always be in $[0,1]$), and every demand will be at most 
satisfied by the sequence of orders following after it appears. 
This interpretation bears some similarity with 
the unit decomposition approach of \cite{TsiMuh03}, where every unit 
of supply can be interpreted as satisfying a particular unit of the demand.
Here, we are accounting for fractions of the total demand, as being 
satisfied by future order quantities.

\section{Conclusions. Future Directions.}
\label{sec:conclusions_future_direct}
We have presented a novel approach for theoretically handling robust, 
multi-stage decision problems. The method strongly utilized the connections between the 
geometrical properties of the feasible sets (zonogons), and the objective 
functions being optimized, in order to prune the set of relevant 
points and derive properties about the optimal policies for the problem. 
We have also shown an interesting implication of our theoretical results 
in the context of a classical problem in inventory management. 

On a theoretical level, one immediate direction of future research would be to study 
systems with mixed (polyhedral) constraints, on both 
state and control at time $t$. Furthermore, we would like to explore 
the possibility of utilizing the same proof technique 
in the context of multi-dimensional problems, as well as for more 
complicated uncertainty sets $\W$. 

Second, we would like to better understand the connections between the 
matching performed in Algorithm 2 
and the properties 
of convex (and supermodular) functions, and explore extensions of the 
approach to handle cost functions that are 
not necessarily convex, as well as non-linear cost structures for the control 
$u_t$. Another potential area of interest would be to use our analysis 
tools to quantify the performance of affine 
policies even in problems where they are known to be suboptimal (such as the one 
suggested in Section \ref{sec:counter_examples}). This could 
potentially lead to fast approximation algorithms, with solid theoretical 
foundations.

On a practical level, we would like to explore potential applications 
arising in robust portfolio optimization, as well as operations management.
Also, we would like to construct a procedure 
that mimics the behavior of our algorithms, but does not require 
knowledge of the optimal value functions $J^*(\cdot)$ or optimal controllers $u^*(\cdot)$.
One potential idea would be to explore which types of cuts could be added 
to the linear program ($AARC$), to ensure that it computes a solution
as ``close'' to the affine controller $q(\mb{w})$ as possible.

\section{Appendix.}
\label{sec:appendix}

\subsection{Dynamic Programming Solution.}
\label{sec:appendix:dynam-progr-solution}
This section contains a detailed proof for the solution of the Dynamic 
Programming formulation, initially introduced in 
Section \ref{sec:DP_formulation}. Recall that the problem we would like to solve 
is the following:
\begin{align}
  J_{mM} \bydef \quad &\underset{u_1}{\min} \left[ c_1 \cdot u_1 + 
    \underset{w_1}{\max} \left[ 
        h_1(x_2) + \dots + \underset{u_k}{\min} \left[ c_k \cdot u_k + 
          \underset{w_k}{\max} \left[ h_k(x_{k+1}) + 
            \dots \right. \right. \right. \right. \nonumber \\
 (DP) \qquad \qquad \qquad & \qquad \qquad \qquad \qquad  \left. \left. \left.
          + \underset{u_{T}}{\min} \left( c_T \cdot u_T + \underset{w_{T}}{\max}
            ~ h_T(x_{T+1}) ~\right) \dots \right] \dots \right] ~\right] \nonumber \\
  \text{s.t.} \quad &  x_{k+1} = x_k + u_k + w_k \nonumber \\  
  & L_k \leq u_k \leq U_k 
  \qquad \qquad \qquad \forall \, k \in \{1,2,\dots,T\}, \nonumber \\
  & w_k \in \W_k = [\vlow{w}{k},\vup{w}{k}] \nonumber 
\end{align}
which gives rise to the corresponding Bellman recursion:
\begin{align*}
  J_{k}^*(x_k) &\bydef \underset{L_k \leq u_k \leq U_k}{\min} \left[~
    c_k \cdot u_k + \underset{w_k \in \W_k}{\max} \left[~ h_k(x_k+u_k+w_k)
      + J_{k+1}^*\left(x_k+u_k+w_k\right)  ~\right] ~\right].
\end{align*}
According to our definition of running cost and cost-to-go, the cost at 
$T+1$ is $J_{T+1}^*=0$, which yields the following Bellman recursion at 
time $T$:
\begin{align*}
  J_{T}^*(x_T) & \bydef 
  \underset{L_T \leq u_T \leq U_T}{\min} \left[~ c_T \cdot u_T + 
  ~ \underset{w_T \in \W_T}{\max} ~ h_T(x_T + u_T + w_T) \right].
\end{align*}

First consider the inner (maximization) problem. Letting 
$y_T \bydef x_T + u_T$, we obtain:
\begin{align}
  g_T(y_T) &\bydef \underset{w_T \in [\vlow{w}{T},\vup{w}{T}]}{\max} 
  ~ h_T(x_T + u_T + w_T) \nonumber \\
  \text{(since $h_T(\cdot)$ convex)} \quad &=
  \max \left\{ h_T \left(y_T + \vlow{w}{T} \right), 
    h_T\left(y_T + \vup{w}{T} \right) \right\}  \label{eq:apdx:DP:gT_definition}
\end{align}
Note that $g_T$ is the maximum of two convex, coercive functions of $y_T$, 
hence it is also convex and coercive (see Theorem 5.5 in \cite{Rock70} or 
Proposition 1.2.4 in \cite{Bkas03} for details). 
The outer (minimization) problem at time $T$ becomes:
\begin{align}
  J_T^*(x_T) &= \underset{L_T \leq u_T(\cdot) \leq U_T}{\min} ~
  c_T \cdot u_T + g_T(x_T+u_T) \nonumber \\
  &= -c_T \cdot x_T + \underset{L_T \leq u_T(\cdot) \leq U_T}{\min} 
  \left[~ c_T \cdot (x_T+u_T) +  g_T(x_T+u_T) ~\right] \nonumber
\end{align}
For any $x_T$, $c_T \cdot (x_T+u_T) + g_T(x_T + u_T)$ is 
a convex function of its argument $y_T = x_T+u_T$. 
As such, by defining $\minofgt{t}$ to be the compact set of
minimizers of the convex and coercive function $c_T \cdot y + g_T(y)$,
we obtain that the optimal controller and optimal value function at time $T$ will be:
\begin{align}
  u_T^*(x_T) &=
  \begin{cases}
    U_T, & ~\text{if}~  x_T < \vlow{y}{T} - U_T\\
    -x_T + y^*, & ~\text{otherwise} \\
    L_T, &  ~\text{if}~ x_T > \vup{y}{T} - L_T
  \end{cases} \label{eq:apdx:DP:uT_star} \\
  J_T^*(x_T) &= 
  \begin{cases}
    c_T \cdot U_T + g_T(x_T+U_T), & ~\text{if}~  x_T < \vlow{y}{T} - U_T\\
    c_T\cdot(y^*- x_T) + g_T(y^*), & ~\text{otherwise} \\
    c_T \cdot L_T + g_T(x_T+L_T), &  ~\text{if}~ x_T > \vup{y}{T} - L_T
  \end{cases} \label{eq:apdx:DP:JT_star} \\
  \text{where}~ y^* &\in \minofgt{T}. \nonumber
\end{align}
The following properties are immediately obvious:
\begin{enumerate}
\item $u_T^*(x_T)$ is piecewise affine (with at most 3 pieces), continuous, 
  monotonically decreasing in $x_T$. 
\item $J_T^*(x_T)$ is convex, since it represents 
  a partial minimization of a convex function with respect 
  to one of the variables (see Proposition 2.3.6 in \cite{Bkas03}).
\end{enumerate}
The results can be immediately extended by induction on $k$:
\begin{lemma}
  \label{sec:DP:lemma_piecewise_uJ}
  The optimal control policy $u_k^*(x_k)$ is piecewise affine, with at most 3 
  pieces, continuous, and monotonically decreasing in $x_k$. The optimal objective 
  function $J_k^*(x_k)$ is convex in $x_t$.
\end{lemma}
\begin{proof}
  The induction is checked at $k=T$. Assume the property is true at $k+1$. Letting
  $y_k \bydef x_k + u_k$, the Bellman recursion at $k$ becomes:
  \begin{align*}
    J_{k}^*(x_k) &\bydef \underset{L_k \leq u_k \leq U_k}{\min} \left[~
      c_k \cdot u_k + g_k\left(x_k+u_k\right) ~\right] \\ 
    g_k\left(y_k\right) &\bydef \underset{w_k \in \W_k}{\max} \left[~ h_k(y_k+w_k)
      + J_{k+1}^*\left(y_k+w_k\right)  ~\right].  
  \end{align*}

  Consider first the maximization problem. Since $h_k$ is convex, and
  (by the induction hypothesis) $J_{k+1}^*$ is also convex, 
  the maximum will be reached on the boundary of 
  $\W_{k}= \left[ \vlow{w}{k},\vup{w}{k}\right]$:
  \begin{align}
    g_k\left(y_k\right) = \max \left\{~ 
      h_k (y_k + \vlow{w}{k}) + J_{k+1}^*\left(y_k + \vlow{w}{k}\right) , \quad
      h_k(y_k + \vup{w}{k}) + J_{k+1}^*\left(y_k + \vup{w}{k}\right) ~\right\} 
    \label{eq:apdx:DP:gk_value}
  \end{align}
  and $g_k(y_k)$ will be also be convex.
  The minimization problem becomes:
  \begin{align}
    J_{k}^*(x_{k}) &= \underset{L_k \leq u_{k} \leq U_k}{\min} ~ \left[ ~
      c_{k} \cdot u_{k} + g_k\left(x_k+u_k\right) ~\right] \nonumber \\
    &= -c_k \cdot x_k + \underset{L_k \leq u_{k} \leq U_k}{\min} ~ \left[ ~
      c_{k} \cdot (x_k+u_{k}) + g_k\left(x_k+u_k\right) ~\right]
  \end{align}
  Defining, as before, $\minofgt{k}$ as the set of minimizers of 
  $c_k \cdot y + g_k(y)$, we get:
  \begin{align}
    u_k^*(x_k) &=
    \begin{cases}
      U_k, & ~\text{if}~  x_k < \vlow{y}{k} - U_k\\
      -x_k + y^*, & ~\text{otherwise} \\
      L_k, &  ~\text{if}~ x_k > \vup{y}{k} - L_k
    \end{cases} \label{eq:apdx:DP:uk_star} \\
    J_k^*(x_k) &=  
    \begin{cases}
      c_k \cdot U_k + g_k(x_k+U_k), & ~\text{if}~  x_k < \vlow{y}{k} - U_k\\
      c_k\cdot (y^* - x_k) + g_k(y^*), & ~\text{otherwise} \\
      c_k \cdot L_k + g_k(x_k+L_k), &  ~\text{if}~ x_k > \vup{y}{k} - L_k
    \end{cases} \label{eq:apdx:DP:Jk_star} \\
    \text{where}~ y^* &\in \minofgt{k}. \nonumber
  \end{align}
  In particular, $u_k^*$ will be piecewise affine with 3 pieces, continuous, 
  monotonically decreasing, and $J_k^*$ will be convex (as the partial 
  minimization of a convex function with respect to one of the variables).
  A typical example of the optimal control law and the optimal value function 
  is shown in Figure \ref{fig:uk_Jk_piecewise} of Section \ref{sec:DP_formulation}. 
\end{proof}

\subsection{Zonotopes and Zonogons.}
\label{sec:zonotope-properties}
In this section of the Appendix, we would like to outline several useful properties of 
the main geometrical objects of interest in our exposition, namely \emph{zonotopes}.
The presentation here parallels that in Chapter 7 of \cite{Ziegl03}, to which the 
interested reader is referred to for a much more comprehensive treatment.

\emph{Zonotopes} are special polytopes that can be viewed in various ways:
as projections of cubes, as Minkowski sums of line segments, and as sets of 
bounded linear combinations of vector configurations. Each description gives 
a different insight into the combinatorics of zonotopes, and there exist some 
very interesting results that unify the different descriptions under a common 
theory. For our purposes, it will be sufficient to understand zonotopes under the 
first two descriptions. In particular, letting $\H_k$ denote the $k$-dimensional 
hypercube, $\H_k = \{\mb{w} \in \R^k \,:\, 0 \leq w_i \leq 1, \, \forall \, i\}$, 
we can introduce the following definition:
\begin{definition} [\textbf{7.13} in \cite{Ziegl03}]
  \label{def:zonotope}
  A \textbf{zonotope} is the image of a cube under an affine projection,
  that is, a $d$-polytope $Z \subseteq \R^d$ of the form
  \begin{align*}
    Z = Z(V) &:= V \cdot \H_k + \mb{z} = \{V \mb{w} + \mb{z} ~:~ 
    \mb{w} \in \H_k\} \\
    &= \{ \mb{x} \in \R^d ~:~ \mb{x}=\mb{z} + \sum_{i=1}^k w_i \mb{v}_i,
    ~0 \leq y_i \leq 1\}
  \end{align*}
  for some matrix (vector configuration) $V=(\mb{v}_1,\dots,\mb{v}_k) \in 
  \R^{d \times k}$.
\end{definition}
The rows of the matrix $V$ are often referred to as the \emph{generators} 
defining the zonotope. An equivalent description of the zonotope can be 
obtained by recalling that every $k$-cube $\H_k$ is a product of line segments 
$\H_k = \H_1 \times \dots \times \H_1$. Since for a linear operator $\pi$ 
we always have:
$\pi(\H_1 \times \dots \times \H_1) 
  = \pi(\H_1) + \dots + \pi(\H_1)$,
by considering an affine map given by $\pi(\mb{w}) = V \mb{w} + \mb{z}$,
it is easy to see that every zonotope is the Minkowski
sum of a set of line segments:
\begin{align*}
  Z(V) = [0,\mb{v}_1] + \dots + [0,\mb{v}_p] + \mb{z}.
\end{align*}

For completeness, we remark that there is no loss of generality in 
regarding a zonotope as a projection from the unit hypercube $\H_k$, since 
any projection from an arbitrary hyper-rectangle in $\R^k$ can always be seen 
as a projection from the unit hypercube in $\R^k$. To see this, consider an arbitrary 
hyper-rectangle in $\R^k$:
\begin{align*}
  \W_k = [\vlow{w}{1},\vup{w}{1}] \times [\vlow{w}{2},\vup{w}{2}] \times 
  \dots \times [\vlow{w}{k},\vup{w}{k}],
\end{align*}
and note that, with $V \in \R^{d \times k}$, and $\mb{a}^T \in \R^k$ denoting 
the $j$-th row of $V$, the $j$-th component of 
$Z(V) \bydef V \cdot \W_k + \mb{z}$ can be written:
\begin{align*}
  Z(V)_j \bydef  z_j + \sum_{i=1}^k \left( a_i \cdot w_i \right) 
  = \left( z_j + \sum_{i=1}^k a_i \cdot \vlow{w}{i} \right) + 
  \sum_{i=1}^k a_i \cdot \left( \vup{w}{i} - \vlow{w}{i} \right) \cdot y_i,
  ~\text{where}~ y_i \in [0,1], \, \forall \, 1 \leq i \leq k.
\end{align*}

An example of a subclass of zonotopes are the \emph{zonogons}, which 
are all centrally symmetric, 2-dimensional $2p$-gons, arising as the 
projection of $p$-cubes to the plane. An example is shown 
in Figure \ref{fig:zonotope_R6} of Section \ref{sec:induct_hypo}.
These are the main objects of 
interest in our treatment, and the following lemma summarizes 
their most important properties: 
\begin{lemma}
  \label{lem:zonotope_properties}
  Let $\H_k = [0,1]^k$ be a $k$-dimensional hypercube, $k \geq 2$. 
  For fixed $\mb{a},\mb{b} \in \R^k$ and $a_0,b_0 \in \R$,
  consider the affine transformation $\pi : \R^k \rightarrow \R^2,~
  \pi(\mb{w}) = \left[ \begin{array}{c}
      \mb{a}^T \\ \mb{b}^T
    \end{array} \right] \cdot \mb{w} + \left[ \begin{array}{c}
      a_0 \\ b_0 \end{array} \right]$ and the zonogon $\Theta \subset \R^2$:
  \begin{align*}
    \Theta &= \pi\left(\H_k\right)
    \bydef \left\{ 
      \mb{\theta} \in \R^2 : \exists \, \mb{w} \in \H_k ~\text{s.t.}~
      \mb{\theta} = \pi(\mb{w}) \right\}.
  \end{align*}
  If we let $\V_\Theta$ denote the set of vertices of $\Theta$, then 
  the following properties are true:
  \begin{enumerate}
  \item $\exists \, \mb{O} \in \Theta$ such that $\Theta$ is symmetric around 
    $\mb{O} ~:~ \forall \, \mb{x} \in \Theta \Rightarrow 
    2\mb{O}-\mb{x} \in \Theta$.
  \item $\left| \V_\Theta \right| = 2p \leq 2k$ 
    vertices. Also, $p < k$ if and only if $\exists \, i \neq j \in \{1,\dots,k\} 
    ~\text{such that}~ \rank \left(  \left[ \begin{array}{cc}
          a_i & a_j \\ b_i & b_j
        \end{array} \right]
    \right)<2$.
  \item If we number the vertices of $\V_\Theta$
    in cyclic order:
    \begin{align*}
      \V_\Theta = \left( \mb{v}_0,\dots,\mb{v}_i,\mb{v}_{i+1},\dots
        ,\mb{v}_{2p-1} \right)  \quad (\mb{v}_{2p+i} \bydef \mb{v}_{(2p+i)\mod{(2p)} })
    \end{align*}
    then
    $2 \mb{O} - \mb{v}_i=\mb{v}_{i+p}$,    and we have 
    the following representation for $\Theta$ as a Minkowski sum of line segments: 
    \begin{align*}
      \Theta &= \mb{O} + \left[ - \frac{\mb{v}_1-\mb{v}_0}{2}, 
        \frac{\mb{v}_1-\mb{v}_0}{2} \right] + \dots + \left[ 
        - \frac{\mb{v}_{p}-\mb{v}_{p-1}}{2}, \frac{\mb{v}_{p}-\mb{v}_{p-1}}{2} \right] \\
      &\bydef \mb{O} + \sum_{i=1}^{p} \lambda_i \cdot \frac{\mb{v}_{i}-\mb{v}_{i-1}}{2},
      \quad -1 \leq \lambda_i \leq 1.
    \end{align*}
  \item If $\exists \, \mb{w}_1,\mb{w}_2 \in \H_k$ such that 
    $\mb{v}_1\bydef \pi(\mb{w}_1) = \mb{v}_2\bydef \pi(\mb{w}_2)$ 
    and $\mb{v}_{1,2} \in \V_\Theta$, then $\exists \, j \in \{1,\dots,k\}$ 
    such that $a_j=b_j=0$.
  \item With the same numbering from $(iii)$ and $k=p$, for any $i \in \{0,\dots,
    2p-1\}$, the vertices of the hypercube that are projecting to 
    $\mb{v}_i$ and $\mb{v}_{i+1}$, respectively, are adjacent, i.e. 
    they only differ in exactly one component.
  \end{enumerate}
\end{lemma}
\begin{proof}
  We will omit a complete proof of the lemma, and will instead simply 
  suggest the main ideas needed for checking the validity of the statements.
  
  For part $(i)$, it is easy to argue that the center of the hypercube, 
  $\mb{O}_{\H}=[1/2, 1/2, \dots, 1/2]^T$, will always project into the center of 
  the zonogon, i.e. $\mb{O} = \pi\left(\mb{O}_{\H}\right)$. This implies that 
  any zonogon will be centrally symmetric, and will therefore have 
  an even number of vertices. 
  
  Part $(ii)$ can be shown by induction on the dimension $k$ of the hypercube, 
  $\H_k$. For instance, to prove the first claim, note that 
  the projection of a polytope is simply the 
  convex hull of the projections of the vertices, and therefore 
  projecting a hypercube of dimension $k$ simply amounts to 
  projecting two hypercubes of dimension 
  $k-1$, one for $w_k=0$ and another for $w_k=1$, and then taking the 
  convex hull of the two resulting polytopes. It is easy 
  to see that these two polytopes in $\R^2$ are themselves zonogons, 
  and are translated copies of each other (by an amount $[a_k,b_k]^T$).
  Therefore, by the induction hypothesis, they have at most $2(k-1)$ 
  vertices, and taking their convex hull introduces at most two new 
  vertices, for a total of at most $2(k-1)+2=2k$ vertices. The second claim 
  can be proved in a similar fashion.
 
  One way to prove part $(iii)$ is also by induction on $p$, by taking any pair of 
  opposite (i.e. parallel, of the same length) edges and 
  showing that they correspond to a Minkowski summand of the zonogon.
  
  Part $(iv)$ also follows by induction. Using the same argument as 
  for part $(ii)$, note that the only ways to have two distinct vertices 
  of the hypercube $\H_k$ (of dimension $k$) project onto the same vertex of 
  the zonogon $\Theta$ is to either have this situation happen for one of the two $k-1$ 
  dimensional hypercubes (in which case the induction hypothesis would complete
  the proof), or to have zero translation between the two 
  zonogons, which could only happen if $a_k=b_k=0$.
  
  Part $(v)$ follows by using parts $(iii)$ and $(iv)$ and the definition of a zonogon as 
  the Minkowski sum of line segments. In particular, since the difference between 
  two consecutive vertices of the zonogon, $\mb{v}_i, \mb{v}_{i+1}$, for 
  the case $k=p$, is always given 
  by a single column of the projection matrix (i.e. $[a_j,b_j]^T$, for some $j$), 
  then the unique vertices of $\H_k$ that were projecting onto $\mb{v}_i$ and 
  $\mb{v}_{i+1}$, respectively, must be incidence vectors that differ in exactly 
  one component, i.e. are adjacent on the hypercube $\H_k$.
\end{proof}

\subsection{Technical Lemmas.}
\label{sec:technical-lemmas}
This section of the Appendix contains a detailed proof for the 
technical Lemma \ref{lem:rside_delta_gamma} introduced in Section 
\ref{sec:further_underst_induc_hypo},
which we include below, for convenience.

{\sc Lemma \ref{lem:rside_delta_gamma}.}
\emph{
  When the zonogon $\Theta$ has a non-trivial intersection with 
  the band $\B$ (case \textup{\textbf{[C4]}}), 
  the convex polygon $\Delta_\Gamma$ and the set of 
  points on its right side, $\rside(\Delta_\Gamma)$, satisfy the following
  properties:
  \begin{enumerate}
  \item  $\rside(\Delta_\Gamma)$ is the union of two sequences 
    of consecutive vertices (one starting at $\tilde{\mb{v}}_0$, and one ending at 
    $\tilde{\mb{v}}_k$), and possibly an additional vertex, $\tilde{\mb{v}}_t$:
    \begin{equation*}
      \rside(\Delta_\Gamma) = \left\{\tilde{\mb{v}}_0,\tilde{\mb{v}}_1,\dots,
        \tilde{\mb{v}}_{s} \right\} \cup \{\tilde{\mb{v}}_t\} 
      \cup \left\{ \tilde{\mb{v}}_{r},\tilde{\mb{v}}_{r+1}
        \dots,\tilde{\mb{v}}_{k}\right\}
      ,~ \text{for some}~ 
      s \leq r \in \{0,\dots,k\}.
    \end{equation*}
  \item With $\cotan{\cdot}{\cdot}$ given by \eqref{eq:cotan_definition} applied 
    to the $(\tilde{\gamma}_1,\tilde{\gamma}_2)$ coordinates, we have that:
    \begin{align*}
      \begin{cases}
        \cotan{\tilde{\mb{v}}_{s}}{\tilde{\mb{v}}_{\min(t,r)}}
        \geq \frac{a_{s+1}}{b_{s+1}}, & \textup{whenever $t > s$}\\
        \cotan{\tilde{\mb{v}}_{\max(t,s)}}{\tilde{\mb{v}}_{r}} \leq \frac{a_{r}}{b_{r}},
        & \textup{whenever $t < r$}.
      \end{cases}
    \end{align*}
  \end{enumerate}
}
\begin{proof}[Proof of Lemma \ref{lem:rside_delta_gamma}]
  In the following exposition, we will use the same notation as introduced in Section
  \ref{sec:further_underst_induc_hypo}. Recall that case \textbf{[C4]} on which the 
  lemma is focused corresponds to a nontrivial intersection of the zonotope $\Theta$  
  with the horizontal band $\B$ defined in (\ref{eq:BLU_definition}). As suggested 
  in Figure \ref{fig:zono_cases_nontrivial} of Section \ref{sec:further_underst_induc_hypo},
  this case can be separated into three subcases, depending on the 
  position of the vertex $\mb{v}_t$ relative to the band $\B$, where the index 
  $t$ is defined in (\ref{eq:t_index_definition}). Since the proof 
  of all three cases is essentially identical, we will focus on the more ``complicated''
  situation, namely when $\mb{v}_t \in \B$. The 
  corresponding arguments for the other two cases should be straightforward.
  
  First, recall that $\Delta_\Gamma$ is given by \eqref{eq:delta_gamma_convex_hull_defn},
  i.e. $\Delta_\Gamma = \conv \left( \{ \tilde{\mb{v}}_0,\dots,\tilde{\mb{v}}_k\} \right)$,
  where the points $\tilde{\mb{v}}_i$ are given by \eqref{eq:vi_tilde_definition},
  which results from applying mapping 
  \eqref{eq:mapping_for_gamma_tilde_ustar_explicit} to $\mb{v}_i \in \Theta$.
  From Definition \ref{def:right_side} of the \emph{right side}, it can be seen 
  that the points of interest to us, namely $\rside(\Delta_\Gamma)$, will be a maximal 
  subset $\left\{ \tilde{\mb{v}}_{i(1)},\tilde{\mb{v}}_{i(2)},\dots,\tilde{\mb{v}}_{i(m)} \right\} 
  \subseteq \left\{ \tilde{\mb{v}}_0,\dots,\tilde{\mb{v}}_k \right\}$, satisfying:
  \begin{equation}
    \left\{
      \begin{aligned}
        &\tilde{\mb{v}}_{i(1)} = \argmax_{\tilde{\gamma}_1} ~ \argmin_{\tilde{\gamma}_2}
        \left[ ~ \tilde{\mb{\gamma}} ~:~ \tilde{\mb{\gamma}}=
          \left( {\tilde{\gamma}_1}, {\tilde{\gamma}_2} \right) \in 
          \left\{ \tilde{\mb{v}}_0,\dots,\tilde{\mb{v}}_k \right\}
          ~\right] \\
        &\tilde{\mb{v}}_{i(m)} = \argmax_{\tilde{\gamma}_1} ~ \argmax_{\tilde{\gamma}_2}
        \left[ ~ \tilde{\mb{\gamma}} ~:~ \tilde{\mb{\gamma}}=
          \left( {\tilde{\gamma}_1}, {\tilde{\gamma}_2} \right) \in 
          \left\{ \tilde{\mb{v}}_0,\dots,\tilde{\mb{v}}_k \right\}
          ~\right] \\
        &\cotan{\tilde{\mb{v}}_{i(1)}}{\tilde{\mb{v}}_{i(2)}} > 
        \cotan{\tilde{\mb{v}}_{i(2)}}{\tilde{\mb{v}}_{i(3)}} > \dots >
        \cotan{\tilde{\mb{v}}_{i(m-1)}}{\tilde{\mb{v}}_{i(m)}}.
      \end{aligned}
    \right.
    \label{eq:rside_conditions}
  \end{equation}
  
  For the analysis, we will find it useful to define the following 
  two indices:
  \begin{equation}
    \begin{aligned}
      \hat{s} \bydef \min \left\{i \in \{0,\dots,k\} : \theta_2(\mb{v}_i) \geq y^*-U \right\},
      \quad 
      \hat{r} \bydef \max \left\{i \in \{0,\dots,k\} : \theta_2(\mb{v}_i) \leq y^*-L \right\}.
    \end{aligned}
    \label{eq:shat_rhat_index_definition}
  \end{equation}
  In particular, $\hat{s}$ is the index of the first vertex of $\rside(\Theta)$ 
  falling inside $\B$, and $\hat{r}$ is the index of the last vertex of $\rside(\Theta)$ 
  falling inside $\B$. Since we are in 
  the situation when $\mb{v}_t \in \B$, it can be seen that 
  $0 \leq \hat{s} \leq t \leq \hat{r} \leq k$, and thus, from \eqref{eq:t_index_definition}
  (the definition of $t$) and \eqref{eq:conditions_generators_theta} 
  (typical conditions for the right side of a zonogon):
  \begin{align}
    \frac{a_1}{b_1} > \dots > \frac{a_{\hat{s}}}{b_{\hat{s}}} > \dots 
    > \frac{a_t}{b_t} > c \geq \frac{a_{t+1}}{b_{t+1}}
    > \dots > \frac{a_{\hat{r}}}{b_{\hat{r}}} > \dots 
    > \frac{a_k}{b_k}.
    \label{eq:relation_at_ar_as_vt_in_band}
  \end{align}

  With this new notation, we proceed to prove the first result in the claim.
  First, consider all the vertices $\mb{v}_i \in \rside(\Theta)$ falling 
  strictly below the band $\B$, i.e. satisfying $\theta_2[ \mb{v}_i ] < y^* - U$.
  From the definition of $\hat{s}$, \eqref{eq:shat_rhat_index_definition},
  these are exactly $\mb{v}_0,\dots,\mb{v}_{\hat{s}-1}$, and mapping
  \eqref{eq:mapping_for_gamma_tilde_ustar_explicit} applied to them will yield:
  $\tilde{\mb{v}}_i = \left(~ \theta_1[\mb{v}_i] + c \cdot U,~ 
      \theta_2[\mb{v}_i] + U ~\right)$.
  In other words, any such points will simply be translated by $(c \cdot U, U)$.
  Similarly, any points $\mb{v}_i \in \rside(\Theta)$ falling 
  strictly above the band $\B$, i.e. $\theta_2[ \mb{v}_i ] > y^* - L$, 
  will be translated by $(c \cdot L, L)$, so that we have:
  \begin{equation}
    \begin{aligned}
      \tilde{\mb{v}}_i &= \mb{v}_i + (c \cdot U, U), \quad 
      i \in \{0,\dots,\hat{s}-1\}, \\
      \tilde{\mb{v}}_i & = \mb{v}_i + (c \cdot L, L), \quad 
      i \in \{\hat{r}+1,\dots,k\},
    \end{aligned}
    \label{eq:vtilde_for_translations}
  \end{equation}
  which immediately implies, since $\mb{v}_i \in \rside(\Theta)$, that:
  \begin{equation}
    \left\{
    \begin{aligned}
      &\cotan{\tilde{\mb{v}}_0}{\tilde{\mb{v}}_1} >
      \cotan{\tilde{\mb{v}}_1}{\tilde{\mb{v}}_2} > \dots 
      > \cotan{\tilde{\mb{v}}_{\hat{s}-2}}{\tilde{\mb{v}}_{\hat{s}-1}}, \\
      &\cotan{\tilde{\mb{v}}_{\hat{r}+1}}{\tilde{\mb{v}}_{\hat{r}+2}} > 
      \cotan{\tilde{\mb{v}}_{\hat{r}+2}}{\tilde{\mb{v}}_{\hat{r}+3}} > \dots 
      > \cotan{\tilde{\mb{v}}_{k-1}}{\tilde{\mb{v}}_k}.
    \end{aligned}
    \right.
    \label{eq:cotangents_translated_vis} 
  \end{equation}

  For any vertices inside $\B$, i.e. $\mb{v}_i \in \rside(\Theta) \cap \B$,
  mapping \eqref{eq:mapping_for_gamma_tilde_ustar_explicit} will yield:
  \begin{align}
    \tilde{\mb{v}}_i &= \left(~ \theta_1[\mb{v}_i] - 
      c \cdot \theta_2[\mb{v}_i] + c \cdot y^*,~ y^* ~\right), \quad 
    i \in \{\hat{s},\dots,t,\dots,\hat{r}\},
    \label{eq:gamtilda_coordinates_vi_inband}
  \end{align}
  that is, they will be mapped into points with the same $\tilde{\gamma}_2$
  coordinates. Furthermore, using (\ref{eq:simple_notation_theta12}), it can be 
  seen that $\tilde{\mb{v}}_t$ will have the largest $\tilde{\gamma}_1$ 
  coordinate among all such $\tilde{\mb{v}}_i$:
  \begin{align}
    \tilde{\gamma}_1[\tilde{\mb{v}}_t] - \tilde{\gamma}_1[\tilde{\mb{v}}_i] 
    &\bydef \theta_1[\mb{v}_t] - \theta_1[\mb{v}_i] - c 
    \cdot \left( \theta_2[\mb{v}_t] - \theta_2[\mb{v}_i]  \right) \nonumber \\
    &\overset{\tiny \eqref{eq:simple_notation_theta12}}{=} 
    \begin{cases}
      \sum_{j=i+1}^t a_j - c \cdot \sum_{j=i+1}^t b_j 
      \overset{\tiny \eqref{eq:relation_at_ar_as_vt_in_band}}{\geq} 0, &
      ~\text{if $\hat{s} \leq i < t$}\\
      - \sum_{j=t+1}^i a_j + c \cdot \sum_{j=t+1}^i b_j 
      \overset{\tiny \eqref{eq:relation_at_ar_as_vt_in_band}}{\geq} 0, &
      ~\text{if $t < i \leq \hat{r}$}.
    \end{cases}
    \label{eq:relation_vertices_in_band}
  \end{align}
  
  Furthermore, since the mapping \eqref{eq:mapping_for_gamma_tilde_ustar_explicit} 
  yielding $\tilde{\gamma}_2$ is only a function
  of $\theta_2$, and is monotonic non-decreasing (strictly monotonic increasing 
  outside the band $\B$), vertices 
  $\mb{v}_0,\dots,\mb{v}_k \in \rside(\Theta)$ will be mapped 
  into points $\tilde{\mb{v}}_0,\dots,\tilde{\mb{v}}_k \in \tilde{\Gamma}$ 
  with non-decreasing $\tilde{\gamma}_2$ coordinates:
  \begin{align*}
    \tilde{\gamma}_2[ \tilde{\mb{v}}_0 ] < \tilde{\gamma}_2[ \tilde{\mb{v}}_1 ]
    < \dots < \tilde{\gamma}_2[ \tilde{\mb{v}}_{\hat{s}-1} ] <
    y^* = \tilde{\gamma}_2[ \tilde{\mb{v}}_{\hat{s}} ] = \dots = 
    \tilde{\gamma}_2[ \tilde{\mb{v}}_{t} ] = \dots = 
    \tilde{\gamma}_2[ \tilde{\mb{v}}_{\hat{r}} ] < 
    \tilde{\gamma}_2[ \tilde{\mb{v}}_{\hat{r}+1} ] < \dots < 
    \tilde{\gamma}_2[ \tilde{\mb{v}}_{k} ].
  \end{align*}
  
  Therefore, combining this fact with \eqref{eq:cotangents_translated_vis} and 
  \eqref{eq:relation_vertices_in_band}, we can conclude 
  that the points $\tilde{\mb{v}}_i$ satisfying 
  conditions \eqref{eq:rside_conditions} are none other than:
  \begin{align*}
    \rside(\Delta_\Gamma) = \left\{ \tilde{\mb{v}}_0,\tilde{\mb{v}}_1,\dots,\tilde{\mb{v}}_s,
    \tilde{\mb{v}}_t,\tilde{\mb{v}}_r,\tilde{\mb{v}}_{r+1},\tilde{\mb{v}}_k \right\},
  \end{align*}
  where the indices $s$ and $r$ are given as:
  \begin{equation}
    \begin{aligned}
      s &\bydef 
      \begin{cases}
        \max \left\{ i \in \{1,\dots,\hat{s}-1\} ~:~ 
          \cotan{\tilde{\mb{v}}_{i-1}}{\tilde{\mb{v}}_{i}} > 
          \cotan{\tilde{\mb{v}}_{i}}{\tilde{\mb{v}}_{t}} \right\} \\
        0, ~\text{if the above condition is never true},
      \end{cases} \\
      r &\bydef 
      \begin{cases}
        \min \left\{ i \in \{\hat{r}+1,\dots,k-1\} ~:~ 
          \cotan{\tilde{\mb{v}}_{t}}{\tilde{\mb{v}}_{i}} > 
          \cotan{\tilde{\mb{v}}_{i}}{\tilde{\mb{v}}_{i+1}} \right\} \\
        k, ~\text{if the above condition is never true}.
      \end{cases}
    \end{aligned}
    \label{eq:solution_sr_indices}
  \end{equation}
  This completes the proof of part $(i)$ of the Lemma. We remark that, for the 
  cases when $\mb{v}_t$ falls strictly below $\B$ or strictly above $\B$, one can 
  repeat the exact same reasoning, and immediately argue that the same result would hold.

  In order to prove the first claim in part $(ii)$, we first recall that, 
  from \eqref{eq:solution_sr_indices}, if $s<\hat{s}-1$, we must have:
  \begin{align*}
    \cotan{\tilde{\mb{v}}_{s}}{\tilde{\mb{v}}_{s+1}} \leq 
    \cotan{\tilde{\mb{v}}_{s+1}}{\tilde{\mb{v}}_{t}},
  \end{align*}
  since otherwise, we would have taken $s+1$ instead of $s$ in \eqref{eq:solution_sr_indices}.
  But this immediately implies that:
  \begin{align*}
    \cotan{\tilde{\mb{v}}_{s}}{\tilde{\mb{v}}_{s+1}} &\leq 
    \cotan{\tilde{\mb{v}}_{s+1}}{\tilde{\mb{v}}_{t}}  ~
    \overset{\tiny \eqref{eq:cotan_definition}}{\Leftrightarrow} &
    \frac{\tilde{\gamma}_1[\tilde{\mb{v}}_{s+1}]-\tilde{\gamma}_1[\tilde{\mb{v}}_{s}]}
    {\tilde{\gamma}_2[\tilde{\mb{v}}_{s+1}]-\tilde{\gamma}_2[\tilde{\mb{v}}_{s}]} &\leq
    \frac{\tilde{\gamma}_1[\tilde{\mb{v}}_{t}]-\tilde{\gamma}_1[\tilde{\mb{v}}_{s+1}]}
    {\tilde{\gamma}_2[\tilde{\mb{v}}_{t}]-\tilde{\gamma}_1[\tilde{\mb{v}}_{s+1}]} ~\Rightarrow
    \text{(mediant inequality)} \\
    \frac{\tilde{\gamma}_1[\tilde{\mb{v}}_{s+1}]-\tilde{\gamma}_1[\tilde{\mb{v}}_{s}]}
    {\tilde{\gamma}_2[\tilde{\mb{v}}_{s+1}]-\tilde{\gamma}_2[\tilde{\mb{v}}_{s}]} &\leq
    \frac{\tilde{\gamma}_1[\tilde{\mb{v}}_{t}]-\tilde{\gamma}_1[\tilde{\mb{v}}_{s}]}
    {\tilde{\gamma}_2[\tilde{\mb{v}}_{t}]-\tilde{\gamma}_1[\tilde{\mb{v}}_{s}]} \quad
    \overset{\tiny \eqref{eq:vtilde_for_translations}}{\Leftrightarrow}
    & \frac{a_{s+1}}{b_{s+1}} &\leq \cotan{\tilde{\mb{v}}_{s}}{\tilde{\mb{v}}_{t}},
  \end{align*}
  which is exactly the first claim in part $(ii)$. Thus, the only case to discuss 
  is $s=\hat{s}-1$. Since $s \geq 0$, it must be that, in this case, there are 
  vertices $\mb{v}_i \in \rside(\Theta)$ falling strictly below the band $\B$. 
  Therefore, we can introduce the following point in $\Theta$:
  \begin{equation}
    \begin{aligned}
      M &\bydef \argmax_{\theta_1} \left\{\, (\theta_1,\theta_2) \in \Theta 
        \,:\, \theta_2 = y^*-U \,\right\}
    \end{aligned}
    \label{eq:point_M_in_Theta_def}
  \end{equation}
  Referring back to Figure \ref{fig:matching_affine_control} in Section 
  \ref{sec:constr-affine-contr}, it can be seen that
  $M$ represents the point with 
  smallest $\theta_2$ coordinate in $\B \cap \rside(\Theta)$, and
  $M \in [\mb{v}_{\hat{s}-1},\mb{v}_{\hat{s}}]$.
  If we let $(\theta_1[M],\theta_2[M])$ denote the coordinates of $M$, 
  then by applying mapping (\ref{eq:mapping_for_gamma_tilde_ustar_explicit}) to $M$,
  the coordinates of the point $\tilde{M} \in \tilde{\Gamma}$ are:
  \begin{equation}
    \tilde{M} =\left(\, \theta_1[M] + c \cdot U,~ \theta_2[M] + U \,\right)
    = \left(\, \theta_1[M] + c \cdot U,~ y^* \,\right).
    \label{eq:point_Mtilde_def}
  \end{equation}
  Furthermore, a similar argument with \eqref{eq:relation_vertices_in_band} can 
  be invoked to show that $\tilde{\gamma}_1[\tilde{M}] \leq \tilde{\gamma}_1[\tilde{\mb{v}}_t]$.
  With $s=\hat{s}-1$, we then have:
  \begin{align*}
    \cotan{\tilde{\mb{v}}_{s}}{\tilde{\mb{v}}_{t}} &
    ~ \, \overset{\tiny \eqref{eq:cotan_definition}}{=}
    \frac{\tilde{\gamma}_1[\tilde{\mb{v}}_{t}]-\tilde{\gamma}_1[\tilde{\mb{v}}_{\hat{s}-1}]}
    {\tilde{\gamma}_2[\tilde{\mb{v}}_{t}]-\tilde{\gamma}_2[\tilde{\mb{v}}_{\hat{s}-1}]} ~\geq 
    \quad \text{(since $\tilde{\gamma}_2[\tilde{\mb{v}}_{t}]=
      \tilde{\gamma}_2[\tilde{M}]=y^* >
      \tilde{\gamma}_2[\tilde{\mb{v}}_{\hat{s}-1}]$)}\\
    &\quad \geq \frac{\tilde{\gamma}_1[\tilde{M}]-\tilde{\gamma}_1[\tilde{\mb{v}}_{\hat{s}-1}]}
    {\tilde{\gamma}_2[\tilde{M}]-\tilde{\gamma}_2[\tilde{\mb{v}}_{\hat{s}-1}]} \\
    &\overset{\tiny \eqref{eq:vtilde_for_translations}, \eqref{eq:point_Mtilde_def}}{=} 
    \frac{\theta_1[M] - \theta_1[\mb{v}_{\hat{s}-1}]}
    {\theta_2[M] - \theta_2[\mb{v}_{\hat{s}-1}]} ~= 
    \quad \text{(since $M \in [\mb{v}_{\hat{s}-1},\mb{v}_{\hat{s}}]$)} \\
    &\quad \, = \frac{a_{s+1}}{b_{s+1}} ~,
  \end{align*}
  which completes the proof of the first claim in part $(ii)$.

  The proof of the second claim in $(ii)$ proceeds in an analogous fashion, by first examining 
  the trivial case $r > \hat{r}+1$ in \eqref{eq:solution_sr_indices}, and then 
  introducing $N \bydef \argmax_{\theta_1} \left\{\, (\theta_1,\theta_2) \in \Theta 
    \,:\, \theta_2 = y^*-L \,\right\}$ for the case $r = \hat{r}+1$.
\end{proof}

\small
\bibliographystyle{plainnat}
\bibliography{MOR_robust}

\end{document}